\documentclass[reqno, 12pt]{amsart}

\UseRawInputEncoding
\usepackage{amsmath}
\usepackage{amsthm}
\usepackage{amsfonts}
\usepackage{amssymb}
\usepackage{bm}
\usepackage{mathrsfs}
\usepackage{epsfig}
\usepackage{slashed}
\usepackage{mathtools}
\usepackage{enumerate}
\usepackage{fullpage}
\usepackage{enumitem}
\usepackage{xcolor}
\usepackage{upgreek}

\usepackage{hyperref}
\usepackage{cleveref}

\newtheorem{lemma}{Lemma}[section]
\newtheorem{theorem}[lemma]{Theorem}

\newtheorem{proposition}[lemma]{Proposition}
\theoremstyle{remark}
\newtheorem{remark}{Remark}[section]

\theoremstyle{definition}

\newcommand{\mb}{\mathbf}
\newcommand{\mc}{\mathcal}

\newcommand{\N}{\mathbb{N}}
\newcommand{\R}{\mathbb{R}}
\newcommand{\C}{\mathbb{C}}

\newcommand{\Sp}{\mathbb{S}^8}
\newcommand{\B}{\mathbb{B}^9}
\newcommand{\Hb}{\overline{\mathbb{H}}}

\newcommand{\eps}{\varepsilon}

\newcommand{\D}{\Delta}
\newcommand{\ds}{{\delta^*}}

\newcommand{\la}{\lambda}
\newcommand{\ve}{\varepsilon}

\newcommand{\ra}{\rightarrow}

\newcommand{\DD}{\mathcal{D}}
\newcommand{\pt}{\partial}

\DeclareMathOperator \ran{ran}

\numberwithin{equation}{section}

\newcommand{\nr}[1]{\left\Vert #1\right\Vert}         % norme
\renewcommand{\Re}{\mathop{\rm{Re}}\nolimits}        % partie réelle
\renewcommand{\a}{\alpha}\renewcommand{\b}{\beta}\newcommand{\g}{\gamma}\renewcommand{\d}{\delta}
 \renewcommand{\l}{\lambda}

\newcommand{\s}{\sigma}
\renewcommand{\t}{\tau}\renewcommand{\o}{\omega}\renewcommand{\O}{\Omega}

\usepackage{mathrsfs}
\thanks{Irfan Glogi\'c is supported by the Austrian Science Fund FWF, Projects  P 34378 and P 30076. Funded by the Deutsche Forschungsgemeinschaft (DFG, German Research Foundation) - Project-ID 258734477 - SFB 1173}

\begin{document}

\title[Self-similar blowup for the quadratic wave equation]
{On blowup for the supercritical quadratic wave equation}

\author{Elek Csobo}
\address{Karlsruher Institut f\"ur Technologie, Fakult\"at f\"ur Mathematik,  Englerstra\ss e 2, 76131 Karlsruhe, Germany}

\address{Goethe-Universit\"at Frankfurt, Institut f\"ur Mathematik, Robert-Mayer-Stra\ss e 10, 60629 Frankfurt am Main, Germany}

\address{Universit\"at Innsbruck, Fakult\"at f\"ur Mathematik, Informatik und Physik, Institut f\"ur Mathematik, Technikerstra\ss e 13, 6020 Innsbruck, Austria}
\email{Elek.Csobo@uibk.ac.at}

\author{Irfan Glogi\'c}
\address{Universitat Wien, Fakult\"at f\"ur Mathematik, Oskar-Morgenstern-Platz 1, 1090 Vienna, Austria}
\email{irfan.glogic@univie.ac.at}

\author{Birgit Sch\"orkhuber}
\address{Universit\"at Innsbruck, Fakult\"at f\"ur Mathematik, Informatik und Physik, Institut f\"ur Mathematik, Technikerstra\ss e 13, 6020 Innsbruck, Austria}
\email{Birgit.Schoerkhuber@uibk.ac.at}


\begin{abstract} 
We study singularity formation for the quadratic wave equation in the energy supercritical case, i.e., for $d \geq 7$. We find in closed form a new, non-trivial, radial, self-similar blowup solution $u^*$ which exists for all $d \geq 7$. For $d=9$, we study the stability of $u^*$ without any symmetry assumptions on the initial data and show that there is a family of perturbations which lead to blowup via $u^*$. In similarity coordinates,  this family represents a co-dimension one Lipschitz manifold modulo translation symmetries. The stability analysis relies on delicate spectral analysis for a non-selfadjoint operator. In addition, in $d=7$ and $d=9$, we prove non-radial stability of the well-known ODE blowup solution. Also, for the first time we establish persistence of regularity for the wave equation in similarity coordinates. 

\end{abstract}

\maketitle

\newcommand{\BB}{\mathscr B}
\newcommand{\EE}{\mathscr E}\newcommand{\EEg}{\mathscr E}
\newcommand{\HH}{\mathcal H}\newcommand{\HHg}{{\mathscr H_\g}} \newcommand{\tHH}{{\tilde {\mathscr H}}}
\newcommand{\dHH}{{\mathbullet {\mathscr H}}} \newcommand{\dHHg}{{\mathring {\mathscr H}_\g}}
\newcommand{\LL}{\mathscr L}
\newcommand{\XX}{\mathscr X}
\newcommand{\Aga}{\Ac_{\g,\a}}\newcommand{\Aoo}{\Ac_{0,0}}

\section{Introduction}

In this paper, we are concerned with the quadratic wave equation 
\begin{equation}\label{NLW:p}
	(\partial^2_t - \Delta_x) u(t,x)=u(t,x)^2,
\end{equation}
where $(t,x) \in I \times\R^d$, for some interval $I \subset \R$ containing zero. 

It is well-known that in all space dimensions Eq.~\eqref{NLW:p} admits solutions that blow up in finite time, starting from smooth and compactly supported initial data. This follows from a classical result by Levine \cite{Lev74}, which provides an open set of such initial data. However, Levine's argument is indirect, and therefore does not give insight into the profile of blowup.
A more concrete example can be produced by using the well-known \emph{ODE solution},
\begin{equation}\label{Def:ODEblowup}
	u_T^{ODE}(t,x) := \frac{6}{(T-t)^2}, \quad T>0.
\end{equation} 
By truncating the initial data $(u^{ODE}_T(0,\cdot),\partial_t u^{ODE}_T(0,\cdot))$ outside a ball of radius larger than $T$, and using finite speed of propagation, one constructs smooth and compactly supported initial data that lead to blow up at $t=T$. 
What is more, invariance of Eq.~\eqref{NLW:p} under the rescaling
\begin{equation}\label{Eq:QuadraticScaling}
	u(t,x) \mapsto u_\la(t,x):=\la^{-2} u(t/\la,x/\la), \quad \la >0,
\end{equation} 
allows one to look for \emph{self-similar} blowup solutions of the following form
\begin{equation*}
	u(t,x)=\frac{1}{(T-t)^2} \, \phi \left( \frac{x}{T-t}\right).
\end{equation*}
Note that \eqref{Def:ODEblowup} is a self-similar solution with trivial profile $\phi \equiv 6$. We note that the rescaling \eqref{Eq:QuadraticScaling} leaves invariant the energy norm $\dot H^1(\R^d) \times L^2(\R^d)$ of $(u(t,\cdot),\partial_t u(t,\cdot))$ precisely when $d=6$, in which case Eq.~\eqref{NLW:p} is referred to as \emph{energy critical}. In this case, it can be easily shown that in addition to \eqref{Def:ODEblowup} no other radial and smooth self-similar solutions to Eq.~\eqref{NLW:p} exist, see \cite{KavWei90}. However, in the \emph{energy supercritical} case, i.e., for $d \geq 7$, numerics \cite{Kyc11} indicate that in addition to \eqref{Def:ODEblowup} there are non-trivial, radial, globally defined, smooth, and decaying similarity profiles. In fact, for $d=7$ there are infinitely many of them, all of which are positive, as recently proven by Dai and Duyckaerts \cite{dai2020selfsimilar}. A similar result is expected to hold for all $7 \leq d \leq 15$, see \cite{Kyc11}. 

From the point of view of the Cauchy problem for Eq.~\eqref{NLW:p}, the relevant similarity profiles appear to be the trivial one \eqref{Def:ODEblowup} and its first non-trivial ``excitation". Namely, numerical work on supercritical power nonlinearity wave equations in the radial case \cite{BizChmTab04,GloMalSch20} yields evidence that generic blowup is described by the ODE profile, while the threshold separating generic blowup from global existence is given by the stable manifold of the first excited profile, see also \cite{Biz00}. The first step in showing such genericity results would be to establish stability of the ODE profile, and show that its first excitation is co-dimension one stable (which indicates that the stable manifold splits the phase space locally into two connected components). The only result so far for Eq.~\eqref{NLW:p} in this direction is by Donninger and the third author \cite{DonningerSchoerkhuber2017}, who proved radial stability of $u_T$ for all odd $d \geq 7$. 
%What makes study of the excited states more difficult is the fact that they are typically not known in closed form. 
In this paper, we exhibit in closed form what appears to be the first excitation of \eqref{Def:ODEblowup} for every $d \geq 7$. Namely, we have the following self-similar solution to Eq.~\eqref{NLW:p}  
\begin{equation}\label{Eq:Blowup_sol_radial}
u^*(t,x):=\frac{1}{t^2} U\left(\frac{|x|}{t} \right),
\end{equation}
where 
\begin{equation}\label{Eq:Blowup_profile_radial}
U(\rho)=\frac{c_1 - c_2 \rho^2}{(c_3 + \rho^2)^2},
\end{equation}
with
\begin{align*}
c_1=\frac{4}{25}((3d-8)d_0+8d^2-56d+48), \quad c_2=\frac{4}{5} d_0,  \quad c_3=\frac{1}{15}(3d-18+d_0),
\end{align*}
and $d_0 =\sqrt{6(d-1)(d-6)}$. We note that $c_3>0$ when $d \geq 7$, and thus $U \in C^{\infty}[0,\infty)$. To the best of our knowledge, this solution has not been known before, and with the intent of studying threshold behavior, the main object of this paper is to show a variant of co-dimension one stability of $u^*$. \\

Note that $U$ has precisely one zero at $\rho^* =\rho^*(d) >  2$. 
In particular,  this profile is not positive and therefore not a member of the family of self-similar profiles constructed in \cite{dai2020selfsimilar}. However, it is strictly positive inside the backward light cone of the blowup point $(0,0)$. Hence, in this local sense $u^*$ provides a solution to the more frequently studied focusing equation 
\begin{equation}\label{NLW:Abs_p}
(\partial^2_t - \Delta_x) u(t,x)=|u(t,x)|u(t,x).
\end{equation}

What is more, as an outcome of our stability analysis we get that small perturbations of both the ODE profile and $u^*$ stay positive under the evolution of Eq.~\eqref{NLW:p}  and therefore yield solutions to Eq.~\eqref{NLW:Abs_p} as well.

\newpage

\subsection{Main results}

\subsubsection{Preliminaries}
By action of symmetries, the solution \eqref{Eq:Blowup_sol_radial} gives rise to a $(2d+1)$-parameter family of (in general non-radial) blowup solutions.
Namely, Eq.~\eqref{NLW:p} is invariant under space-time translations
\[
S_{T,x_0}(t,x) :=(t-T,x-x_0),
\]
for $T >0, x_0 \in \R^d$, time reflections 
\[
R(t,x): =(-t,x),
\]
as well as Lorentz boosts, which we write in terms of hyperbolic rotations,
\[
\Lambda(a) : =\Lambda^d(a^d)\circ \Lambda^{d-1}(a^{d-1})\circ \dots\circ \Lambda^1(a^1),
\]
where $a\in\R^d$ and $\Lambda^j(a^j)$ for $j = 1, \dots d$, are given by
\[
\begin{cases}
t \mapsto t \cosh(a^j)+x^j\sinh(a^j),\\
x^j \mapsto t \sinh(a^j)+x^j\cosh(a^j),\\
x^k \mapsto x^k \quad (k \neq j).
\end{cases}
\]
We then let
\begin{equation}\label{Def:Symmetries}
	\Lambda_{T,x_0}(a):=R \circ \Lambda(a) \circ S_{T,x_0},
\end{equation}
and thereby obtain the following $(2d+1)$-parameter family of solutions to Eq.~\eqref{NLW:p} 
\[
u^*_{T,x_0,a}(t,x) :=u^* \circ \Lambda_{T,x_0}(a)(t,x). 
\]
We note  that for 
\[ (t',x') :=  \Lambda_{T,x_0}(a)(t,x), \]
we have
\begin{equation}\label{Eq:t'x'}
	|x'|^2-t'^2=|x-x_0|^2-(T-t)^2.
\end{equation}
Furthermore, for $\xi,a \in \R^d$, we set\footnote{For simplicity, we use Einstein summation convention throughout the paper.} 
\begin{align}\label{Def:gamma}
 \gamma(\xi,a):=A_0(a)-A_j(a)\xi^j,
\end{align} 
 where 
\begin{align*} 
A_{0}(a) &:=\cosh (a^{1}) \cosh (a^{2}) \cdots   \cosh (a^{d}),\\
A_{1 }(a) &:= \sinh(a^1) \cosh (a^{2}) \dots \cosh (a^d), \\
A_{2 }(a) &:= \sinh(a^2) \cosh (a^{3}) \dots \cosh (a^d), \\
\vdots \\
A_{d}(a) &:= \sinh(a^d).
\end{align*} 
Then, it is easy to check that 
\begin{align}\label{Eq:t'x'_2}
 t'=(T-t) \gamma(\tfrac{x-x_0}{T-t},a), \quad x'^j = (t-T) \partial_{a^j}  \gamma(\tfrac{x-x_0}{T-t},a) B^{j}(a),
 \end{align}
for $j = 1,\dots, d$, where 
\[ B^{j}(a) = \prod_{i=j+1}^{d} \cosh(a^{i})^{-1}. \]
Now, by using relations \eqref{Eq:t'x'} and \eqref{Eq:t'x'_2} we find more explicitly that
\begin{equation}\label{Def:u*_expl}
	u^*_{T,x_0,a}(t,x) = \frac{1}{(T-t)^2}U_a\left(\frac{x-x_0}{T-t}\right),
\end{equation}
with $U_a: \R^d \to \R$ given by
\begin{equation} \label{explicit-form}
U_a(\xi)=\frac{(c_1-c_2)\gamma(\xi,a)^2+c_2(1- |\xi|^2)}{((1+c_3)\gamma(\xi,a)^2+|\xi|^2-1)^2}.
\end{equation}
Note that for $a = 0$, $U_0(\xi) = U(|\xi|)$ with $U$ being the radial profile in \eqref{Eq:Blowup_profile_radial}. Also, since $c_1 > c_2$ for all $d \geq 7$, there exists a  positive constant $c_0 = c_0(d)$ such that 
 \begin{equation}\label{Eq:PosU^*}
 U_a \geq c_0 >0 \text{ on } \mathbb B^d
 \end{equation}
for all $a \in \R^d$, where $\mathbb B^d$ denotes the open unit ball in $\R^d$. In summary, we have that for $a \in \R^d$, $x_0 \in \R^d$ and $T>0$, Eq.~\eqref{NLW:p} admits an explicit solution \eqref{Def:u*_expl}, which starts off smooth, blows up at $x=x_0$ as $t \rightarrow T^-$, and is strictly positive on the backward light cone
 \begin{equation*}%\label{Lightcone_shifted}
\mathcal C_{T,x_0} :=\bigcup_{t \in [0,T)} \{t\} \times \overline{\mathbb  B^{d}_{T-t}(x_0)}
\end{equation*}
of the blowup point $(T,x_0)$, see Section \ref{Sec:notation} for the notation, which makes it a solution inside $\mc C_{T,x_0}$ to Eq.~\eqref{NLW:Abs_p} as well. Furthermore, simply by scaling
we have that for $k \in \N_0$,
\begin{equation}\label{Eq:Profile_Sobolevnorms}
\left \|  U_a\left(\frac{\cdot-x_0}{T-t}\right) \right \|_{\dot H^k(\mathbb B^d_{T-t}(x_0))} \simeq (T-t)^{\frac{d}{2}-k},
\end{equation}
and hence 
\begin{equation*}
\| u^*_{T,x_0,a}(t,\cdot) \|_{\dot H^k(\mathbb B^d_{T-t}(x_0))}  \simeq (T-t)^{\frac{d}{2} - 2 -k},
\end{equation*}
which implies that the solution blows up in local homogeneous Sobolev semi-norms of order $k > s_c = \frac{d}{2} - 2$. Here, $s_c$ denotes the critical regularity, i.e., $\dot H^{s_c}(\R^d)$ is left invariant under the rescaling \eqref{Eq:QuadraticScaling}. 

%Obviously, the same is true for any self-similar solution with locally smooth blowup profile.

\subsubsection{Conditional stability of blowup via $u^*$}
The main goal of this paper is to investigate stability of blowup governed by $u^*$. For $T=1$, $x_0=0$ and $a=0$, the blowup initial data are given by 
\[ u^*_{1,0,0}(0,x) = U(|x|), \quad \partial_t u^*_{1,0,0}(0,x) = 2 U(|x|) + |x| U'(|x|).  \]
With these preparations at hand, we can formulate the following stability result, where we restrict ourselves to the case $d=9$. 
\begin{theorem} \label{main}
	Let $d=9$. Define functions $h_j: \R^9 \to \R$, $j=1,2$ by 
	\begin{align}\label{Def:CorrFunc}
		h_1(x)=\frac{1}{(7+5|x|^2)^3}, \quad h_2(x)=\frac{35-5|x|^2}{(7+5|x|^2)^4}.
	\end{align}
	There exist constants $M>0$, $\delta>0$, and $\o>0$, such that for all real valued $(f,g)\in C^\infty(\overline{\mathbb  B^{9}_2})\times C^\infty(\overline{\mathbb  B^{9}_2})$ satisfying
	\begin{align*}
		\nr{(f,g)}_{H^{6}(\B_2)\times H^{5}(\B_2)}\leq \frac{\delta}{M}
	\end{align*}
	the following holds: There are parameters $a\in \overline{\mathbb  B^{9}_{M\delta/\o}}$, $x_0\in \overline{\mathbb  B^{9}_\delta}$, $T\in [1-\delta,1+\delta]$, and $\a\in [-\delta,\delta]$, depending Lipschitz continuously on $(f,g)$, such that for initial data 
	\begin{equation}\label{Eq:Initial_Data}
		u(0,\cdot) = U(|\cdot|)  + f + \alpha h_1 , \quad \partial_t u(0,\cdot)   =2 U(|\cdot|) + |\cdot| U'(|\cdot|)   + g + \alpha h_2
	\end{equation}
	there exists a unique solution $u \in C^{\infty}(\mathcal C_{T,x_0})$  to Eq.~\eqref{NLW:p}. Furthermore, this solution blows up at $(T,x_0)$ and can be written as 
	\[ u(t,x) = \frac{1}{(T-t)^2} \left 
	[ U_a \left (\frac{x-x_0}{T-t}\right ) +\varphi(t,x) \right  ] \]
	where 
	$ \| \varphi(t,\cdot) \|_{L^{\infty}(\mathbb  B^{9}_{T-t}(x_0))} \lesssim (T-t)^{\omega}$ and 
	\begin{align*}
		(T-t)^{k-\frac{9}{2}}\nr{\varphi(t,\cdot) }_{\dot{H}^k\left( \mathbb  B^{9}_{T-t}(x_0)\right)}\lesssim   (T-t)^\o,
	\end{align*}
	for $k=0,\dots, 5$. In particular, 
	\begin{align}\label{Conv_Sol}
		\begin{split}
			(T-t)^{k-\frac{5}{2}}\nr{u(t,\cdot) - u^*_{T,x_0,a}(t,\cdot) }_{\dot H ^k( \mathbb  B^{9}_{T-t}(x_0))}\lesssim   (T-t)^\o, \\
			(T-t)^{k-\frac{5}{2}}\nr{\partial_t u(t,\cdot) - \partial_t u^*_{T,x_0,a}(t,\cdot) }_{\dot H ^{k-1}( \mathbb  B^{9}_{T-t}(x_0))}\lesssim    (T-t)^\o,
		\end{split}
	\end{align} 
	for $k = 1,\dots, 5$. Moreover, $u$ is strictly positive on $\mc C_{T,x_0}$ and hence the statement above applies to Eq.~\eqref{NLW:Abs_p} as well.
\end{theorem}

We note that the normalizing factor on the left hand side of Eq.~\eqref{Conv_Sol} appears naturally and  corresponds to the behavior of  the blowup solution we perturbed around, see \eqref{Eq:Profile_Sobolevnorms}.

Some further remarks on the result are in order. 
\begin{remark}
The proof of Theorem \ref{main} relies on stability analysis in  similarity coordinates, in which the above set of perturbations has a co-dimension one interpretation. More precisely, we construct a Lipschitz manifold which is of co-dimension $11$, where $10$ co-dimensions are related to instabilities caused by translation symmetries of the equation and the remaining one is characterized by $(h_1,h_2)$. This is elaborated on in Sec.~\ref{Sec:SimCoord}, see in particular Propositions \ref{prop.instability-similarity} and \ref{Prop:Data_Manifold}. We believe that this manifold gives rise to a proper co-dimension one manifold in a suitable physical data space. However, by the local nature of our approach and the presence of translation symmetries, this is not entirely clear.
\end{remark}

\begin{remark}
\textit{Regularity of the initial data.}
It is only the transformation from similarity coordinates to physical coordinates that induces the higher regularity assumption on the data, from which we can easily deduce the Lipschitz dependence on the blowup parameters. We nonetheless believe that this can be optimized by a more refined analysis.
\end{remark}

\begin{remark}
\textit{Persistence of regularity.}
While  persistence of regularity is standard for the wave equation in physical coordinates, it has not yet been considered for the local problem in similarity coordinates. In fact, all of the related works so far, such as  \cite{DonningerSchoerkhuber2016}, \cite{ChaDon19}, \cite{GlogicSchoerkhuber2021}, are based on a notion of strong solution in similarity coordinates. In this paper, we close this gap and rigorously prove regularity of solutions for smooth initial data. Our proof relies on estimates for the free wave evolution in similarity coordinates in arbitrarily high Sobolev spaces, see Proposition \ref{time-evolution} below.  
\end{remark}

\begin{remark}
\textit{Generalization to other space dimensions.}
Large parts of the proof of Theorem \ref{main} can be generalized to other odd space dimensions. However, the analysis of the underlying spectral problem is quite delicate and only for $d=9$ we are able to solve it rigorously. Nevertheless, from numerical computations, we have strong evidence that the situation is analogous in other space dimensions in the sense that the linearization has exactly one \textit{genuine} unstable eigenvalue. 
\end{remark}

\subsubsection{Stable ODE blowup without symmetry}
For both Eqns.~\eqref{NLW:p} and \eqref{NLW:Abs_p}, stability of the ODE blowup solution under small radial perturbations has been proven by Donninger and the third author \cite{DonningerSchoerkhuber2017}  in all odd space dimensions $d \geq 7$. By exploiting the framework of the proof of Theorem \ref{main}, we generalize the result from \cite{DonningerSchoerkhuber2017} to non-radial perturbations in dimensions $d=7$ and $d = 9$. \\

Before we state the result, we apply the symmetry transformations \eqref{Def:Symmetries} to the ODE profile \eqref{Def:ODEblowup} to obtain the following family of blowup solutions to both Eqns.~\eqref{NLW:p} and \eqref{NLW:Abs_p},
\begin{equation}\label{Eq:ODE_Blowup_nonrad}
u^{\mathrm{ODE}}_{T,x_0,a}(t,x) := \frac{1}{(T-t)^2}\kappa_a \left(\frac{x-x_0}{T-t}\right),
\end{equation}
where
\begin{equation}\label{Eq:kappa_a_explicit}
\kappa_a(\xi) = 6 \gamma(\xi,a)^{-2}.
\end{equation} 
To shorten the notation, we write $\mc C_{T} := \mc C_{T,0}$ for the backward light cone with vertex $(T,0)$.

\begin{theorem} \label{main_ODE}
Let $d \in \{7,9\}$. There are constants $C>0$, $\d>0$, and $\omega >0$, such that for any real valued $(f,g)\in C^\infty(\overline{\mathbb B^d_2})\times C^\infty(\overline{\mathbb B^d_2})$, satisfying
\begin{align}\label{Eq:SmallData}
\nr{(f,g)}_{H^{\frac{d+3}{2}}(\mathbb B^d _2)\times H^{\frac{d+1}{2}}(\mathbb B^d _2)}\leq \frac{\d}{C}
\end{align}
the following holds. There exist parameters $a\in \overline{\mathbb B^d_{C\d/\o}}$ and $T\in [1-\d,1+\d]$ depending Lipschitz continuously on $(f,g)$, such that for initial data 
\begin{equation*}
u(0,\cdot) = 6 + f \quad \partial_t u(0, \cdot)   = 12 + g
\end{equation*}
there exists a unique solution $u \in C^{\infty}(\mathcal C_{T})$  to Eq.~\eqref{NLW:p}. This solution blows up at $(T,0)$ and can be written as 
\[ u(t,x) = \frac{1}{(T-t)^2} \left [ 
\kappa_a \left (\frac{x}{T-t}\right ) +\varphi(t,x) \right  ] \]
where $\varphi$ satisfies
$ \| \varphi(t,\cdot) \|_{L^{\infty}(\mathbb B^{d}_{T-t})} \lesssim  (T-t)^{\omega}$ and 
\begin{align*}
(T-t)^{k-\frac{d}{2} }\nr{\varphi(t,\cdot) }_{\dot{H}^k\left( \mathbb B^{d}_{T-t}\right)}\lesssim   (T-t)^\o,
\end{align*}
for $k=0,\dots, \frac{d+1}{2}$. In particular, 
\begin{align}\label{Conv_Sol_ODE}
\begin{split}
(T-t)^{k-\frac{d}{2}+2}\nr{u(t,\cdot) - u^{ODE}_{T,0,a}(t,\cdot) }_{\dot H ^k( \mathbb B^{d}_{T-t})}\lesssim   (T-t)^\o, \\
(T-t)^{k-\frac{d}{2}+2}\nr{\partial_t u(t,\cdot) - \partial_t u^{ODE}_{T,0,a}(t,\cdot) }_{\dot H ^{k-1}( \mathbb B^{d}_{T-t})}\lesssim   (T-t)^\o,
\end{split}
\end{align} 
for $k = 1,\dots, \frac{d+1}{2}$.
Furthermore, $u$ is strictly positive and the statement above therefore applies to Eq.~\eqref{NLW:Abs_p} as well.
\end{theorem}

We note that due to the invariance of $u^{\mathrm{ODE}}_{1,0,0}$ under spatial translations the blowup location $x_0 = 0$ does not change under small perturbations.

\begin{remark}\label{Rem:ODE}
Stability of the ODE blowup solution for energy supercritical wave equations outside radial symmetry has been established first in $d=3$ by Donninger and the third author \cite{DonningerSchoerkhuber2016}. For the cubic wave equation, the corresponding result was obtained by Chatzikaleas and Donninger \cite{ChaDon19} in $d=5,7$. Compared to these works, one important improvement in Theorem \ref{main_ODE} is the regularity of the solution which allows for the classical interpretation. Furthermore, we prove Lipschitz dependence of the blowup time and the blowup point on the initial data. Finally, from a technical perspective, the adapted inner product defined in Sec.~\ref{Sec:Free} is simpler than the corresponding expressions in \cite{ChaDon19} and can easily be generalized. 
\end{remark}

\subsection{Related results}
Wave equations with focusing power nonlinearities provide the sim-plest possible models for the study of nonlinear wave dynamics and have been investigated intensively in the past decades. Consequently, local well-posedness and the behavior of solutions for small initial data are by now well understood, see e.g. \cite{LindbladSogge95}. Concerning global dynamics for large initial data, substantial progress has been made more recently for energy critical problems. This includes fundamental works on the  characterization of the threshold between finite-time blowup and dispersion in terms of the well-known stationary ground state solution, cf.~\cite{KenigMerle2008}, \cite{KriegerSchlagNakanishi2015} and the references therein. 

In contrast, large data results for energy supercritical equations are rare.  For various models, the ODE blowup is known to provide a stable blowup mechansim and Theorem \ref{main_ODE} further extends these results, see  Remark \ref{Rem:ODE}. In \cite{BizonMaisonWasserman2007}, non-trivial self-similar solutions are constructed for odd supercritical nonlinearities in dimension three,  and \cite{dai2020selfsimilar} provides a generalization to $d \geq 4$. Also, in the three dimensional case, large global solutions were obtained for a supercritical nonlinearity in \cite{KriegerSchlag2017}.
Finally, for $d \geq 11$ and large enough nonlinearities, manifolds of co-dimension greater or equal than two have been constructed in  \cite{Collot2018} that lead to non-selfsimilar blowup in finite time. 

In the description of threshold dynamics for energy supercritical wave equations, self-similar solutions appear to play the key role. This has been observed numerically for power-type nonlinearities \cite{BizChmTab04,GloMalSch20}, but also for more physically relevant models such as wave maps \cite{BizonChmajTabor2000,BiernatBizonMaliborski2016} or the Yang-Mills equation in equivariant symmetry \cite{BizonTabor2001, Bizon2002}. We note that the latter reduces essentially to a radial quadratic wave equation in $d \geq 7$, hence Eq.~\eqref{NLW:p} provides a toy model.
From an analytic point of view, threshold phenomena for energy supercritical wave equations are entirely unexplored. Moreover, results analogous to the energy critical case seem completely out of reach.
 
However, very recently, the first \textit{explicit} candidate for a self-similar threshold solution has been found by the second and third author in \cite{GlogicSchoerkhuber2021} for the focusing cubic wave equation in all supercritical space dimensions $d \geq 5$. In $d=7$, by the conformal symmetry of the linearized equation, the genuine unstable direction could be given in closed form,  see also \cite{GloMalSch20}, which allowed for a rigorous stability analysis. Interestingly, the same effect occurs for the quadratic wave equation and the new self-similar solution \eqref{Eq:Blowup_sol_radial} in $d=9$, which explains the specific choice of the space dimension in Theorem \ref{main}. In view of our results, we conjecture that the self-similar profile $U$ given in \eqref{Eq:Blowup_profile_radial} plays an important role in the threshold dynamics for Eqns.~\eqref{NLW:p} and ~\eqref{NLW:Abs_p}.\\

In the proofs of Theorem \ref{main} and Theorem  \ref{main_ODE} we build on methods developed in  earlier works, in particular  \cite{GlogicSchoerkhuber2021} and
 \cite{DonningerSchoerkhuber2016}.
However, several aspects, in particular the spectral analysis, are specific to the problem and rather delicate. Furthermore, we add important generalizations such as the preservation of regularity, which improves the statements of \cite{GlogicSchoerkhuber2021} or \cite{DonningerSchoerkhuber2016}. The presentation of our results is completely self-contained and all necessary details are provided in the proofs.

\subsection{Notation}\label{Sec:notation}

Throughout the whole paper the Einstein summation convention is in force, i.e., we sum over repeated upper and lower indices, where latin indices run from $1$ to $d$.
We write $\N$ for the natural numbers $\{1,2,3, \dots\}$, $\N_0 := \{0\} \cup \N$. Furthermore, $\R^+ := \{x \in \R: x >0\}$. Also, $\Hb$ stands for the closed complex right half-plane.
By $\mathbb B_R^d(x_0)$ we denote the open ball of radius $R >0$ in $\mathbb R^d$ centered at $x_0 \in \mathbb R^d$. The unit ball is abbreviated by $\mathbb B^d := \mathbb B_1^d(0)$
and $\mathbb S^{d-1} := \partial  \mathbb B^d$.  
\smallskip
The notation $a\lesssim b$ means $a\leq Cb$ for an absolute constant $C>0$ and we write $a\simeq b$ if $a\lesssim b$ and $b \lesssim a$.  	
If $a \leq C_{\varepsilon} b$ for a constant $C_{\varepsilon}>0$ depending on some parameter $\varepsilon$, we write $a \lesssim_{\varepsilon} b$. 

By $L^2(\mathbb B_R^d(x_0))$ and $H^{k}(\mathbb B_R^d(x_0))$, $k \in \N_0$, we denote the  Lebesgue  
and Sobolev spaces obtained from the completion of $C^{\infty}(\mathbb B_R^d(x_0))$ with respect to the usual norm
\[ \|  u \|^2_{H^k(\mathbb B_R^d(x_0))} := \sum_{ |\alpha| \leq k} \|\partial^{\alpha}  u \|^2_{L^2(\mathbb B_R^d(x_0))},\]
with $\alpha \in \N_0^d$ denoting a multi-index and $\partial^{\alpha} u = \partial_1^{\alpha_1} \dots \partial_d^{\alpha_d} u$, where $\partial_i u(x) = \partial_{x_j} u(x)$. 
For vector-valued functions, we use boldface letters, e.g., $\mb f = (f_1,f_2)$ and we sometime write $[\mb f]_1 := f_1$ to extract a single component. 
Throughout the paper,
$W(f,g)$ denotes the Wronskian of two functions $f,g \in C^{1}(I)$, $I \subset \R$, where we use the convention
$W(f,g)=fg'-f'g$, with $f'$ denoting the first derivative.
\smallskip
On a Hilbert space $\mc H$ we denote by $\mc B(\mc H)$ the set of bounded linear operators.
For a closed linear operator $(L, \mc D(L))$ on $\mc H$, we define the resolvent set $\rho(L)$ as the set of all $\la\in\mathbb{C}$ such that $R_{L}(\lambda):=(\lambda- L)^{-1}$ exists as a bounded operator on the whole underlying space.
Furthermore, the spectrum of $L$ is defined as $\sigma(L):=\mathbb{C}\setminus \rho(L)$ and the point spectrum is denoted by $\sigma_p(L) \subset  \sigma(L)$. 

\subsubsection*{Spherical harmonics}
Fix a dimension $d \geq 3$. For $\ell \in \mathbb{N}_0$, an eigenfunction for the Laplace-Beltrami operator  on $\mathbb S^{d-1}$ with eigenvalue $\ell(\ell + d -2)$ is called a spherical harmonic function of degree $\ell$. For each $\ell \in \N$, we denote by $M_{d,\ell}$ the number of linearly independent spherical harmonics of degree $\ell$, and for  $\Omega_{\ell} := \{1,\dots, M_{d,\ell}\}$ we designate by $\{Y_{\ell, m} : m \in \Omega_{\ell} \}$ a set of orthonormal spherical harmonics, i.e.,
\[ \int_{\mathbb S^{d-1}} Y_{\ell,m}(\omega) \overline{Y_{\ell,m'}(\omega)}d \sigma(\omega)  = \delta_{m m'}. \]
Obviously, one has $\Omega_{0} = \{1 \}$, $\Omega_{1} = \{1, \dots, d\}$,
and we can take $Y_{0,1}(\omega) = c_1$, $Y_{1, m}(\omega) = \tilde  c_{m} \omega_m$ 
for suitable normalization constants $c_1, \tilde c_{m} \in \R$. For $g \in C^{\infty}(\mathbb S^{d-1})$, we define
$\mc P_{\ell}: L^2(\mathbb S^{d-1}) \to L^2(\mathbb S^{d-1})$ by
\begin{align*}
\mc P_{\ell} g(\omega)  := \sum_{m \in \Omega_{\ell}} (g|Y_{\ell, m} )_{L^2(\mathbb S^{d-1})} Y_{\ell, m}(\omega).
\end{align*}
It is well-known, see e.g.~\cite{Atkinson}, that $\mc P_{\ell}$ defines a self-adjoint projection on $L^2(\mathbb S^{d-1})$ and that 
$ \lim_{n \to \infty}   \| g  - \sum_{\ell = 0}^{n} \mc P_{\ell} g \|_{L^2(\mathbb S^{d-1})} = 0 $.
This can be extended to Sobolev spaces, in particular,  $\lim_{n \to \infty} \|g-   \sum_{\ell = 0}^{n} \mc P_{\ell}  g  \|_{H^k(\mathbb S^{d-1})} = 0$
for all $g \in C^{\infty}(\mathbb S^{d-1})$, see e.g.~\cite{DonningerSchoerkhuber2016}, Lemma A.1. Furthermore, given $f \in C^{\infty}(\overline{\mathbb B^d_R})$ by setting
\begin{align}\label{Decomp:Projection}
[ P_{\ell} f](x)  := \sum_{m \in \Omega_{\ell}} (f(|x| \cdot)|Y_{\ell, m} )_{L^2(\mathbb S^{d-1})} Y_{\ell, m} \left (\tfrac{x}{|x|}\right),
\end{align}
we have that (see for example Lemma A.2 in \cite{DonningerSchoerkhuber2016})
\begin{align}\label{Decomp:SpherHarm_Hk}
\lim_{n \to \infty} \|f- \sum_{\ell = 0}^{n} P_{\ell} f  \|_{H^k(\mathbb B^d_R)} = 0.
\end{align}

\section{The stability problem in similarity coordinates}\label{Sec:SimCoord}
In this section we formulate the equation \eqref{NLW:p} in similarity variables. The advantage of the new setting is the fact that self-similar solutions become time-independent and stability of finite time blowup turns into asymptotic stability of static solutions. Then we state the main results in the new coordinate system.\\

Given $T>0$ and $x_0\in \R^d$, we define \emph{similarity coordinates}
\[
\tau:=-\log(T-t)+\log T, \quad \xi:=\frac{x-x_0}{T-t}.
\]
Note that in $(\tau,\xi)$, the backward light cone $\mc C_{T,x_0}$ corresponds to the infinite cylinder 
\begin{align*}
	\mc Z :=  \bigcup_{\tau \geq 0} \{ \tau \} \times \overline{\mathbb  B^{d}}.
\end{align*}

Furthermore, by setting
\[
\psi(\tau,\xi):=T^2e^{-2\tau}u(T-Te^{-\tau},Te^{-\tau}\xi+x_0),
\]
Eq.~\eqref{NLW:p} transforms into
\begin{equation}\label{Eq:sim-coordinates}
	\left( \partial_{\tau} ^2+5\partial_\tau+2\xi\cdot\nabla\partial_\tau+(\xi\cdot\nabla)^2-\Delta+5 \xi\cdot\nabla+6\right)\psi(\tau,\xi)= \psi(\tau,\xi)^2.
\end{equation}
To get a first order formulation we define 
\begin{equation}
	\psi_1(\tau,\xi) :=\psi(\tau,\xi), \quad \psi_2(\tau,\xi) :=\partial_\tau\psi(\tau,\xi)+\xi\cdot\nabla  \psi(\tau,\xi)+ 2 \psi(\tau,\xi),
\end{equation}
and let $\Psi(\tau)=(\psi_1(\tau,\cdot),\psi_2(\tau,\cdot))$, by means of which Eq.~\eqref{Eq:sim-coordinates} can be written as 
\begin{equation}\label{Eq:Abstract_NLW_sim}
	\partial_\tau\Psi(\tau)=\tilde{\mb L} \Psi(\tau) + \mb F(\Psi(\tau)),
\end{equation}
where
\begin{align*}
	\tilde{\mb L}\mb u(\xi)=
	\begin{pmatrix}
		-\xi\cdot\nabla u_1(\xi)-2u_1(\xi)+u_2(\xi)\\
		\Delta u_1(\xi)-\xi\cdot\nabla u_2(\xi)-3u_2(\xi)
	\end{pmatrix}, \quad \mb F( \mb u)=
	\begin{pmatrix}
		0\\
		u_1^2
	\end{pmatrix},
\end{align*}
for $\mb u = (u_1,u_2)$. Note that in the new variables, the solutions $u^*_{T,x_0,a}$ and $u^{ODE}_{T,x_0,a}$ become static. Namely, every $a \in \R^d$ yields smooth, positive, and $\tau$-independent solutions $\mb U_a=(U_{1,a},U_{2,a})$ and $\bm{\kappa}_a=(\kappa_{1,a},\kappa_{2,a})$ of \eqref{Eq:Abstract_NLW_sim} given by  
\[U_{1,a}(\xi)=U_a(\xi), \quad U_{2,a}(\xi)=\xi \cdot \nabla U_a(\xi)+2 U_a(\xi),\]
and 
\[\kappa_{1,a}(\xi)=\kappa_a(\xi), \quad \kappa_{2,a}(\xi)=\xi \cdot \nabla \kappa_a(\xi)+2 \kappa_a(\xi).\]
We study Eq.~\eqref{Eq:Abstract_NLW_sim} for small perturbations of $\mb U_a$, respectively $\bm{\kappa}_a$,  in the Hilbert space
\[\HH:= H^{\frac{d+1}{2}}(\mathbb B^d)\times H^{\frac{d-1}{2}}(\mathbb B^d)  \]
equipped with the standard norm
\begin{align*}
	\| \mb u \|^2 := \| u_1 \|^2_{H^{\frac{d+1}{2}}(\mathbb B^d)}+ \|u_2\|^2_{H^{\frac{d-1}{2}}(\mathbb B^d)}.
\end{align*}
Also, denote $\mc B_{R}:=\{ \mb u \in \mc H \, | \, \|\mb u \| \leq R \}$. \\

In Proposition \ref{time-evolution} below we show that for $d 
\in \{7,9\}$ the operator  $\tilde{\mb L}: C^{\infty}(\overline{\mathbb B^d}) \times C^{\infty}(\overline{\mathbb B^d})  \subset \mc H \to \mc H$, which describes the free wave evolution in similarity coordinates, is closable and its closure, which we denote  by $\mb L: \mc D(\mb L) \subset \mc H  \to \mc H$, generates a strongly-continuous one-parameter semigroup $(\mb S(\tau))_{\tau \geq 0} \subset \mc B(\mc H)$. By using  Sobolev embedding, it is easy to see that the nonlinearity satisfies 
\begin{equation*}
	\| \mb F(\mb u) \| = \| u_1^2 \|_{H^{\frac{d-1}{2}}(\mathbb B^d)} \leq  \| u_1^2 \|_{H^{\frac{d+1}{2}}(\mathbb B^d)} \lesssim \|u_1\|^2_{H^{\frac{d+1}{2}}(\mathbb B^d)} \lesssim  \| \mb u \|^2,
\end{equation*}
for all $\mb u \in \mc H$, hence $\mb F$ is well-defined  on $\mc H$. 

\subsection{Stability of $\mb U_a$}
The key to proving Theorem~~\ref{main} is the following result, which  establishes for $d=9$ conditional orbital asymptotic stability of the family of static solutions $\{ \mb U_a : a \in \R^9 \}$. 

\begin{proposition}\label{prop.instability-similarity}
	Let $d=9$. There are constants $C>0$ and $\omega>0$ such that the following holds. For all sufficiently small $\delta > 0$ there exists a co-dimension eleven Lipschitz manifold $\mathcal{M}=\mathcal{M}_{\d,C} \subset \mc B_{\delta/C}$ with $\mb 0 \in \mc M $ such that for any $\Phi_0\in \mathcal{M}$ there are $\Psi \in C([0,\infty),\HH)$ and $a \in \overline{\mathbb B^9_{\delta/\omega}}$ such that
	\begin{align}\label{Eq:Evol_Psi}
		\begin{split}
			\Psi(\t)&=\mb S(\tau)(\mb U_0+\Phi_0) + \int_0^{\tau} \mb S(\tau -  \sigma)\mb F(\Psi(\sigma)) d \sigma,
		\end{split}
	\end{align}
	and 
	\[ \|\Psi(\t) - \mb U_a \|  \lesssim \delta e^{-\o\t}\]
	for all $\tau \geq 0$.
\end{proposition}

\smallskip
The number of co-dimensions in Proposition \ref{prop.instability-similarity} is related to the number of unstable eigen- values of the linearization around $\mb U_a$, and the dimension of the corresponding eigenspaces, see Sec.~\ref{Sec:Spectral_Analysis_U}. In fact, ten of these instabilities are caused by the translation symmetries of the problem, and can be controlled by choosing appropriately the blowup parameters $(T,x_0)$. There is, therefore, only one genuine unstable direction.
Next, we state a persistence of regularity result for solutions to Eq.~\eqref{Eq:Evol_Psi}.

\begin{proposition}\label{Prop:smoothness}
	If the initial data $\Phi_0$ from Proposition \ref{prop.instability-similarity} is in $C^{\infty}(\overline{\B}) \times C^{\infty}(\overline{\B})$ then the corresponding solution $\Psi$ of \eqref{Eq:Abstract_NLW_sim} belongs to $C^{\infty}(\mc Z) \times  C^{\infty}(\mc Z)$. In particular, $\Psi$ satisfies \eqref{Eq:Abstract_NLW_sim} in the classical sense. 
\end{proposition}
\begin{remark}
	That this proposition is not vacuous, i.e., that there exist $\Phi_0 \in \mc M \cap (C^{\infty}(\overline{\B}) \times C^{\infty}(\overline{\B}))$, follows from Proposition \ref{Prop:Data_Manifold} below.
\end{remark}
The proofs of Propositions \ref{prop.instability-similarity} and \ref{Prop:smoothness} are provided in Sec.~\ref{Sec:Proof_Propsitions}. \\

In order to derive Theorem \ref{main} from the above results we prescribe in physical variables initial data of the following form
\begin{equation}
	u(0,\cdot) = u^*_{1,0,0}(0,\cdot) + f, \quad \partial_t u(0,\cdot)   =\partial_t u^*_{1,0,0}(0,\cdot) + g,
\end{equation} 
for free functions $(f,g)$ defined on a suitably large ball centered at the origin. In similarity variables, this  transforms into initial data $\Psi(0) = \mb U_0 + \Phi_0$ for Eq.~\eqref{Eq:Abstract_NLW_sim}, with 
\begin{align}\label{Data}
	\Phi_0 = \Upsilon((f,g) ,T,x_0),  
\end{align}
where 
\begin{equation}\label{Def:InitialData_Operator}
	\Upsilon ((f,g) ,T,x_0) := \mc R((f,g),T,x_0)+\mc R(\mb U_0,T,x_0)-\mc R(\mb U_0,1,0)
\end{equation}
and 
\[
\mc R((f_1,f_2),T,x_0)=\begin{pmatrix}
	T^2f_1(T \cdot+x_0)\\
	T^3 f_2(T\cdot+x_0)
\end{pmatrix}.
\]

The next statement asserts that for all small $(f,g)$, there is a choice of parameters $x_0$, $T$, and $\alpha$, for which 
$\Upsilon((f + \alpha h_1 ,g + \alpha h_2), T, x_0)$ belongs to the	 manifold  $\mathcal{M}$ from Proposition \ref{prop.instability-similarity}.  

\begin{proposition}\label{Prop:Data_Manifold}
Let $(h_1,h_2)$ be defined as in \eqref{Def:CorrFunc}. There exists $M> 0$ such that for all sufficiently small $\delta > 0$ the following holds. For any $(f,g) \in  H^6(\B_2) \times H^5(\B_2)$ satisfying  
	\[ \|(f,g)\|_{H^6(\B_2) \times H^5(\B_2)} \leq  \tfrac{\delta}{M^2},\]
	there are $x_0 \in \overline{\mathbb B^{9}_{\delta/M}}$, $T \in [1 - \frac{\delta}{M}, 1+\frac{\delta}{M}]$ and $\alpha \in [-\frac{\delta}{M},\frac{\delta}{M}]$, depending Lipschitz continuously on $(f,g)$, such that  
	\[\Upsilon((f + \alpha h_1 ,g + \alpha h_2), T, x_0) \in \mathcal{M}_{\d,C}, \]
	where $\mc M_{\d,C}$  is the manifold from Proposition \ref{prop.instability-similarity}.
\end{proposition}

Theorem \ref{main} is then obtained by transforming the results of Propositions \ref{prop.instability-similarity}, \ref{Prop:smoothness}, and \ref{Prop:Data_Manifold} back to coordinates $(t,x)$.

\begin{remark} 
	We note that when proving stability of the ODE blowup solution for $d \in \{7,9\}$, similar results are obtained. In fact, the proof implies the existence of a Lipschitz manifold $\mc N$ of co-dimension $d+1$ in the Hilbert space $\mc H$, according to $d+1$ directions of instability induced by  translation invariance. 
	A result similar to Proposition \ref{Prop:Data_Manifold} guarantees that for 
	\textit{any} small enough data $(f,g)$ one can suitably adjust the blowup time $T$ and the blowup point $x_0$ such that $\Upsilon((f  ,g), T, x_0) \in \mc N$, which gives Theorem \ref{main_ODE} on stable blowup. This point of view further justifies using co-dimension one terminology to describe stability of $u^*$. 
\end{remark}

\subsubsection{Time-evolution for small perturbations - Modulation ansatz}\label{Sec:Modulation}

In the following, we assume that $a = a(\tau)$, $a(0)=0$ and $\lim_{\tau \to \infty} a(\tau) = a_{\infty}$. Inserting the ansatz
\begin{align}\label{Def:Psi_modulation_ansatz}
	\Psi(\tau)=\mb U_{a(\tau)}+\Phi(\tau)
\end{align}
into \eqref{Eq:Abstract_NLW_sim} we obtain
\begin{align*}
	\partial_\tau\Phi(\tau)=(\tilde{\mb L}+\mb L'_{a(\tau)})\Phi(\tau)+\mb F (\Phi(\tau))-\partial_\tau \mb U_{a(\tau)},
\end{align*}
with
\begin{align*}
	\mb L'_{a(\tau)}\mb u =\begin{pmatrix}0 \\V_{a(\t)} u_1\end{pmatrix},  \quad V_{a}(\xi)=2 U_{a}(\xi). 
\end{align*}

In the following, we  define
\[
\mb G_{a(\tau)}(\Phi(\t))=[\mb L'_{a(\t)}-\mb L'_{a_\infty}]\Phi(\t)+\mb F(\Phi(\t)).
\] 

and study the evolution equation
\begin{align} \label{rewritten}
	\begin{split}
		\pt_\t\Phi(\t)&=[\tilde{\mb L}+ \mb L'_{a_\infty}]\Phi(\t)+\mb G_{a(\t)}(\Phi(\t))-\pt_\t \mb U_{a(\t)},
	\end{split}
\end{align}
with initial data $\Phi(0) = \mb u \in \mc H$. This naturally splits into three parts: First, in Sec.~\ref{Sec:Free}, we study the time evolution governed by $\tilde{\mb L}$ using semigroup theory. In Sec.~\ref{Sec:Lin}, we analyze the linearized problem, where we consider $\tilde{\mb L}+ \mb L'_{a_\infty}$ as a (compact) perturbation of the free evolution and investigate the underlying spectral problem, restricting to $d=9$. Resolvent bounds allow us to transfer the spectral information to suitable growth estimates for the linearized time evolution. 
The nonlinear problem will be analyzed in integral form in Sec.~\ref{Sec:Nonl_U}, using modulation theory and fixed-point arguments. Also, we  prove Propositions \ref{prop.instability-similarity} - \ref{Prop:Data_Manifold}
and based on this, Theorem \ref{main}. 
In Sec.~\ref{Sec:Stable_ODE_blowup} we give the main arguments to prove Theorem \ref{main_ODE}.
\section{The free wave evolution in similarity variables}\label{Sec:Free}

In this section we prove well-posedness of the linear version of Eq.~\eqref{Eq:Abstract_NLW_sim} in $\mc H$. In other words, we show that the (closure of the) operator $\tilde{\mb L}$ generates a strongly continuous one-parameter semigroup of bounded operators on $\mc H$. What is more, in view of the regularity result Proposition \ref{Prop:smoothness}, we consider the evolution in Sobolev spaces of arbitrarily high integer order. In Sec.~\ref{Sec:Lin} we  then restrict the problem again to $\mc H$. \\

For $k \geq 1$, let
\[ 
\mc H_k: =H^k(\mathbb B^d)\times H^{k-1}(\mathbb B^d)
\]
be equipped with the standard norm denoted by $\| \cdot \|_{H^k(\mathbb B^d) \times H^{k-1}(\mathbb B^d)}$. We set
\[ \DD(\tilde{ \mb L}) := C^\infty(\overline{\mathbb B^d})\times C^\infty(\overline{\mathbb B^d}) \]
and consider the densely defined operator  
$\tilde{\mb L}: \DD(\tilde{ \mb L}) \subset \mc H_k \to \mc H_k$.  We now state the central result of this section.

\begin{proposition} \label{time-evolution}
	Let $d \in \{7,9\}$ and $k \geq 3$. The operator $\tilde{\mb L}:  \DD(\tilde{ \mb L})   \subset \mc H_k  \to \mc H_k$ is closable and its closure $\mb L_k: \DD(\mb L_k) \subset \mc H_k   \to \mc H_k$ generates a strongly continuous semigroup 
	$\mb S_k:[0,\infty)\rightarrow \mathcal{B}(\mc H_k)$ which satisfies
	\begin{align}\label{BoundSk}
		\nr{\mb S_k(\tau) \mb u }_{H^k(\mathbb B^d) \times H^{k-1}(\mathbb B^d)} \leq M_k e^{-\frac{1}{2}\tau} \| \mb u \|_{H^k(\mathbb B^d) \times H^{k-1}(\mathbb B^d)},
	\end{align}
	for all $\mb u \in \mc H_k$, all $\tau\geq 0$, and some $M_k>1$. Furthermore, the following holds for the spectrum of $\mb L_k$ 
	\begin{equation}\label{Eq:Spec_L_k}
		\sigma(\mb L_k)\subset \left\{z\in\C: \Re z\leq-\tfrac{1}{2}\right\},
	\end{equation}
	and the resolvent has the following bound 
		\[ \nr{\mb R_{\mb L_k}(\l) \mb f}_{H^k(\mathbb B^d) \times H^{k-1}(\mathbb B^d)}\leq  \frac{M_k  }{\Re\l+\frac{1}{2}}\|\mb f \|_{H^k(\mathbb B^d) \times H^{k-1}(\mathbb B^d)},  \]
	for $\l\in\C$ with $\Re\l>-\frac{1}{2}$, and $\mb f \in \mc H_k$.
\end{proposition}

\begin{remark}
	We prove Proposition \ref{time-evolution} via the Lumer-Phillips Theorem. By using the standard inner product on $\mc H_k$, one can easily prove existence of the semigroup $(\mb S_k(\tau))_{\tau \geq 0}$, but in order to show that it decays exponentially, and to prove the growth bound \eqref{BoundSk} in particular, we need to introduce an appropriate equivalent inner product. Necessity for such approach will become apparent in the proof of Lemma \ref{Le:Free_Evol_LP} below.  We  note that for $d =9$ the restriction on $k$ is optimal within the class of integer Sobolev spaces. In particular, for scaling reasons exponential decay cannot be expected at lower integer regularities. For $d=7$, a similar statement can be obtained for $k=2$. 
\end{remark}

For $d \in \{7,9\}$ and $k \geq 3$ we define the following sesquilinear form 
\begin{align*}
	(\cdot |\cdot)_{\mc H_k}: \left( C^\infty(\overline{\mathbb{B}^d})\times C^\infty(\overline{\mathbb{B}^d})\right)^2\rightarrow \C, \quad (\mb u|\mb v)_{\mc H_k} =\sum_{j=1}^{k} (\mb u| \mb v)_j,
\end{align*}
where 
\begin{align*}
	\begin{split}
		(\mb u | \mb v)_1&=\int_{\mathbb S^{d-1}}\pt_iu_1(\o)\overline{\pt^iv_1(\o)}d\s(\o)+\int_{\mathbb S^{d-1}}u_1(\o)\overline{v_1(\o)} d\s(\o)
		+\int_{\mathbb S^{d-1}}u_2(\o)\overline{v_2(\o)}d\s(\o) \\
		(\mb u | \mb v)_2&=\int_{\mathbb B^{d}}\pt_i \Delta u_1(\xi)\overline{\pt^i\Delta v_1(\xi)}d\xi
		+\int_{\mathbb B^{d}}\pt_i\pt_ju_2(\xi)\overline{\pt^i\pt^jv_2(\xi)}d\xi
		\\
		&+\int_{\mathbb S^{d-1}}\pt_iu_2(\o)\overline{\pt^iv_2(\o)}d\s(\o),
		\\
		(\mb u | \mb v)_3&=4 \int_{\mathbb B^{d}}\pt_i\pt_j\pt_k u_1(\xi)\overline{\pt^i\pt^j\pt^kv_1(\xi)}d\xi
		+4 \int_{\mathbb B^{d}}\pt_i\pt_j u_2(\xi)\overline{\pt^i\pt^j v_2(\xi)}d\xi
		\\
		&+4 \int_{\mathbb S^{d-1}}\pt_i\pt_ju_1(\o)\overline{\pt^i\pt^jv_1(\o)}d\s(\o),
	\end{split}
\end{align*}
and for $j \geq 4$ we use the standard $\dot H^{j}(\mathbb B^{d}) \times \dot H^{j-1}(\mathbb B^{d})$ inner product
\begin{align}
	(\mb u | \mb v)_j&=(u_1|v_1)_{\dot{H}^j(\mathbb B^{d})}+(u_2|v_2)_{\dot H^{j-1}(\mathbb B^{d})}.
\end{align}

We then set
\[
\nr{\mb u}_{\HH_k}:=\sqrt{(\mb u|\mb u)_{\HH_k}}.
\]
%In what follows, we show that this defines a norm on $C^\infty(\overline{\mathbb{B}^d})\times C^\infty(\overline{\mathbb{B}^d})$, which by boundedness and density can be extended to a norm on the whole $\mc H_k$. The new norm is moreover equivalent to the standard one on $\mc H_k$. 
In the following, for brevity, we use the notation $(\cdot|\cdot)_j = \|\cdot\|^2_j$, $j = 1, \dots, k$ for different parts of $(\cdot|\cdot)_{\mc  H_k}$. 

\begin{lemma}\label{equivalent}
	Let $d \in \{7,9\}$ and $k \geq 3$. We have that 
	\[ \| \mb u \|_{\mc H_k} \simeq  \| \mb u \|_{H^k(\mathbb B^{d}) \times H^{k-1}(\mathbb B^{d})} \]
	for all $\mb u\in C^\infty(\overline{\mathbb{B}^d})\times C^\infty (\overline{\mathbb{B}^d})$. In particular, $\| \cdot \|_{\mc H_k}$ defines an equivalent norm on $\mc H_k$. 
\end{lemma}

\begin{proof}
	Note that it suffices to prove the following
	\begin{align}\label{Eq:Equ}
		\|\mb u\|^2_{H^3(\mathbb B^{d}) \times H^2(\mathbb B^{d})}  \lesssim \sum_{j = 1}^{3} \| \mb u \|^2 _{j}  \lesssim \|\mb u\|^2_{H^3(\mathbb B^{d})\times H^2(\mathbb B^{d})}.
	\end{align} 
	The first estimate in \eqref{Eq:Equ} follows from the fact that 
	\[ \nr{u}^2_{L^2(\mathbb{B}^d)}\lesssim \nr{\nabla u}^2_{L^2(\mathbb{B}^d)} + \nr{u}^2_{L^2(\mathbb{S}^{d-1})},
	\]
	for all $u\in C^\infty(\overline{\mathbb{B}^d})$, which is a simple consequence of the identity
	\begin{equation}\label{Eq:DivergenceTh_Id}
		\int_{\mathbb{S}^{d-1}}   |u(\omega)|^2 d\sigma(\omega) = 
		\int_{\mathbb B^d}  \mathrm{div} ( \xi |u(\xi)|^2 ) d \xi = 
		\int_{\mathbb B^d} \left ( d |u(\xi)|^2 + \xi^i u(\xi) \overline{\partial_i u(\xi)}
		+ \xi^i \overline{u(\xi)} \partial_i u(\xi)  \right) d \xi.
	\end{equation}
	Using this, it is easy to see that 
	\[ \|u \|_{H^2(\mathbb B^d)} \lesssim \int_{\mathbb B^d} \partial_i \partial_j u(\xi) \overline{\partial^i \partial^j u(\xi)}d\xi + \int_{\mathbb{S}^{d-1}} \partial_i u(\omega) \overline{\partial^i u(\omega)} d\sigma(\omega) + \int_{\mathbb{S}^{d-1}}  |u(\omega)|^2 d\sigma(\omega), \]
	for all $u \in C^\infty (\overline{\mathbb{B}^d})$.  Similar bounds imply the first inequality in \eqref{Eq:Equ}.
	Another consequence of Eq.~\eqref{Eq:DivergenceTh_Id} is the trace theorem, which asserts that
	\[
	\int_{\mathbb{S}^{d-1}}|u(\o)|^2d\sigma(\o)\lesssim\nr{u}^2_{H^1(\mathbb B^d)},
	\]
	for all $u\in C^\infty (\overline{\mathbb{B}^d})$; using this, it is straightforward to obtain the second inequality in  \eqref{Eq:Equ}. Hence, we obtain the claimed 
	estimates in Lemma \ref{equivalent} for all  $\mb u\in C^\infty(\overline{\mathbb{B}^d})\times C^{\infty} (\overline{\mathbb{B}^d})$ and by density, we extend this to all of $\mc H_k$.  
\end{proof}

Now we turn to proving Proposition \ref{time-evolution}. As the first auxiliary result, we have the following dissipation property of $\tilde{\mb L}$. 

\begin{lemma}\label{Le:Free_Evol_LP}
	Let $d \in \{7,9\}$ and $k \geq 3$. Then, 
	\[ \Re (\tilde{\mb L} \mb u | \mb u )_{\HH_k} \leq  - \tfrac{1}{2} \| \mb u \|^2_{\mc H_k} \]
	for all $\mb u \in \DD(\tilde{ \mb L}).$
\end{lemma} 

The proof is provided in Section \ref{App1} of the appendix. To apply the Lumer-Phillips theorem, we also need the following density property of $\tilde{\mb L}$.
\begin{lemma}\label{Lem:Density}
	Let $d \in \{7,9\}$ and $k \geq 3$. There exists $\lambda > - \frac12$ such that $ \ran(\lambda-\tilde{ \mb L})$ is dense in $\mc H_k$. 
\end{lemma} 

\begin{proof}
	Let $d \in \{7,9\}$ and $k \geq 3$. We prove the statement by showing that there is exists a $\lambda$ such that given $\mb f \in \mc H_k$ and $\varepsilon>0$ there is some $\mb f_\varepsilon$ in the $\varepsilon$-neighborhood of $\mb f$ for which the equation $(\l-\tilde{\mb L})\mb u=\mb f_\varepsilon$ admits a solution in $\mathcal{D}(\tilde{\mb L})$. First, by density there is $\tilde{\mb f} \in C^\infty(\overline{\mathbb{B}^d}) \times C^\infty(\overline{\mathbb{B}^d})$ for which $\|\tilde{\mb f}-\mb f\|_{H^k(\mathbb B^d)\times H^{k-1}(\mathbb B^d)}<\frac{\eps}{2}$. Then, for $n\in\N$ we define $\mb f_n:=(f_{1,n},f_{2,n})$ with
	\begin{align*}
		f_{1,n}=\sum_{\ell=0}^n P_{\ell} \tilde{f}_1 \quad \text{and} \quad f_{2,n}=\sum_{\ell=0}^n P_{\ell} \tilde{f}_2,
	\end{align*}
	where $P_\ell$ are the projection operators defined in \eqref{Decomp:Projection}. Furthermore, according to \eqref{Decomp:SpherHarm_Hk} there exists an index $N\in\N$ for which $\|\mb f_N-\tilde{\mb f}\|_{H^k(\mathbb B^d)\times H^{k-1}(\mathbb B^d)}<\frac{\eps}{2}$. It is therefore sufficient to consider
	\begin{align}\label{density0}
		(\l-\tilde{\mb L})\mb u=\mb f_N
	\end{align}
	and produce a solution $\mb u\in \mathcal{D}(\tilde{\mb L})$. First, we rewrite Eq.~\eqref{density0} as a system of equations in $u_1$ and $u_2$
	\begin{gather}
		-(\d^{ij}-\xi^i\xi^j)\pt_i\pt_j u_1(\xi)+2(\l+3) \xi^i\pt_iu_1(\xi)+(\l+3)(\l+2)u_1(\xi)=g_N(\xi),\label{density2} \\
		u_2(\xi)=\xi^i\pt_iu_1(\xi)+(\l+1)u_1(\xi)-f_{1,N}(\xi),\label{density1}
	\end{gather}
	where
	\begin{align*}
		g_N(\xi)=\xi^i\pt_i f_{1,N}(\xi)+(\l+3)f_{1,N}(\xi)+f_{2,N}(\xi).
	\end{align*}
	We now treat the case $d=9$, for which we choose $\l=\frac{5}{2}$. With this choice, Eq.~\eqref{density2} reads as
	\begin{align} \label{density3}
		-(\d^{ij}-\xi^i\xi^j)\pt_i\pt_ju_1(\xi) + 11 \xi^i\pt_iu_1(\xi) +\frac{99}{4} u_1(\xi)=g_N(\xi).
	\end{align}
	Note that $g_N$ is a finite linear combination of spherical harmonics, and this allows us to decompose the PDE \eqref{density3} which is posed on $\mathbb{B}^9$ into a finite number of ODEs posed on the interval $(0,1)$. To this end, we switch to spherical coordinates $\rho=|\xi|$ and $\omega=\frac{\xi}{|\xi|}$. In particular, the relevant differential expressions transform in the following way
	\begin{align*}
		\xi^i\pt_iu(\xi)&=\rho\pt_\rho u(\rho\omega),
		\\
		\xi^i\xi^j\pt_i\pt_j u(\xi)&=\rho^2\pt^2_\rho u(\rho\omega),
		\\
		\pt^i\pt_i u(\xi)&=\left(\pt_\rho^2+\frac{8}{\rho}\pt_\rho+\frac{1}{\rho^2}\Delta_\omega^{\mathbb{S}^8} \right)u(\rho\omega).
	\end{align*}
	Consequently, Eq.~\eqref{density3} becomes
	\begin{align}\label{density31}
		\left(-(1-\rho^2)\pt_\rho^2+\left(-\frac{8}{\rho}+11\rho \right)\pt_\rho +\frac{99}{4}-\frac{1}{\rho^2}\Delta_\omega^{\mathbb{S}^8} \right)u(\rho\omega)=g_N(\rho\omega).
	\end{align}
%	Note that the spectrum of the Laplace-Beltrami operator on $L^2(\mathbb{S}^8)$ is given by
%	\[
%	\sigma(-\Delta_\omega^{\mathbb{S}^8})=\sigma_p(-\Delta_\omega^{\mathbb{S}^8})=\{\ell(\ell+7): \ell \in\N_0 \}.
%	\]
%	
%	
	Now we decompose the right hand side of \eqref{density31} into spherical harmonics
	\begin{align*}
		g_N(\rho \o)=\sum_{\ell=0}^N \sum_{m\in \Omega_\ell}g_{\ell,m}(\rho)Y_{\ell,m}(\o),
	\end{align*}
	for some $g_{\ell,m}\in C^\infty[0,1]$. Then by inserting the ansatz
	\begin{align}\label{Eq:u_1}
		u_1(\rho\o)=\sum_{\ell=0}^N \sum_{m\in \Omega_\ell}u_{\ell,m}(\rho)Y_{\ell,m}(\o)
	\end{align}
	into Eq.~\eqref{density31}, we obtain a system of ODEs
	\begin{align} \label{density4}
		\left( -(1-\rho^2)\pt_\rho^2 +\left(-\frac{8}{\rho}+11\rho\right)\pt_\rho + \frac{\ell(\ell+7)}{\rho^2}+\frac{99}{4} \right)u_{\ell,m}(\rho)=g_{\ell,m}(\rho),
	\end{align}
	for $\ell=0,\dots, N $ and $m\in \Omega_\ell $.
	For later convenience, we first set $v_{\ell,m}(\rho)=\rho^3u_{\ell,m}(\rho)$ and thereby transform \eqref{density4} into
	\begin{align} \label{density5}
		\left(-(1-\rho^2)\pt^2_\rho +\left(-\frac{2}{\rho} +5 \rho\right)\pt_\rho + \frac{(\ell+4)(\ell+3)}{\rho^2}+\frac{15}{4}\right)v_{\ell,m}(\rho)=\rho^3 g_{\ell,m}(\rho).
	\end{align}
%	Similar equations have been considered in previous works and we refer in particular to \cite{GlogicSchoerkhuber2021}, Proof of Lemma $3.4$ for more references.
 Then, by means of further change of variables $v_{\ell,m}(\rho)=\rho^{\ell+3}w_{\ell,m}(\rho^2)$ we turn the homogeneous version of \eqref{density5} into a hypergeometric equation in its canonical form
	\begin{align}\label{density6}
		z(1-z)w''_{\ell,m}(z)+\left(c-(a+b+1)z \right)w'_{\ell,m}(z)-abw_{l,m}(z)=0,
	\end{align}
	where
	\begin{align*}
		a=\frac{9+2\ell}{4}, \quad b=a+\frac{1}{2}=\frac{11+2\ell}{4}, \quad c=2a=\frac{9+2\ell}{2}.
	\end{align*}
	Equation \eqref{density6} admits two solutions
	\begin{align*}
		\phi_{0,\ell}(z)={}_2F_1\left(a,a+\frac{1}{2},2a,z\right), \quad \phi_{1,\ell}(z)={}_2F_1\left(a,a+\frac{1}{2},\frac{3}{2},1-z \right),
	\end{align*}
	which are analytic around $z=0$ and $z=1$ respectively, see \cite{NIST10}. In fact, the functions $\phi_{0,\ell}$ and $\phi_{1,\ell}$ can be expressed in closed form
	\begin{align*}
		\phi_{0,\ell}(z)&=\frac{1}{\sqrt{1-z}} \left( \frac{2}{1+\sqrt{1-z}}\right)^{\frac{7}{2}+\ell},
		\\
		\phi_{1,\ell}(z)&= \sqrt{1-z}\left(\left(\frac{1}{1-\sqrt{1-z}} \right)^{\frac{7}{2}+\ell} - \left( \frac{1}{1+\sqrt{1-z}}\right)^{\frac{7}{2}+ \ell} \right),
	\end{align*}
	see  \cite{NIST10}, p.~386-387. Now by undoing the change of variables from above, we get
	$\psi_{\ell,0} = \rho^{\ell + 3}\phi_{0,\ell}(\rho^2)$ and $\psi_{\ell,1} = \rho^{\ell + 3}\phi_{1,\ell}(\rho^2)$ as solutions to the homogeneous version of Eq.~\eqref{density5}. Furthermore, the Wronskian is $W(\psi_{0,\ell},\psi_{1,\ell})(\rho) = C_{\ell} (1-\rho^2)^{-\frac{3}{2}} \rho^{-2}$ for some non-zero constant $C_{\ell}$. Then, by the variation of constants formula we obtain a solution to Eq.~\eqref{density5} on $(0,1)$,	
	\begin{align}\label{Eq:Nonhom_sol}
		v_{\ell,m}(\rho) &= - \psi_{\ell,0}(\rho) \int_{\rho}^{1} \frac{\psi_{\ell,1}(s)}{W(\psi_{\ell,0},\psi_{\ell,1})(s)} \frac{s^3 g_{\ell,m}(s)}{1-s^2} ds 
		-  \psi_{\ell,1}(\rho)  \int_{0}^{\rho} \frac{\psi_{\ell,0}(s)}{W(\psi_{\ell,0},\psi_{\ell,1})(s)} \frac{s^3 g_{\ell,m}(s)}{1-s^2} ds \nonumber\\
		&=- \psi_{\ell,0}(\rho) \int_{\rho}^{1} \psi_{\ell,1}(s)\sqrt{1-s}h_{\ell,m}(s) ds 
		-  \psi_{\ell,1}(\rho)  \int_{0}^{\rho} \psi_{\ell,0}(s)\sqrt{1-s}h_{\ell,m}(s) ds, %\label{Eq:Nonhom_sol}
	\end{align}
	where $h_{\ell,m} \in C^\infty[0,1]$. Obviously $v_{\ell,m} \in C^\infty(0,1)$. We claim that $v_{\ell,m} \in C^\infty(0,1]$. To see this, we note that at
	$\rho=1$, the set of Frobenius indices of Eq.~\eqref{density5} is $\{-\tfrac{1}{2},0\}$. Hence, near $\rho=1$, there is another solution, linearly independent of $\psi_{\ell,1}$, which has the following form $(1-\rho)^{-1/2}\psi_{\ell,2}(\rho)$, where $\psi_{\ell,2}$ is analytic at $\rho=1$. Hence,
	\begin{equation}\label{Eq:Frob_sol1}
		\psi_{\ell,0}(\rho)= c_{\ell,1}\psi_{\ell,1}(\rho) + c_{\ell,2}\frac{\psi_{\ell,2}(\rho)}{\sqrt{1-\rho}},
	\end{equation}
	for some constants $c_{\ell,1},c_{\ell,2}$. Now, by letting
	\[ \alpha_{\ell,m} :=  \int_{0}^{1} \psi_{\ell,0}(s)\sqrt{1-s}h_{\ell,m}(s) ds, \]
	and inserting \eqref{Eq:Frob_sol1} into Eq.~\eqref{Eq:Nonhom_sol}, we get that 
	\begin{align*}
		v_{\ell,m}(\rho) =
		-c_{\ell,2} \frac{\psi_{\ell,2}(\rho)}{\sqrt{1-\rho}} &\int_{\rho}^{1} \psi_{\ell,1}(s)\sqrt{1-s}h_{\ell,m}(s) ds \\
		&-\alpha_{\ell,m}\psi_{\ell,1}(\rho)
		+  c_{\ell,2}\psi_{\ell,1}(\rho)  \int_{\rho}^{1} \psi_{\ell,2}(s)h_{\ell,m}(s) ds.
	\end{align*}
	The second and the third term above are obviously smooth up to $\rho=1$; for the first term, the square root factors in fact cancel out, as can easily be seen via substitution $s=\rho+ (1-\rho)t$, and smoothness of $v_{\ell,m}$ up to $\rho=1$ follows. Consequently, the function $u_1$ defined in \eqref{Eq:u_1} belongs to $C^{\infty}(\overline{\mathbb B^{9}} \setminus \{0\})$ and it solves Eq.~\eqref{density3} in the classical sense away from zero. Furthermore, from~\eqref{Eq:Nonhom_sol} one can check that $u_{\ell,m}$ and
	$u'_{\ell,m}$ are bounded near zero, and hence $u_1 \in H^1(\mathbb B^9)$. In particular, $u_1$ solves  Eq.~\eqref{density3} in the weak sense on $\mathbb B^{9}$ and since the right hand side is a smooth function, we conclude that  $u_1 \in C^{\infty}(\mathbb B^{9})$ by elliptic regularity. Consequently,  $u_1 \in C^{\infty}(\overline{\mathbb B^{9}})$, and therefore $u_2 \in C^{\infty}(\overline{\mathbb B^{9}})$ according to Eq.~\eqref{density1}. In conclusion, $\mb u :=(u_1,u_2) \in \mathcal{D}(\tilde{\mb L})$ solves Eq.~\eqref{density0}.\\
	For $d=7$, the same proof can be repeated by choosing $\lambda = \frac{3}{2}$. Namely, by decomposing the functions into spherical harmonics and by introducing the new variable $\tilde v_{\ell,m}(\rho)=\rho^2 u_{\ell,m}(\rho)$, the problem is reduced to 
	\begin{align*} 
		\left(-(1-\rho^2)\pt^2_\rho +\left(-\frac{2}{\rho} +5 \rho\right)\pt_\rho + \frac{(\ell+3)(\ell+2)}{\rho^2}+\frac{15}{4}\right)\tilde{v}_{\ell,m}(\rho)=\rho^2 g_{\ell,m}(\rho),
	\end{align*}
	which the same as Eq.~\eqref{density5} up to a shift in $\ell$ and the weight on the right hand side. Hence, the same reasoning applies. 
\end{proof}

\begin{proof}[Proof of Proposition \ref{time-evolution}]
	Based on Lemmas \ref{Le:Free_Evol_LP} and \ref{Lem:Density}, the Lumer-Phillps theorem (see \cite{Engel}, p.~83, Theorem 3.15) together with Lemma \ref{equivalent} implies that $\tilde{\mb L}$ is closable in $\mc H_k$, and that its closure $\mb L_k$ generates a semigroup $(\mb S_k(\tau))_{\tau \geq 0}$ for which \eqref{BoundSk} holds. The rest of the proposition follows from standard semigroup theory results, see e.g.~\cite{Engel}, p.~55, Theorem 1.10. 
\end{proof}

We conclude this section with proving certain restriction properties of the semigroups $(\mb S_{k}(\tau))_{\tau \geq 0}$. This will be crucial in showing persistence of regularity for the nonlinear equation. 

\begin{lemma}\label{lemma.restriction}
	Let $d \in \{7,9\}$ and $k \geq 3$. For any $j \in \N$, the semigroup $(\mb S_{k+j}(\tau))_{\tau \geq 0}$ is the restriction of $(\mb S_{k}(\tau))_{\tau \geq 0}$ to $\HH_{k+j}$, i.e.,
	\begin{equation*}
		\mb S_{k+j}(\tau)=\mb S_{k}(\tau)|_{\HH_{k+j}}
	\end{equation*}
	for all $\tau\geq0$.
	 In particular, we have 
	the growth bound
	\[ \|\mb S_{k}(\tau) \mb u\|_{H^{k+j}(\mathbb B^d) \times H^{k+j-1}(\mathbb B^d)}  \lesssim_j e^{-\frac12 \tau} \| \mb u\|_{H^{k+j}(\mathbb B^d) \times H^{k+j-1}(\mathbb B^d)}  \]
	for all $\mb u \in \HH_{k+j}$ and all $\tau\geq0$.
\end{lemma}

\begin{proof} Let $d \in \{7,9\}$ and $k \geq 3$.
	We prove the claim only for $j=1$, as the general case follows from the arbitrariness of $k$. The crucial ingredients of the proof are continuity of the embedding $\mc H_{k+1} \hookrightarrow \mc H_{k}$,
%	 i.e., that
%	\begin{align}\label{NormEst}
%		\| \mb u \|_{{H^{k}(\mathbb B^d) \times H^{k-1}(\mathbb B^d)}} \lesssim \| \mb u \|_{{H^{k+1}(\mathbb B^d) \times H^{k}(\mathbb B^d)}}
%	\end{align}
%	for all $\mb u \in \mc H_{k+1}$, 
	and the fact that $\DD(\tilde{ \mb L})$ is a core for both $\mb L_k$  and $\mb L_{k+1}$. First, we prove that $\mb L_{k+1}$ is a restriction of $\mb L_k$; more precisely we show that
	\begin{align}\label{inclusion}
		\mathcal{D}(\mb L_{k+1})\subset \mathcal{D}(\mb L_k) \quad \mathrm{and} \quad \mb L_{k+1}\mb u=\mb L_k\mb u,
	\end{align}
	for all $\mb u\in\mathcal{D}(\mb L_{k+1})$. For $\mb u\in \mathcal{D}(\tilde{\mb L})$, from the definition of $\mb L_{k+1}$ and $\mb L_k$,  it follows that 
	$\mb u\in \mathcal{D}(\mb L_{k+1})\cap \mathcal{D}(\mb L_k)$ and $\mb L_{k+1}\mb u=\mb L_k\mb u=\tilde{\mb L}\mb u$. Let now
	$\mb u\in \mathcal{D}(\mb L_{k+1})$. Since $(\mb L_{k+1},\mathcal{D}(\mb L_{k+1}))$ is closed, there exists a sequence $(\mb u_n)_{n\in\N}\subset \mathcal{D}(\tilde{\mb L})$, such that
	\[
	\mb u_n \xrightarrow{\HH_{k+1}} \mb u, \quad \mathrm{and} \quad \tilde{\mb L}\mb u_n\xrightarrow{\HH_{k+1}}\mb L_{k+1}\mb u.
	\]
	From the embedding $\mc H_{k+1} \hookrightarrow \mc H_{k}$ we infer that 
	\[
	\mb u_n \xrightarrow{\HH_{k}} \mb u, \quad \mathrm{and} \quad \tilde{\mb L}\mb u_n\xrightarrow{\HH_{k}}\mb L_{k+1}\mb u,
	\]
	and by the closedness of $\mb L_k$ it follows that  $\mb u\in\mathcal{D}(\mb L_k)$ and $\mb L_{k+1}\mb u=\mb L_{k}\mb u$.
	Now let $\lambda\in \rho(\mb L_{k+1})\cap\rho(\mb L_k)$. From \eqref{inclusion} we get that $\mb R_{\mb L_{k+1}}(\lambda)=\mb R_{\mb L_k}(\lambda)|_{\HH_{k+1}}$. Now, given $\mb u\in \HH_{k+1}$ we get by the Post-Widder inversion formula (see \cite{Engel}, p.~223, Corollary 5.5) and the embedding $\mc H_{k+1} \hookrightarrow \mc H_{k}$, that for every $\t>0$,
	\[
	\mb S_{k+1}(\t)\mb u =\lim_{n\rightarrow \infty}\left[\frac{n}{\t}\mb R_{\mb L_{k+1}}\left(\frac{n}{\t}\right)\right]^n\mb u=
	\lim_{n\rightarrow \infty}\left[\frac{n}{\t}\mb R_{\mb L_{k}}\left(\frac{n}{\t}\right)\right]^n\mb u=\mb S_{k}(\t)\mb u.
	\]
	 This proves that 
	$(\mb S_{k+1}(\tau))_{\tau \geq 0}$ is the restriction of $(\mb S_{k}(\tau))_{\tau \geq 0}$ to $\HH_{k+1}$. Consequently, from Proposition \ref{time-evolution} we have that 
	\begin{align*}
		\|\mb S_{k}(\tau) \mb u\|_{H^{k+1}(\mathbb B^d) \times H^{k}(\mathbb B^d)} =   \|\mb S_{k+1}(\tau) \mb u\|_{H^{k+1}(\mathbb B^d) \times H^{k}(\mathbb B^d)} 
		\lesssim e^{-\frac12 \tau} \| \mb u\|_{H^{k+1}(\mathbb B^d) \times H^{k}(\mathbb B^d)},  
	\end{align*}
 for all $\mb u \in \HH_{k+1}$ and all $\tau \geq 0$. 
\end{proof}

\section{Linearization around a self-similar solution - Preliminaries on the structure of the spectrum}\label{Sec:Lin}

From now on, for fixed $d  \in \{7,9\}$, we will work solely in the Sobolev space
$H^{\frac{d+1}{2}}(\mathbb B^d)\times H^{\frac{d-1}{2}}(\mathbb B^d)$ which we earlier denoted by $\mc H_{\frac{d+1}{2}}$. To abbreviate the notation, we write 
\[ \mc H := \mc H_{\frac{d+1}{2}}.\]

We also denote by $(\mb S(\tau))_{\tau \geq 0}$ and $\mb L:  \mathcal{D}(\mb L) \subset \mc H \to \mc H$, the corresponding semigroup $(\mb S_k(\tau))_{\tau \geq 0}$ and its generator $\mb L_k$ for $k = \frac{d+1}{2}$.\\

With an eye towards studying the flow near the orbit $\{ \mb U_{a} : a \in \R^d\}$, see Sec.~\ref{Sec:Modulation}, in this section we describe some general properties of the underlying linear operator
\begin{align*}
	 \tilde{\mb L}+\mb L_a', 
	\quad \mb L'_a \mb u :=\begin{pmatrix}
		0
		\\
		V_au_1
	\end{pmatrix}.
\end{align*}
Where
\begin{align}\label{Def:Potential}
	V_a(\xi) := 2 U_a(\xi),
\end{align}
with $ U_a$ given in Eq.~\eqref{explicit-form}. 
\begin{remark}
	We emphasize that the results of this section apply to any smooth $V_a: \overline{\mathbb B^d} \to \R$ that depends smoothly on the parameter $a$. Obviously, such potentials arise in the linearization around smooth self-similar profiles.
\end{remark}

\begin{proposition}\label{group-Sa}
	Fix $d  \in \{7,9\}$. For every $a \in \R^d$, the operator $\mb L'_a:\HH\rightarrow\HH$ is compact, and the operator
	\[ \mb L_a:=\mb L+\mb L'_a, \quad  \DD(\mb L_a):=\DD(\mb  L)\subset \HH\rightarrow \HH, \]
	generates a strongly-continuous semigroup $\mb S_a:[0,\infty)\rightarrow \mathcal{B}(\HH)$. Furthermore, given $\d>0$, there is $K > 0$ such that
	\[ \nr{\mb  L_a-\mb  L_b}\leq K|a-b| \]
	for all $a,b\in \overline{\mathbb{B}^d_\d}$ .
\end{proposition}

\begin{proof}
	The compactness of $\mb L'_a$ follows from the smoothness of $V_a$ and the compactness of the embedding $H^{
		\frac{d+1}{2}}(\mathbb B^d) \hookrightarrow H^{\frac{d-1}{2}}(\mathbb B^d)$. The fact that $\mb L_a$ generates a semigroup is a consequence of the Bounded Perturbation Theorem, see e.g.~\cite{Engel}, p.~158.
	For the Lipschitz dependence on the parameter $a$, we first note that by the fundamental theorem of calculus we have that
	\begin{align}\label{Eq:Lipschitzbounds_Ua}
		V_a(\xi) - V_b(\xi) = (a^j - b^j) \int_0^{1} \partial_{\alpha_j} V_{\alpha(s)}(\xi) ds ,
	\end{align}
	for $\alpha(s) = b + s(a-b)$. This implies that given $\d >0$ we have that 
	\begin{align}\label{Eq:Selfsim_Sol_Lipschitz}
		\| V_a - V_b \|_{\dot H^k(\mathbb B^d)} \lesssim_k |a - b|
	\end{align}
	for all $a,b \in  \overline{\mathbb B^d_{\delta}}$. In particular,
	\begin{align*}
		\nr{V_a-V_b}_{W^{\frac{d-1}{2},\infty}(\mathbb B^d)}\lesssim |a-b|,
	\end{align*}
	and we thus have that
	\begin{align*}
		\nr{(V_a-V_b)u}_{H^{\frac{d-1}{2}}(\mathbb B^d)}\lesssim |a-b|\nr{u}_{H^{\frac{d-1}{2}}(\mathbb B^d) }\lesssim |a-b|\nr{u}_{H^{\frac{d+1}{2}(\mathbb B^d)}},
	\end{align*}
	for all $ u\in C^\infty(\overline{\mathbb{B}^d})$ and all $a,b\in \overline{\mathbb B^d_{\delta}}$, which implies the claim.
\end{proof}

Next, we show that the unstable spectrum of $\mb L_a:\DD(\mb L_a) \subset \HH \rightarrow \HH$ consists of isolated eigenvalues and is confined to a compact region. This is achieved by proving bounds on the resolvent and using compactness of the perturbation.

\begin{proposition} \label{Prop:Structure_Spectrum}
	Fix $d  \in \{7,9\}$. Let $\varepsilon > 0$ and $\d>0$. Then there are constants $\kappa>0$ and $c>0$, such that
	\begin{equation}\label{Eq:Rel_Est}
		\nr{\mb R_{\mb L_a}(\l)}\leq c
	\end{equation}
	for all $a\in\overline{\mathbb{B}^d_\d}$, and for all $\l\in\C$ satisfying $\Re\l\geq-\frac{1}{2}+\eps$ and $|\l|\geq\kappa$. Furthermore, if $\l\in\sigma(\mb L_a)$ with $\Re\l>-\frac{1}{2}$, then $\l$ is an isolated eigenvalue.
\end{proposition}

\begin{proof}
	Let $\la \in \C$ with $\Re \la > -\frac{1}{2}$. Then Proposition \ref{time-evolution} implies that $\la \in \rho(\mb L)$, and we therefore have the identity 
	\begin{align} \label{identity}
		\l-\mb L_a=[1-\mb L'_a \mb R_{\mb L}(\l)](\l-\mb L).
	\end{align}
	In what follows we prove that for suitably chosen $\la$, the Neumann series $\sum_{k=0}^\infty[\mb L'_a \mb R_{\mb L}(\l)]^k$ converges. According to Eq.~\eqref{identity} this yields $\mb R_{\mb L_a}(\l)=\mb R_{\mb L}(\l)\sum_{k=0}^\infty[\mb L'_a \mb R_{\mb L}(\l)]^k$ and then Eq.~\eqref{Eq:Rel_Est} follows from Proposition \ref{time-evolution}. First, observe that given $\delta>0$ we have that
	\begin{equation}\label{Eq:Neuman_est}
		\nr{\mb L'_a \mb R_{\mb L}(\l) \mb f}=\nr{V_a[ \mb R_{\mb L}(\l) \mb f]_1}_{H^{\frac{d-1}{2}}(\mathbb{B})^d} \lesssim \nr{[\mb R_{\mb L}(\l)\mb f]_1}_{H^{\frac{d-1}{2}}(\mathbb{B}^d)},
	\end{equation}
	for all $a\in\overline{\mathbb{B}^d_\d}$ and all $\mb f \in \mc H$. Now, given $\mb f\in\HH$ let $\mb u = \mb R_{\mb L}(\la)\mb f$. Since $(\l-\mb L)\mb u=\mb f$, from the first component of this equation we get that
	\[
	\xi^j\partial_ju_1(\xi)+(\l+2)u_1(\xi)-u_2(\xi)=f_1(\xi)
	\]
	in the weak sense on the ball $\mathbb{B}^d$.
	Consequently
	\[
	\nr{u_1}_{H^{\frac{d-1}{2}}(\mathbb{B}^d)}\lesssim\frac{1}{|\l+2|}\left(\nr{u_1}_{H^{\frac{d+1}{2}}(\mathbb{B}^d)}+\nr{u_2}_{H ^{\frac{d-1}{2}}(\mathbb{B}^d)}+\nr{f_1}_{H^{\frac{d-1}{2}}(\mathbb{B}^d)}\right).
	\]
	Then Proposition~\ref{time-evolution} implies that given $\varepsilon>0$
	\[
	\nr{[\mb R_{\mb L}(\l) \mb f]_1}_{H^{\frac{d-1}{2}}(\mathbb{B}^d)}\lesssim|\l|^{-1}\left( \nr{\mb R_{\mb L}(\l)\mb f}+\nr{\mb f}\right)\lesssim |\l|^{-1}\nr{\mb f}
	\]
	for all $\l\in\C$ with $\Re\l\geq-\frac{1}{2}+\eps$ and all $\mb f \in \mc H$. Together with \eqref{Eq:Neuman_est}, this gives
	\[
	\nr{\mb L'_a \mb R_{\mb L}(\l)\mb f} \lesssim|\l|^{-1}\nr{\mb f},
	\]
	and the uniform bound \eqref{Eq:Rel_Est} holds for some $c>0$ when we restrict to $|\la|\geq \kappa$ for suitably large $\kappa$. 
	The second statement follows from the compactness of $\mb L'_a$. Indeed, if $\Re \l>-\frac{1}{2}$ then $\l \in \rho(\mb L)$, and according to \eqref{identity} we have that $\lambda \in \sigma(\mb L_a)$ only if the operator $1-\mb L'_a \mb R_{\mb L}(\l)$ is not bounded invertible, which is equivalent to $1$ being an eigenvalue of the compact operator $\mb L'_a \mb R_{\mb L}(\l)$, which according to Eq.~\eqref{identity} implies that $\la$ is an eigenvalue of $\mb L_a$. The fact that $\la$ is isolated follows from the Analytic Fredholm Theorem (see \cite{Simon2015_4}, Theorem 3.14.3, p.~194) applied to the mapping $\l\mapsto  \mb L'_a \mb R_{\mb L}(\l)$ defined on  $\mathbb{H}_{-\frac{1}{2}}=\{ \l\in\C : \Re\l > - \frac{1}{2}\}$.
\end{proof}

\begin{remark}
	The previous proposition implies that there are finitely many  unstable spectral points of $\mb L_a$, i.e., the ones belonging to $\Hb:=\{ \la \in \C: \Re \la \geq 0$\}, all of which are eigenvalues. This can actually be abstractly shown just by using the compactness of $\mb L'_a$, see \cite{Glo21}, Theorem B.1. We nonetheless need Proposition \ref{Prop:Structure_Spectrum} as it allows us later on to reduce the spectral analysis of $\mb L_a$ for all small $a$ to the case $a=0$, see Sec.~\ref{Sec:L_a}.
\end{remark}

  Note that the eventual presence of unstable spectral  points of $\mb L_a$ prevents decay of the associated semigroup $(\mb S_a(\tau))_{\tau \geq 0}$ on the whole space $\mc H$. What is more, since $\mb L'_a$ is compact, a spectral mapping theorem for the unstable spectrum holds (see \cite{Glo21}, Theorem B.1), and hence eventual growing modes of $(\mb S_a(\tau))_{\tau \geq 0}$ are completely determined by the unstable spectrum of $\mb L_a$ and the associated eigenspaces. Therefore, in what follows we turn to spectral analysis of $\mb L_a$. First, we show an important result which relates solvability of the spectral equation $(\la - \mb L_a)\mb u=0$ for $a=0$, $\la \in \Hb$, to the existence of \emph{smooth} solutions to a certain ordinary differential equation. We note that for $a=0$, the potential $V_a$ is radial, more precisely $V_0(\xi) = 2 U_0(\xi) = 2 U(|\xi|)=:V(|\xi|)$ with $U$ given in \eqref{Eq:Blowup_profile_radial}.

\begin{proposition}\label{Prop:Spectral_ODE}
	Fix $d  \in \{7,9\}$. Let  $\lambda \in \C$ with $\Re \lambda \geq 0$. Then $\l\in\sigma(\mb L_0)$ if and only if there are $\ell \in \N_0$ and $f \in C^{\infty}[0,1]$ such that	
	\begin{align}\label{diff-op}
		\mathcal{T}^{(d)}_\ell(\lambda)f(\rho) :=(1-\rho^2)f''(\rho)+&\left(\frac{d-1}{\rho}-2(\lambda+3)\rho \right)f'(\rho)
		\\
		&-\left( (\lambda+2)(\lambda+3)+\frac{\ell(\ell+d-2)}{\rho^2} - V(\rho) \right)f(\rho) = 0
		\nonumber
	\end{align}
	for all $\rho \in (0,1)$.
	
\end{proposition}

\begin{proof}
	Let $\la \in \Hb \cap \sigma(\mb L_0)$. By Proposition \ref{Prop:Structure_Spectrum}, $\lambda$ is an eigenvalue, and hence there is a nontrivial $\mb u \in \DD(\mb L_0)$ satisfying $(\lambda-\mb L_0)\mb u=0$. By a straightforward calculation we get that the components $u_1$ and $u_2$ satisfy the following two equations
	\begin{equation} \label{elliptic}
		-(\delta^{ij}-\xi^i\xi^j)\partial_{i}\partial_{j}u_1(\xi)+2(\lambda+3)\xi^j\partial_{j}u_1(\xi)+(\lambda+3)(\lambda+2)u_1(\xi)-V_0(\xi)u_1(\xi)=0
	\end{equation}
	and
	\begin{equation}
		u_2(\xi)=\xi^j\partial_{j}u_1(\xi)+(\lambda+2)u_1(\xi),
	\end{equation}
	weakly on $\mathbb{B}^d$. Since $u_1\in H^\frac{d+1}{2}(\mathbb{B}^d)$, we get by elliptic regularity that $u_1\in C^\infty(\mathbb{B}^d)$.
	Furthermore, we may decompose $u_1$ into spherical harmonics
	\begin{align}\label{Eq:ExpSph}
		u_1(\xi)=\sum_{\ell=0}^\infty\sum_{m\in \Omega_\ell}\left(u_1(|\xi|\cdot)|Y_{\ell,m} \right)_{L^2(\mathbb{S}^{d-1})}Y_{\ell,m}\left(\frac{\xi}{|\xi|}\right)=\sum_{\ell=0}^\infty\sum_{m\in\Omega_{\ell}}u_{\ell,m}(\rho)Y_{\ell,m}(\omega),
	\end{align}
	where $\rho=|\xi|$ and $\omega=\xi/|\xi|$. To be precise, the expansion above holds in $H^k(\mathbb{B}^d_{1-\epsilon})$ for any $k\in\N$ and $\epsilon>0$, see Eqns.~\eqref{Decomp:Projection} and \eqref{Decomp:SpherHarm_Hk}. Since the potential $V_0$ is radially symmetric, Eq.~\eqref{elliptic} decouples by means of \eqref{Eq:ExpSph} into a system of infinitely many ODEs
	\begin{equation}\label{ODE}
		\mathcal{T}^{(d)}_\ell(\lambda)u_{\ell,m}(\rho)=0,
	\end{equation}
	posed on the interval $(0,1)$, where  the operator $\mathcal{T}^{(d)}_\ell(\lambda)$ is given by \eqref{diff-op}. Since $u_1$ is non-trivial, there are indices $\ell\in\N_0$ and $m\in\Omega_\ell$, such that $u_{\ell,m}$ is non-zero and satisfies \eqref{ODE}. Furthermore, since $u_1 \in C^{\infty}(\mathbb B^d) \cap H^{
		\frac{d+1}{2}}(\mathbb B^d)$, we have that $u_{\ell,m}\in C^\infty[0,1)\cap H^{\frac{d+1}{2}}(\tfrac{1}{2},1)$. Now we prove that $u_{\ell,m}$ is smooth up to $\rho=1$.
	Note that $\rho = 1$ is a regular singular point of equation  \eqref{ODE}, and the corresponding set of Frobenius indices 
	is $\{0,2-\l\}$ when $d=9,$ and $\{0,1-\l\}$ when $d=7$. In the first case, if $\l\notin\{0,1,2\}$, then $u_{\ell,m}$ is either analytic or behaves like $(1-\rho)^{2-\l}$ near $\rho=1$. If $\l\in\{0,1,2\}$, then the non-analytic behavior can be described by $(1-\rho)^2\log(1-\rho)$, $(1-\rho)\log(1-\rho)$ or $\log(1-\rho)$. In each case, singularity can be excluded by the requirement that $u_{\ell,m}\in H^{5}(\tfrac{1}{2},1)$. This implies that $u_{\ell,m}$ belongs to $C^\infty[0,1]$ and solves Eq.~\eqref{diff-op} on $(0,1)$. The same reasoning applies to the case $d=7$. Implication in the other direction is now obvious. \\
\end{proof}

\begin{remark}\label{Remark:SpectralProb_ODEAnalysis}
	Note that the Frobenius theory implies that smooth solutions $f$ from Proposition \ref{Prop:Spectral_ODE} are in fact analytic on $[0,1]$, in the sense that they can be extended to an analytic function on an open interval that contains $[0,1]$. Consequently, determining the unstable spectrum of $\mb L_0$ amounts to solving the connection problem for a family of ODEs. We note that the connection problem is so far completely resolved only for hypergeometric equations, i.e., the ones with three regular singular points, while the ODE \eqref{diff-op} has six of them. In fact, their number can, by a suitable change of variables, be reduced to four, but this nonetheless renders the  standard ODE theory useless. 
	Nevertheless, by building on the techniques developed recently to treat such problems, see~\cite{CosDonGloHua16,CosDonGlo17,Glo18,GlogicSchoerkhuber2021}, for $d=9$ we are able to solve the connection problem for \eqref{diff-op} and we thereby provide in the following section a complete characterization of the unstable spectrum of $\mb L_0$. 
\end{remark}

\section{Spectral analysis for perturbations around $\mb U_a$ - The case $d=9$}\label{Sec:Spectral_Analysis_U}

From now on we restrict ourselves to $d=9$.  

\subsection{Analysis of the spectral ODE}\label{Sec:ODE}

In this section we investigate the ODE \eqref{diff-op} for $d=9$, and for convenience we shorten the notation by letting $\mathcal{T}_\ell(\lambda) :=\mathcal{T}^{(9)}_\ell(\lambda)$, i.e., we have
\begin{align*}
	\mathcal{T}_\ell(\lambda)f(\rho) :=(1-\rho^2)f''(\rho)+&\left(\frac{8}{\rho}-2(\lambda+3)\rho \right)f'(\rho) \nonumber
	\\
	&-\left( (\lambda+2)(\lambda+3)+\frac{\ell(\ell+7)}{\rho^2}-V(\rho) \right)f(\rho),
	\label{diff-op_ODE}
\end{align*}
where the potential is given by 
\begin{equation*}
	V(\rho) = \frac{480 (7- \rho^2)}{(7+5\rho^2)^2}.
\end{equation*}
Now, in view of Proposition \ref{Prop:Spectral_ODE}, given $\ell\in\N_0$, we define the following set
\begin{align*}
	\begin{split}
		\Sigma_\ell:=\{\la \in \Hb : \text{there exists } f_\ell(\cdot;\l)\in C^\infty[0,1]  \text{ satisfying } \mathcal{T}_\ell(\lambda)f_\ell(\cdot;\l)=0 \text{ on } (0,1)\}.
	\end{split}
\end{align*}
The central result of our spectral analysis is the following proposition.

\begin{proposition}\label{prop:SpectralODE}The structure of $ \Sigma_{\ell}$ is as follows.
	\begin{enumerate}%[(i)]
		\item For $\ell = 0$,  $\Sigma_{0} = \{1,3\}$ with corresponding solutions
		\begin{equation*}
			f_0(\rho;1) =\frac{1-\rho^2}{(7+5 \rho^2)^3} \quad \text{and} \quad f_0(\rho;3)=\frac{1}{(7+5 \rho^2)^3},
		\end{equation*}
		which are unique up to a constant multiple.
		\item For $\ell =1$, $\Sigma_{1} = \{0,1\}$, and the corresponding solutions are  
		\begin{equation*}
			f_1(\rho;0) =\frac{\rho(7-3\rho^2)}{(7+5 \rho^2)^3} \quad \text{and} \quad  f_1(\rho;1)  =\frac{\rho (77-5 \rho^2)}{(7+5 \rho^2)^3}.
		\end{equation*}

		\item For all $\ell \geq 2$,  $\Sigma_{\ell} = \varnothing$. 
	\end{enumerate}
\end{proposition}
To prove this proposition, we use an adaptation of the ODE techniques devised in \cite{CosDonGloHua16,CosDonGlo17,Glo18} and \cite{GlogicSchoerkhuber2021}. We will therefore occasionally refer to these works throughout the proof. Also, we found it convenient to split the proof into two cases, $\ell \in \{ 0,1\}$ and $\ell \geq 2$.
\subsubsection{Proof of Proposition \ref{prop:SpectralODE} for $\ell\in \{0,1\}$}
For a detailed heuristic discussion of our approach we refer the reader to \cite{GlogicSchoerkhuber2021}, Sec.~4.1. Namely, the first step is to transform $\mathcal{T}_{\ell}(\lambda) f(\rho) =0$ to an ``isospectral" equation with four regular singular points.  For this,  we let $x = \rho^2$ and we define the new dependent variable $y$ via
\begin{equation*}
	f(\rho) =  \rho^{\ell} (\tfrac{7}{5}+ \rho^2)^{-3} y(\rho^2).
\end{equation*}
This yields the following equation in its canonical Heun form (see \cite{NIST10})
\begin{align}\label{Eq:ODEheun} 
	y''(x)+\left(\frac{\gamma(\ell)}{x}+\frac{\delta(\lambda)}{x-1}-\frac{6}{x-\mu}\right)y'(x)+\frac{\alpha(\ell,\lambda)\beta(\ell,\lambda) x - q(\ell,\lambda) }{x(x-1)(x-\mu)}y(x)=0,
\end{align}
with singularities at $x \in \{0,1,\mu, \infty\}$, where $\mu = - \frac{7}{5}$, $\gamma(\ell) = \frac{9+2 \ell}{2}$, $\delta(\lambda) = \lambda -1$,  $\alpha(\ell,\lambda) = \frac{1}{2} (\lambda -3+\ell)$,  $\beta(\ell,\lambda) = \frac{1}{2} (\lambda -4+\ell)$, and 
\[ q(\ell,\lambda) =  - \frac{1}{20} \left (7 (\lambda -3) (\lambda +8)+7 \ell^2+(14 \lambda +95) \ell \right). \]

By Frobenius theory, any $y \in C^{\infty}[0,1]$ that solves Eq.~\eqref{Eq:ODEheun} on $(0,1)$ is in fact analytic on the closed interval $[0,1]$. Furthermore, the Frobenius indices of Eq.~\eqref{Eq:ODEheun} at $x=0$ are $s_1=0$ and $s_2= -\frac{(7+2\ell)}{2}$. Therefore, for every $\la \in \C$ there is a unique solution (up to a constant multiple) to Eq.~\eqref{Eq:ODEheun}, which is analytic at $x=0$. Furthermore, this solution has a power series expansion of the following form
\begin{equation}\label{Eq:Expansion}
	y_{\ell,\lambda}(x)=\sum_{n=0}^{\infty}a_n(\ell, \la) x^n, \quad a_0(\ell,\la)=1.
\end{equation}  
To determine the coefficients $a_n$, we insert the ansatz \eqref{Eq:Expansion} into Eq.~\eqref{Eq:ODEheun} and thereby obtain the recurrence relation
\begin{equation}\label{Eq:RecRel}
	a_{n+2}(\ell,\la)=A_n(\ell,\la)a_{n+1}(\ell, \la)+B_n(\ell, \la)a_{n}(\ell, \la),
\end{equation}
where
\begin{gather*}
	A_n (\ell,\la)=\frac{7 \lambda  (\lambda +9)+7\ell^2 +\ell (8n + 14 \la +103)+8 n^2+4 (7 \lambda
		+34) n-40}{14 (n+2) (2 \ell+2 n+11)}, \label{Eq:A} \\
	B_n(\ell,\la)=\frac{5 (\lambda +\ell+2 n-4) (\lambda
		+\ell+2 n-3)}{14 (n+2) (2 \ell+2 n+11)}, \label{Eq:B}
\end{gather*}
with the initial condition
\begin{equation}\label{Eq:InitCond}
	a_{-1}(\ell, \la)=0 \quad \text{and} \quad a_0(\ell, \la)= 1.
\end{equation} 	
Now, note that $\la \in \Sigma_{\ell}$ precisely when the radius of convergence of the series \eqref{Eq:Expansion} is larger than 1. To analyze this radius, we resort to results from the theory of difference equations with variable coefficients. Namely, since $\lim_{n\ra\infty}A_n(\ell, \la)=\frac{2}{7}$ and $\lim_{n\ra\infty}B_n(\ell, \la)=\frac{5}{7}$, the so-called characteristic equation of Eq.~\eqref{Eq:RecRel} is 
\begin{equation*}
	t^2-\frac{2}{7}t-\frac{5}{7}=0,
\end{equation*}
and according to Poincar\'e's theorem (see, for  example, \cite{Ela05}, p.~343, or \cite{GlogicSchoerkhuber2021}, Appendix A) we have that either $a_n(\ell, \la)=0$ eventually in $n$ or
\begin{equation*}
	\lim\limits_{n\ra\infty} \frac{a_{n+1}(\ell, \la)}{a_{n}(\ell, \la)}=1
\end{equation*}
or
\begin{equation*}
	\lim\limits_{n\ra\infty} \frac{a_{n+1}(\ell, \la)}{a_{n}(\ell, \lambda)}=  - \frac{5}{7}.
\end{equation*}
To explore this further, we treat cases $\ell=0$ and $\ell=1$ separately. \medskip

\emph{The case $\ell =0$.} 
First, we observe that in this case, there are explicit polynomial solutions for $\lambda = 1$ and  $\lambda = 3$, given by 
\begin{equation}\label{Eq:Poly_Eigenf}
	y_{0,1}(x)=1-x \quad \text{and} \quad y_{0,3}(x)=1,
\end{equation}
respectively. These in turn correspond to $f_0(\cdot;1)$ and $f_0(\cdot;3)$, stated in Proposition \ref{prop:SpectralODE}. So we have that $\{1,3\} \subset \Sigma_0 $. We now show the reversed inclusion. Let $\la \in  \Hb \setminus \{1,3\}$. Since $\ell = 0$, from \eqref{Eq:A} and \eqref{Eq:B} we have 
\begin{gather*}
	A_n (0,\la)=\frac{7 \lambda  (\lambda
		+9)+8 n^2+4 (7 \lambda +34)
		n-40}{14 (n+2) (2 n+11)}, \\
	B_n(0,\la)=\frac{5 (\lambda +2 n-4)
		(\lambda +2 n-3)}{14 (n+2)
		(2 n+11)}.
\end{gather*}
Now, note that the assumption that $a_n(0,\lambda) = 0$ eventually in $n$ contradicts the initial condition \eqref{Eq:InitCond}, as follows by backward substitution. Consequently, we have that either
\begin{equation}\label{Eq:lim01}
	\lim\limits_{n\ra\infty} \frac{a_{n+1}(0, \la)}{a_{n}(0, \la)} = 1,
\end{equation}
or 
\begin{equation}\label{Eq:lim02}
	\lim\limits_{n\ra\infty} \frac{a_{n+1}(0, \la)}{a_{n}(0, \la)} = -\frac{5}{7}.
\end{equation}
We prove that \eqref{Eq:lim01} holds, wherefrom it follows that the radius of convergence of the series \eqref{Eq:Expansion} (when $\ell=0$) is 1, and therefore $\la \notin \Sigma_0 $. To that end, we first compute  
\[ a_2(0,\la)=\frac{(\lambda -3) (\lambda -1) (7\la^2+126\la+680)}{5544},\]
and
\[ a_3(0,\la)=\frac{(\lambda -3) (\lambda -1) (49\la^4+1519\la^3+18494\la^2+84224\la+46080)}{3027024}.\]
Then we define 
\[ r_2(0,\la):=\frac{a_3(0,\la)}{a_2(0,\la)}, \] 
where the common factor $(\la-3)(\la-1)$ (which is an artifact of the existence of polynomial solutions \eqref{Eq:Poly_Eigenf}) is canceled, and consequently, according to \eqref{Eq:RecRel}, for $n \geq 2$ we let
\begin{equation}\label{Eq:RecRel__2}
	r_{n+1}(0,\la)=A_n(0,\la)+\frac{B_n(0,\la)}{r_n(0, \la)}.
\end{equation}
To show \eqref{Eq:lim01}, our strategy is the following. For Eq.~\eqref{Eq:RecRel__2} we construct an approximate solution $\tilde{r}_n$ (which we also call a \emph{quasi-solution})  for which $\lim_{n\rightarrow \infty} \tilde{r}_n(0,\la)=1$ and which is provably close enough to $r_n$ so as to rule out \eqref{Eq:lim02}. The quasi-solution we use  is 
\begin{align}\label{Def:Quasi0}
	\tilde r_{n}(0,\lambda) := \frac{\lambda ^2}{2(2n+9)(n+1)}+\frac{\lambda  (4 n+9)}{2(2n+9)(n+1)}+\frac{2 n+2}{2 n+9}.
\end{align}
We have elaborated on constructing such expressions in \cite{GlogicSchoerkhuber2021}, Sec.~4.2.2 and in \cite{CosDonGloHua16}, Sec.~4.1; one can also check \cite{Glo18}, Secs.~2.6.3 and 2.7.2. Concerning \eqref{Def:Quasi0}, suffice it to say here that we chose a quadratic polynomial in $\la$ with rational coefficients in $n$ so as to emulate the behavior of $r_n(0,\la)$ for both large and small values of the participating parameters. To show that the quasi-solution indeed resembles $r_n(0,\lambda)$, we define the following relative difference function
\begin{equation}\label{Def:Delta} 
	\delta_n(0,\la):=\frac{r_n(0,\la)}{\tilde{r}_n(0,\la)}-1,
\end{equation}
and show that it is small uniformly in $\la$ and $n$. To this end we substitute \eqref{Def:Delta} into \eqref{Eq:RecRel__2} and thereby derive the recurrence relation for $\delta_n$,
\begin{equation}\label{Eq:Delta_Rec}
	\delta_{n+1}(0,\la) =\ve_n(0,\la) -C_n(0,\la) \frac{\delta_n(0,\la)}{1+\delta_n(0,\la)},
\end{equation}
where
\begin{equation}\label{Eq:Eps_and_C}
	\ve_n(0,\la)=\frac{A_n (0,\la) \tilde{r}_n(0,\la)+B_n(0,\la)}{\tilde{r}_n(0,\la)\tilde{r}_{n+1}(0,\la)}-1 \quad \text{and} 
	\quad C_n(0,\la)=\frac{B_n(0,\la)}{\tilde{r}_n(0,\la)\tilde{r}_{n+1}(0,\la)}.
\end{equation}
We have the following result.
\begin{lemma}\label{Lem:Est0}
	For all $n \geq 6$ and $\la \in \Hb$ the following estimates hold
	\begin{align}\label{Eq:Est_l0}
		\begin{split}
			|\delta_6(0,\la)| \leq\tfrac{1}{5}, \quad 
			|\ve_n(0,\la)| \leq \tfrac{3}{140}+\tfrac{23}{40 n}, \quad \text{and} \quad |C_n(0,\la)| \leq\tfrac{5}{7} - \tfrac{23}{10n}.
		\end{split}
	\end{align}
\end{lemma}
Note that from \eqref{Eq:Est_l0} and \eqref{Eq:Delta_Rec}, by a simple induction we infer that $|\delta_n(0,\la)| \leq\tfrac{1}{5} $ for all $n \geq 6$. This then via \eqref{Def:Delta} and the fact that $\lim_{n \rightarrow \infty} \tilde{r}_n(0,\la)=1$ excludes \eqref{Eq:lim02}, and we are done. It therefore remains to prove the preceding lemma. 
\begin{proof}[Proof of Lemma \ref{Lem:Est0}]
	First we show that for $n \geq 6$ the functions $\delta_6(0,\cdot)$, $\ve_n(0,\cdot)$, and $C_n(0,\cdot)$ are analytic in $\Hb$. This, based on \eqref{Def:Delta} and \eqref{Eq:Eps_and_C}, follows from the fact that the zeros of $\tilde{r}_n(0,\cdot)$ and the poles of $r_6(0,\cdot)$ are all contained in the (open) left half plane. This is immediate for $\tilde{r}_n(0,\cdot)$ as it is a quadratic polynomial with two negative zeros. As for the zeros od the denominator of $r_6(0,\la)$, which is a polynomial of degree 10, this, although it can be proven by elementary means, can be straightforwardly checked by the Routh-Hurwitz stability criterion, see \cite{GlogicSchoerkhuber2021}, Sec.~A.2. Furthermore, being rational functions, $\delta_6(0,\cdot)$, $\ve_n(0,\cdot)$, and $C_n(0,\cdot)$ are all polynomially bounded in $\Hb$. Therefore, to prove the lemma, it is enough to establish the estimates \eqref{Eq:Est_l0} on the imaginary axis only as they can be then extended to all of $\Hb$ by the Phragmen-Lindel\"of principle (in its sectorial form), see e.g.~\cite{Tit58}, p.~177.
	
	In the following we prove only the third estimate in \eqref{Eq:Est_l0}, as the rest of two are shown similarly.
	We proceed with writing $C_{n+6}(0,\la)$ (note the shift in the index) as the ratio of two polynomials $P_1(n,\la)$ and $P_2(n,\la)$, both of which belong to $\mathbb{Z}[n,\la]$. Then for $t \in \R$ we have the following representation on the imaginary line  
	\begin{equation*}
		|P_j(n,it)|^2=Q_j(n,t^2),
	\end{equation*}
	for $j \in \{1,2\}$, where $Q_1(n,t^2) \in \mathbb{Z}[n,t^2]$ and $Q_2(n,t^2) \in \mathbb{N}_0[n,t^2]$. Now the desired estimate is equivalent to
	\begin{equation*}
		\frac{Q_1(n,t^2)}{Q_2(n,t^2)} \leq
		\left(\frac{5}{7} - \frac{23}{10(n+6)} \right)^2,
	\end{equation*}
	which is in turn equivalent to
	\begin{equation*}
		(50n+139)^2Q_2(n,t^2)-\big(70(n+6)\big)^2 Q_1(n,t^2) \geq 0.
	\end{equation*}
	Finally, the last inequality trivially holds as the polynomial on the left (when expanded) has manifestly positive coefficients.
\end{proof}

\emph{The case $\ell =1$.} We proceed similarly to the previous case, and we therefore only provide the relevant expressions. For $\lambda = 0$ and  $\lambda = 1$ we respectively have explicit polynomial solutions 
\begin{equation*}
	y_{1,0}(x)=1-\tfrac{3}{7} x \quad \text{and} \quad y_{1,1}(x)=1 - \tfrac{5}{77}x,
\end{equation*}
which correspond to $f_1(\cdot;0)$ and $f_1(\cdot;1)$ from the statement of the proposition. Therefore $\{0,1\} \subset \Sigma_1$, and we proceed by showing that there are no additional elements in $\Sigma_1$.
Let $\la \in \Hb \setminus \{0,1\}$. For $\ell=1$, the series \eqref{Eq:Expansion} yields a solution to Eq.~\eqref{Eq:ODEheun}, which is analytic at $x=0$. According to \eqref{Eq:RecRel}, we have that
\begin{equation}\label{Eq:RecRel1}
	a_{n+2}(1,\la)=A_n(1,\la)a_{n+1}(1, \la)+B_n(1, \la)a_{n}(1, \la),
\end{equation}
where
\begin{gather*}
	A_n (1,\la)=\frac{7 (\lambda +1) (\lambda +10)+8
		n^2+4 (7 \lambda +36) n}{14 (n+2)
		(2 n+13)},\\
	B_n(1,\la)=\frac{5 (\lambda +2 n-3) (\lambda +2
		n-2)}{14 (n+2) (2 n+13)}.
\end{gather*}
Since 
\begin{equation*}
	a_2(1,\la)=\frac{1}{8008}\la(\la-1)(7\la^2+133\la+786),
\end{equation*} 
and
\begin{equation*}
	a_3(1,\la)=\frac{1}{720720}\la(\la-1)(7\la^4+238\la^3+3263\la^2+17828\la+22476),
\end{equation*} 
we define
\[ r_2(1,\la):=\frac{a_3(1,\la)}{a_2(1,\la)}, \] 
where the common linear factors are canceled, and according to \eqref{Eq:RecRel1} we define $r_n$ for $n \geq 2$ by the recurrence 
\begin{equation*}
	r_{n+1}(1,\la)=A_n(1,\la)+\frac{B_n(1,\la)}{r_n(1, \la)}.
\end{equation*}
As a quasi-solution we let
\begin{align*}
	\tilde r_{n}(1,\lambda) := \frac{\lambda ^2}{2(2n+11)(n+1)}+\frac{  (4n+11)\la}{2(2n+11)(n+1)}+\frac{n+1}{n+4},
\end{align*}
and analogously to the previous case we define $\delta_n(1,\la)$, $\ve_n(1,\la)$, and $C_n(1,\la)$. 
Also, by the same method as above, we establish the following result.

\begin{lemma}
	For $n=5$, we have that 
	$|\delta_5(1,\la)| \leq\tfrac{1}{5}.$
	Furthermore, for every  $n\geq 5$, 
	\begin{align}\label{Eq:Estl0}
		\begin{split}
			|\ve_n(1,\la)| \leq \tfrac{3}{140}+\tfrac{5}{8(n+1)} , \quad  |C_n(1,\la)| \leq\tfrac{5}{7} - \tfrac{5}{2(n+1)}, 
		\end{split}
	\end{align}
	uniformly for all $\la\in\Hb$. Consequently, $|\delta_n(1,\la)| \leq\tfrac{1}{5}$ for all $n \geq 5$ and $\la\in\Hb$. This implies that $\lim\limits_{n\ra\infty} r_n(1, \la) = 1$.
\end{lemma}

\subsubsection{Proof of Proposition \ref{prop:SpectralODE} for $\ell \geq 2$} Since the parameter $\ell$ is now free, the analysis is more complicated. Namely, in addition to having to emulate the global behavior in $\ell$ as well, a quasi-solution also has to approximate the actual solution well enough so as to, with an additional parameter $\ell$, obey the estimates analogous to \eqref{Eq:Estl0}. We note that a similar problem was treated by the second and the third author in \cite{GlogicSchoerkhuber2021}, Secs.~4.2.1 and 4.2.2, and we closely follow their approach.  First, we introduce the following change of variable 
$
x = \tfrac{12\rho^2}{5\rho^2+7},
$
by means of which the singular points $\rho=0$ and $\rho=1$ remain fixed, while the remaining finite singularity (which corresponds to $\rho=\infty$) is now further away from the unit disk, at  $x=\frac{12}{5}$.
Furthermore, by applying also the following transformation 
\begin{equation*}
	f(\rho) = x^{\frac{\ell}{2}} \big(\tfrac{12}{5}-x\big)^{\frac{1}{2}(\lambda+3)} \tilde y(x)
\end{equation*}
to $\mathcal{T}_\ell(\lambda)f(\rho)=0$ we arrive at a Heun equation for $\tilde{y}$
\begin{align}\label{Eq:ODEheun_2} 
	\tilde y''(x)+\left(\frac{\tilde \gamma(\ell)}{x}+\frac{\tilde \delta(\lambda)}{x-1}+\frac{\epsilon}{x-\tilde \mu}\right)\tilde{y}'(x)+\frac{\tilde\alpha(\ell,\la)\tilde \beta(\ell,\lambda) x - \tilde q(\ell,\la) }{x(x-1)(x-\mu)}\tilde{y}(x)=0,
\end{align}
where $\tilde \mu = \frac{12}{5}$, $\tilde \gamma(\ell) = \frac{9+2\ell}{2}$, $\tilde \delta(\lambda) =\lambda -1$, $\epsilon = 3/2$, $\tilde \alpha(\lambda) = \frac{1}{2} (\lambda -3+\ell)$,  $\tilde \beta(\lambda) = \frac{1}{2} (\lambda +11+\ell)$, and 
\[ \tilde q(\ell,\lambda) = \frac{1}{20}\big(17 \ell^2+2\ell(55+12\lambda)-7\la^2+80\la-303\big). \]
The Frobenius indices of Eq.~\eqref{Eq:ODEheun_2} at $x=0$ are $s_1=0$ and $s_2=-\frac{7+2\ell}{2}$. Therefore,  we consider the (normalized) analytic solution at $x=0$ 
\begin{equation}\label{Eq:Expansion_2}
	\tilde y(x)=\sum_{n=0}^{\infty}\tilde a_n(\ell, \la) x^n, \quad \tilde a_0(\ell,\la)=1.
\end{equation} 
Inserting \eqref{Eq:Expansion_2} into Eq.~\eqref{Eq:ODEheun_2} yields
\begin{equation}\label{Eq:RecRel_2}
	\tilde  a_{n+2}(\ell,\lambda)=\tilde  A_n(\ell,\la)\tilde  a_{n+1}(\ell,\la)+\tilde  B_n(\ell,\la)\tilde  a_{n}(\ell,\la),
\end{equation}
with
\begin{gather*}
	\tilde  A_n(\ell,\la)=\frac{68 n^2+(48 \lambda +68
		\ell+356)n+ 7 \lambda ^2+17 \ell^2 +24
		\lambda  \ell +128 \lambda  +178 \ell-15}{24(n+2)(2n+2\ell+11)}
\end{gather*}
and
\begin{gather*}
	\tilde  B_n(\ell,\la)=\frac{-5(2n+\la+\ell+11)(2n+\la+\ell-3)}{24(n+2)(2n+2\ell+11)},
\end{gather*}
supplied with the initial condition
$	\tilde  a_{-1}(\ell, \la)=0$, $	\tilde  a_0(\ell,\la)= 1$. Now, $\lim_{n\ra\infty}	\tilde  A_n(\ell,\la)=\frac{17}{12}$ and $\lim_{n\ra\infty}	\tilde  B_n(\ell,\la)=-\frac{5}{12}$, and consequently the characteristic equation of Eq.~\eqref{Eq:RecRel} is 
$
t^2-\tfrac{17}{12}t+\tfrac{5}{12}=0,
$
with solutions $t_1=\frac{5}{12}$ and $t_2=1$. Hence, for 
$$
\hat r_n(\ell,\la) := \frac{\tilde  a_{n+1}(\ell,\la)}{\tilde  a_n(\ell,\la)},
$$
either $	\tilde  a_n(\ell,\la)=0$ eventually in $n$ or 
\begin{equation}\label{Eq:Lim1}
	\lim\limits_{n\ra\infty} \hat r_n(\ell,\la) =1
\end{equation}
or
\begin{equation}\label{Eq:Lim2}
	\lim\limits_{n\ra\infty} \hat r_n(\ell,\la) =   \frac{5}{12}.
\end{equation}
Now, for $\la \in \Hb$, similarly to the previous cases, we exclude the first option by backward substitution. Then, from \eqref{Eq:RecRel_2} we derive the recurrence relation for $\hat{r}_n$
\begin{equation}\label{Eq:RecRel2}
	\hat  r_{n+1}(\ell,\la)=\tilde{A}_n(\ell,\la)+\frac{\tilde{B}_n(\ell,\la)}{\hat{r}_n(\ell,\la)},
\end{equation}
along with the initial condition $r_0(\ell,\la)=A_{-1}(\ell,\la)$.
For a quasi-solution to \eqref{Eq:RecRel2} we use
\begin{equation*}
	{R}_n(\ell,\la):=\frac{7 \lambda ^2}{24 (n+1) (2
		n+2\ell+9)}+\frac{\lambda  (6 n+3 \ell+10)}{3
		(n+1) (2n + 2\ell + 9)}+\frac{17 \ell}{48
		(n+1)}+\frac{n-1}{n+1}.
\end{equation*}
Again, for the exact way of constructing such quasi-solutions we refer the reader to \cite{GlogicSchoerkhuber2021}, Sec.~4.2.2 or \cite{Glo18}, Sec.~2.7.2. Thereupon we set 
\begin{equation}\label{Def:Delta1} 
	\tilde \delta_n(\ell,\la):=\frac{\hat r_n(\ell,\la)}{{R}_n(\ell,\la)}-1,
\end{equation}
to obtain 
\begin{equation*}
	\tilde \delta_{n+1}(\ell,\la) =\tilde \ve_n(\ell,\la) - \tilde C_n(\ell,\la) \frac{\tilde \delta_n(\ell,\la)}{1+ \tilde \delta_n(\ell,\la)},
\end{equation*}
where
\begin{equation*}
	\tilde \ve_n(\ell,\la)=\frac{\tilde A_n (\ell,\la) R_n(\ell,\la)+\tilde B_n(\ell,\la)}{R_n(\ell,\la)R_{n+1}(\ell,\la)}-1 \quad \text{and} 
	\quad \tilde C_n(\ell,\la)=\frac{\tilde B_n(\ell,\la)}{R_n(\ell,\la)R_{n+1}(\ell,\la)}.
\end{equation*}
Now, similarly to the previous cases, we establish the following lemma.
\begin{lemma}
	For all $\ell \geq 2$, $n \geq 3$, and $\la \in 
	\Hb$, the following estimates hold 
	\begin{align*}
		\begin{split}
			|\tilde{\delta}_{3}(\ell,\la)|\leq \tfrac{1}{3}, \quad
			| \tilde \ve_n(\ell ,\la)| \leq \tfrac{1}{8}  , \quad  |\tilde C_n(\ell,\la)| \leq\tfrac{5}{12}. 
		\end{split}
	\end{align*}
	As a consequence, $|\tilde \delta_n(\ell,\la)| \leq\tfrac{1}{3}$ for all $n \geq 3$. 
\end{lemma}
From this lemma, Eq.~\eqref{Def:Delta1}, and the fact that $\lim\limits_{n\ra\infty} R_n(\ell, \la) = 1$, we exclude \eqref{Eq:Lim2} and we therefore have that $\lim\limits_{n\ra\infty} \tilde r_n(\ell, \la) = 1$. Hence, given $\la \in \Hb$, there are no solutions to Eq.~\eqref{Eq:ODEheun_2} which are analytic on $[0,1]$, and consequently $\Sigma_{\ell} = \varnothing$.

\subsection{The spectrum of $\mb L_0$}
 With the results from above at hand, we can provide a complete description of the unstable spectrum of $\mb L_0$.

\begin{proposition} \label{spectrumofL0}
	There exists $\o_0 \in (0,\frac{1}{2}]$, such that
	\begin{align}\label{Eq:Spec_L_a}
		\sigma(\mb L_0) \cap \{\lambda \in\mathbb{C} : \Re \lambda > -\omega_0 \} = \{\lambda_0, \lambda_1, \lambda_2 \},
	\end{align}
	where $\l_0=0$, $\l_1=1$ and $\l_2=3$ are eigenvalues. The geometric eigenspace of $\l_2$ is spanned by $\mb h_0=(h_{0,1},h_{0,2})$, where
	\begin{align} \label{two}
		h_{0,1}(\xi)=\frac{1}{(7+5|\xi|^2)^3}, \quad h_{0,2}(\xi)=\xi^i \partial_{i} h_{0,1}(\xi)+5 h_{0,1}(\xi).
	\end{align}
	Moreover, the geometric eigenspaces of $\l_1$ and $\l_0$ are spanned by  $\{\mb g_0^{(k)}\}_{k=0}^9 =  \{(g_{0,1}^{(k)},g_{0,2}^{(k)}) \}_{k=0}^9$, and $\{\mb q_0^{(j)}\}_{j=1}^9= \{(q_{0,1}^{(j)},q_{0,2}^{(j)} \}_{j=1}^9$, respectively, where we have in closed form
	\begin{align}
		g_{0,1}^{(0)}(\xi)=\frac{1-|\xi|^2}{(7+5|\xi|^2)^3}, \quad g_{0,2}^{(0)}(\xi)=\xi^i \partial_{i} g_{0,1}^{(0)}(\xi)+3g_{0,1}^{(0)}(\xi), \label{zero1}
		\\
		g_{0,1}^{(j)}(\xi)=\frac{\xi^j(77-5|\xi|^2)}{(7+5|\xi|^2)^3} \quad g_{0,2}^{(j)}(\xi)=\xi^i \partial_{i}g_{0,1}^{(j)}(\xi)+3g_{0,1}^{(j)}(\xi),\label{one}
	\end{align}
	for $j=1,\dots, 9$ as well as 
	\begin{align} \label{zero}
		q^{(j)}_{0,1}(\xi)=\frac{\xi^j(7-3|\xi|^2)}{(7+5|\xi|^2)^3}, \quad q_{0,2}^{(j)}(\xi)=\xi^i \partial_{i} q_{0,1}^{(j)}(\xi)+2q_{0,1}^{(j)}(\xi).
	\end{align}
\end{proposition}

\begin{remark}\label{Rem:EF_0}
	Recall that $\mb U_a$ solves the stationary equation $\mb L \mb U_a + \mb F(\mb U_a) = 0$. By the chain rule we get for any $k = 1,\dots, d$, that
	$(\mb L+ \mb F'(\mb U_a) )\partial_{a^k}\mb U_a= \mb L_a\partial_{a^k}\mb U_a =0$. This implies that $\partial_{a^k}\mb U_a$ is an eigenvector of $\mb L_a$ with eigenvalue $\l=0$.
	In particular, a direct calculation shows that  $q^{(j)}_{0,1}(\xi) = c \partial_{a^j} U_a(\xi)|_{a = 0}$.
\end{remark}

\begin{proof} From Propositions \ref{Prop:Structure_Spectrum}, \ref{Prop:Spectral_ODE} and \ref{prop:SpectralODE} we deduce the existence of $\omega_0  \in (0,\frac{1}{2}]$ for which \eqref{Eq:Spec_L_a} holds. To determine the eigenspaces, we do the following. First, in view of Proposition~\ref{prop:SpectralODE} if $\l=3$ then $\ell=0$, and setting $u_{0,1}(\rho)=(7+5\rho^2)^{-3}$ in the expansion \eqref{Eq:ExpSph} yields \eqref{two}. If $\l=1$, then either $\ell=0$ and $u_{0,m}=f_0(\cdot;1)$, or $\ell=1$ and $u_{1,m}=f_1(\cdot;1)$, for $m=1,\dots,9$. Since we can choose $Y_{1,m}(\o)=\tilde{c}_m\o_m$ for $m=1, \dots 9$, these yield \eqref{zero1} and \eqref{one}. For $\l=0$, we have $\ell=1$ with $u_{1,m}=f_1(\cdot,0)$, which similarly leads to \eqref{zero}.
\end{proof}

In what follows, we prove that for each unstable eigenvalue the geometric and the algebraic eigenspaces are the same. To this end, we define the associated Riesz projections. Namely, we set
\[
\mb H_0:=\frac{1}{2\pi i}\int_{\gamma_2} \mb R_{\mb L_0}(\lambda)d\lambda, \quad \mb P_0:=\frac{1}{2\pi i}\int_{\gamma_1}\mb R_{\mb L_0}(\lambda)d\lambda, \quad \mb Q_0:=\frac{1}{2\pi i}\int_{\gamma_0}\mb R_{\mb L_0}(\lambda)d\lambda,
\]
where $\gamma_j(s)=\lambda_j+\frac{\omega_0}{2}e^{2i\pi s}$ for $s\in[0,1]$, and $j=0,1,2$.

\begin{lemma} \label{projections}
	We have that
	\begin{align*}
		\dim \ran \mb H_0=1, \quad \dim\ran  \mb P_0=10, \quad  \dim\ran  \mb Q_0=9.
	\end{align*}
\end{lemma}

\begin{proof}
	We start with the observation that the ranges of the projections are finite-dimensional. Indeed, $\l_j$ would otherwise belong to the essential spectrum of $\mb L_0$ (see \cite{Kato} Theorem~5.28 and Theorem~5.33) which coincides with the essential spectrum of $\mb L$ (since $\mb L_0$ is a compact perturbation of $ \mb L$), but this is in contradiction with \eqref{Eq:Spec_L_k}.
	Now we show that $\dim\ran \mb P_0=10$. We know from properties of the Riesz integral that $\ker (\mb L_0-\l_1) \subset \ran \mb P_0$. We therefore only need to prove the reversed inclusion. First, note that the space $\ran \mb P_0$ reduces the operator $\mb L_0$, and we have that  
	\[
	\sigma(\mb L_0|_{\ran\mb P_0})=\{1\},
	\]
	see e.g.~\cite{Hislop}, Proposition 6.9.
	Consequently, since $\mb P_0$ is finite-rank, the operator
	$
	1-\mb L_0|_{\ran \mb P_0}
	$
	is nilpotent, i.e., there is $m\in\N$, such that
	$
	(1-\mb L_0|_{\ran \mb P_0})^m=0.
	$
 Note that it suffices to show that $m=1$. We argue by contradiction, and hence assume that $m \geq 2$. Then there is $\mb u  \in \mc D(\mb L_0)$ such that
	\[
	(1-\mb L_0)\mb u=\mb v,
	\]
	for a nontrivial $\mb v\in \ker(1-\mb L_0)$. This yields for $u_1$ the following elliptic equation
	\begin{equation} \label{58}
		-(\delta^{ij}-\xi^i\xi^j)\partial_{i}\partial_{j}u_1(\xi)+2(\l+3)\xi^j\partial_{j}u_1(\xi)+(\l+3)(\l+2)u_1(\xi)-V_0(\xi)u_1(\xi)=F(\xi),
	\end{equation}
	where $\l=1$ and
	\[
	F(\xi)=\xi^i\partial_{i}v_1(\xi)+(\l+3)v_1(\xi)+v_2(\xi).
	\]
	Since $\mb v\in\ker(1-\mb L_0)=\text{span}(\mb g_0^{(0)},\dots,\mb g_0^{(9)})$, we have that $\mb v=\sum_{k=0}^9\alpha_k \mb g_0^{(k)}$ for some $\alpha_0,\dots,\alpha_9\in\C$, not all of which are zero. To avoid cumbersome notation we let $g_k= g^{(k)}_{0,1}$. In the new notation, based on \eqref{zero1} and \eqref{one} we have that
	\begin{align*}
		F(\xi)=\sum_{k=0}^9\alpha_k(2\xi^i\partial_{\xi^i}g_k+7g_k).
	\end{align*}
	Furthermore, according to  Proposition~\ref{prop:SpectralODE} we can rewrite $F$ in polar coordinates
	\begin{align*}
		F(\rho\o)=\a_0(2\rho f'_0(\rho)+7f_0(\rho))Y_{0,1}(\o)+\sum_{i=1}^9\a_i(2\rho f'_1(\rho)+7f_1(\rho))Y_{1,i}(\o),
	\end{align*}
	where we denoted $f_0=f_0(\cdot ;1)$ and $f_1=f_1(\cdot ; 1)$.
	By decomposition of $u_1$ into spherical harmonics as in \eqref{Eq:ExpSph}, equation \eqref{58} can be written as a system of ODEs
	\begin{align} \label{5.9}
		\mathcal{T}_0(1)u_{0,1}=-\a_0G_0, \quad \mathcal{T}_1(1)u_{1,j}=-\a_jG_1, \quad j=1,\dots,9,
	\end{align}
	posed on the interval $(0,1)$, where $G_i(\rho)=2\rho f'_i(\rho)+7f_i(\rho)$ for $i=0,1$. Moreover, from the properties of $u_1$, we infer that $u_{\ell,m}\in C^\infty[0,1)\cap H^5(\frac{1}{2},1)$, and by Sobolev embedding we have that $u_{\ell,m}\in C^2[0,1]$. To obtain a contradiction, we show that if some $\alpha_k$ is non-zero then the corresponding ODE in \eqref{5.9} does not admit a $C^2[0,1]$ solution. To start, we assume that $\a_0 
	\neq 0$. For convenience, we can without loss of generality assume that $\a_0=-1$. Then $u_{0,1}$ solves the following ODE
	\begin{align}\label{eigenvalue-one}
		(1-\rho^2)u''(\rho)+\left(\frac{8}{\rho}-8\rho \right)u'(\rho)-\big(12-V(\rho) \big)u(\rho)= G_0(\rho),
	\end{align}
	where
	\begin{align*}
		G_0(\rho)=\frac{5\rho^4-102\rho^2+49}{(7+5\rho^2)^4}.
	\end{align*}
	 Note that 
	\begin{align*}
		u_1(\rho)=f_0(\rho)=\frac{1-\rho^2}{(7+5\rho^2)^3}
	\end{align*}
	is a solution to the homogeneous version of Eq.~\eqref{eigenvalue-one}, and by reduction of order  we find a second one 
	\begin{align*}
		u_2(\rho)=u_1(\rho)\int_{\frac{1}{2}}^\rho \frac{ds}{s^8u_1(s)^2}=\frac{1-\rho^2}{(7+5\rho^2)^3} 
		\int_{\frac{1}{2}}^\rho \frac{(7+5s^2)^6}{s^8(1-s^2)^2} ds.
	\end{align*}
	Furthermore, simple calculation yields that
	\begin{align*}
		u_2(\rho)\simeq \rho^{-7} \quad \text{as} \quad \rho\rightarrow 0^+,
	\end{align*}
	and 
	\begin{align}\label{Eq:u_2_asymp}
		u_2(\rho)=864-3456(1-\rho)\ln(1-\rho)+O(1-\rho)\quad  \text{as} \quad  \rho \rightarrow 1^-.
	\end{align}
	With the fundamental system $\{u_1,u_2\}$ at hand, we can solve Eq.~\eqref{eigenvalue-one} by the variation of parameters formula. Namely, we have that
	\begin{align*}
		u(\rho)=c_1u_1(\rho)+c_2u_2(\rho)-u_1(\rho)\int_0^\rho \frac{u_2(s)G_0(s)s^8}{1-s^2}ds+u_2(\rho)\int_0^\rho\frac{u_1(s)G_0(s)s^8}{1-s^2}ds
	\end{align*}
	for some constants $c_1,c_2\in\C$. If $u\in C^2[0,1]$, then $c_2$ must be equal to zero in the above expression, owing to the singular behavior of $u_2(\rho)$ near $\rho=0$. Then by differentiation we obtain for $\rho\in(0,1)$ that
	\begin{align*}
		u'(\rho)=c_1u'_1(\rho)-u'_1(\rho)\int_0^\rho \frac{u_2(s)G_0(s)s^8}{1-s^2}ds+u'_2(\rho)\int_0^\rho\frac{u_1(s)G_0(s)s^8}{1-s^2}ds.
	\end{align*}
	Now we inspect the asymptotic behavior of $u'$ as $\rho\rightarrow 1^-$. We first note that $u'_1$ is bounded near $\rho=1$. Furthermore, note that 
	\begin{align*}
		\int_0^1 \frac{u_1(s)G_0(s)s^8}{1-s^2}ds = \int_0^1  \frac{s^2}{1-s^2} 
		\frac{d}{ds} \left [ \frac{s^7 ( 1 - s^2)^2}{(7+5 s^2)^6} \right ] ds 
		= - 2 \int_0^1 \frac{s^8 ( 1-s^2)}{(7+5s^2)^6} =:  - C 
	\end{align*}
	for some $C > 0$, which can be calculated explicitly and $C < 4 \times 10^{-8}$.
	Hence, based on \eqref{Eq:u_2_asymp} we have that
	\begin{align*}
		u'_2(\rho)\int_0^\rho\frac{u_1(s)G_0(s)s^8}{1-s^2}ds \sim - 3456 ~C \ln (1-\rho)\quad \text{as} \quad \rho \rightarrow 1^-.
	\end{align*}
	Moreover
	\begin{align*}
		- u'_1(\rho)\int_0^\rho \frac{u_2(s)G_0(s)s^8}{1-s^2} ds \sim \frac{1}{864}\ln(1-\rho) \quad \text{as} \quad \rho \rightarrow 1^-.
	\end{align*}
	Finally, we infer that the two integral terms cannot cancel and thus
	\begin{align*}
		u'(\rho)\simeq  \ln (1-\rho) \quad \text{as} \quad \rho \rightarrow 1^-.
	\end{align*}
	In conclusion, there is no choice of $c_1,c_2$ for which $u$ belongs to $C^2[0,1]$. \\
	
	We similarly treat $\alpha_j$ for $j\in\{1,\dots, 9\}$. It is enough to consider just $\alpha_1$, and without loss of generality assume that $\a_1=-1$. Then \eqref{5.9} yields the following ODE
	\begin{align} \label{5.11}
		(1-\rho^2)u''(\rho)+\left(\frac{8}{\rho}-8\rho \right)u'(\rho)-\left( 12+\frac{8}{\rho^2} -V(\rho) \right)u(\rho)=G_1(\rho),
	\end{align} 
	where
	\begin{align*}
		G_1(\rho)=\frac{\rho(4851-1610 \rho^2 - 25 \rho^4)}{(7+5\rho^2)^4}.
	\end{align*}
	Note that 
	\[u_1(\rho)=f_1(\rho)=\frac{\rho(77-5\rho^2)}{(7+5\rho^2)^3}\]
	is a solution for the homogeneous problem. Similarly as above, we obtain another solution by the reduction formula
	\begin{align*}
		u_2(\rho)=u_1(\rho)\int_{1}^\rho \frac{ds}{s^8u_1(s)^2}=\frac{\rho(77-5\rho)}{(7+5\rho^2)^3}\int^\rho_{1}\frac{(7+5s^2)^6}{s^{10}(77-5s)^2}ds,
	\end{align*}
	and by inspection of the integral we get that $u_2(\rho)\simeq \rho^{-8}$ near the origin and $u_2(\rho)\simeq 1 - \rho$ near $\rho=1$.  Now, the general solution of \eqref{5.11} on $(0,1)$ is given by
	\begin{align}\label{Eq:u_var_param}
		u(\rho)=c_1 u_1(\rho) +c_2u_2(\rho)-u_1(\rho)\int_0^\rho \frac{u_2(s)G_1(s)s^8}{1-s^2}ds
		+u_2(\rho)\int_0^\rho \frac{u_1(s)G_1(s)s^8}{1-s^2}ds.
	\end{align}
	Assumption that $u$ belongs to $C^2[0,1]$ forces $c_2=0$ above, due to the singular behavior of $u_2$ at $\rho=0$. Furthermore, from the last term in \eqref{Eq:u_var_param} we see that $u'(\rho)\simeq \ln(1-\rho)$ as $\rho \rightarrow 1^-$. In conclusion, Eq.~\eqref{5.11} admits no $C^2[0,1]$ solutions, and this finishes the proof for $\mb P_0$. \\

	The remaining two projections are treated similarly, so we omit some details. For $\mb H_0$ we obtain the analogue of \eqref{58} with $F(\xi)=2\xi^i\partial_i h_{0,1}(\xi)+11h_{0,1}(\xi)$. This leads to the following ODE
	\begin{align} \label{5.12}
		\left( 1-\rho^2\right) u''(\rho)+\left( \frac{8}{\rho} - 12\rho\right)u'(\rho) - \left(30-V(\rho) \right)u(\rho)=H(\rho),
	\end{align}
	for 
	$$
	H(\rho)=\frac{77-5\rho^2}{(7+5\rho^2)^4}.
	$$
	The argument, as above, reduces to showing that Eq.~\eqref{5.12} does not admit $C^2[0,1]$ solutions. By Proposition~\ref{prop:SpectralODE}, we have that $u_1(\rho)=(7+5\rho^2)^{-3}$ solves the homogeneous variant of Eq.~\eqref{5.12}, with the reduction formula yielding another solution
	\begin{align}\label{Eq:Gen_sol_H}
		u_2(\rho)=u_1(\rho)\int_{\frac{1}{2}}^\rho \frac{ds}{s^8(1-s^2)^2 u_1(s)^2}=\frac{1}{(7+5\rho^2)^3}
		\int_{\frac{1}{2}}^\rho \frac{(7+5s^2)^6}{s^8(1-s^2)^2}ds.
	\end{align}
	Note that $u_2$ is singular at both $\rho=0$ and $\rho=1$; more precisely $u_2(\rho)\simeq\rho^{-7}$ as $\rho\rightarrow 0^+$, and  $u_2(\rho)\simeq (1-\rho)^{-1}$ as $\rho\rightarrow 1^-$. With $u_1$ and $u_2$ at hand, the general solution of \eqref{5.12} on the interval $(0,1)$ can be written as
	\begin{align*}
		u(\rho)=c_1u_1(\rho)+c_2u_2(\rho)-u_1(\rho)&\int_0^\rho(1-s^2) s^8H(s)u_2(s)ds
		\\
		&+u_2(\rho)\int_0^\rho(1-s^2) s^8H(s)u_1(s)ds,
	\end{align*}
	where the parameters $c_1,c_2\in\C$ are free. Assumption that $u$ is bounded near $\rho=0$ forces $c_2=0$. Note that the first and the third term in \eqref{Eq:Gen_sol_H} are bounded near $\rho =1$. However, due to the singular behavior of $u_2$, the last term is unbounded near $\rho=1$, owing to the integrand being strictly positive on $(0,1)$. In conclusion, the general solution $u$ in \eqref{Eq:Gen_sol_H} is unbounded on $(0,1)$.\\
	
	Finally, for $\mb Q_0$, we have that
	\begin{align*}
		F(\xi)=\sum_{j=1}^9 \a_j\left(2\xi^i\partial_{\xi^i}q^j_{0,1}(\xi)+5q^j_{0,1}(\xi) \right),
	\end{align*}
	and the accompanying analogue of \eqref{eigenvalue-one} is 
	\begin{align*} 
		\left(1-\rho^2\right)u''(\rho)+\left(\frac{8}{\rho}-6\rho \right)u'(\rho)-\left(6+\frac{8}{\rho^2}-V(\rho)\right)=Q(\rho),
	\end{align*}
	where 
	$$
	Q(\rho)=\frac{15\rho^5-406\rho^3+343\rho}{(7+5\rho^2)^4}.
	$$
	A fundamental solution set to the homogeneous version of the above ODE is given by
	\begin{equation*}
		u_1(\rho)=\frac{\rho(7-3\rho^2)}{(7+5\rho^2)^3} \quad \text{and} \quad 	u_2(\rho)=u_1(\rho)\int^\rho_1\frac{1-s^2}{s^8u_1(s)^2}ds,
	\end{equation*}
	and therefore any solution to it on $(0,1)$ can be written as
	\begin{align*}
	u(\rho)=c_1u_1(\rho)+c_2u_2(\rho)-u_1(\rho)\int_0^\rho\frac{u_2(s)Q(s)s^8}{(1-s^2)^2}ds+u_2(\rho)\int_0^\rho\frac{u_1(s)Q(s)s^8}{(1-s^2)^2}ds,
	\end{align*}
	for a choice of $c_1,c_2\in\C$. Again, by similar asymptotic considerations as above, we infer that $u''$ is necessarily unbounded on $(0,1)$, and this concludes the proof.
\end{proof}

\subsection{The spectrum of $\mb L_a$ for $a\neq 0$}\label{Sec:L_a}

We now investigate the spectrum of $\mb L_a$. In particular, by a perturbative argument we show that, for small $a$, an analogue of Proposition \ref{spectrumofL0} holds for $\mb L_a$ as well.
\begin{lemma}\label{Le:Pert_Spectrum}
	There exists $\delta^* >0$ such that for all $a\in\overline{\mathbb{B}^9_{\d^*}}$ the following holds.
	\[
	\sigma(\mb L_a)\cap \left\{ \l\in\C : \Re \l \geq -\frac{\o_0}{2}\right\} = \left\{\l_0,\l_1,\l_2 \right\},
	\]
	where $\o_0$ is the constant from Proposition \ref{spectrumofL0}, and $\l_0=0$, $\l_1=1$, $\l_2=3$ are eigenvalues.
	The geometric eigenspace of $\l_2$ is spanned by $\mb h_a=(h_{a,1},h_{a,2})$, where
	\[
	h_{a,1}(\xi)=\frac{\g(\xi,a)}{\left (12\g(\xi,a)^2+5 |\xi|^2-5 \right )^3} , \quad h_{a,2}(\xi)=\xi^j\partial_j h_{a,1}(\xi)+5h_{a,1}(\xi).
	\]
	Moreover, the geometric eigenspaces of $\l_0$ and $\l_1$ are spanned by $\{\mb g_a^{(k)}\}_{k=0}^9 =  \{(g_{a,1}^{(k)},g_{a,2}^{(k)} )\}_{k=0}^9$, and $\{\mb q_a^{(j)}\}_{j=1}^9= \{(q_{a,1}^{(j)},q_{a,2}^{(j)}) \}_{j=1}^9$ respectively, where 
	\begin{align*}
		g^{(0)}_{a,1}(\xi)&=\frac{(|\xi|^2-1)\g(\xi,a)}{\left (12\g(\xi,a)^2+5|\xi|^2-5 \right )^3} ,
		\quad
		g^{(0)}_{a,2}(\xi)=\xi^j\partial_{\xi^j}g^{(0)}_{a,1}(\xi)+3g^{(0)}_{a,1}(\xi),
		\\
		g^{(k)}_{a,1}(\xi)&= \frac{(72\g(\xi,a)^2+5 -5 |\xi|^2)  \partial_{a_j} \gamma(\xi,a) }{\left (12\g(\xi,a)^2+5|\xi|^2-5 \right )^3}, 
		\quad
		g^{(k)}_{a,2}(\xi)=\xi^j\partial_{\xi^j}g^{(k)}_{a,1}(\xi)+3g^{(k)}_{a,1}(\xi),
	\end{align*}
	and
	\begin{align*}
		q^{(j)}_{a,1}(\xi)&= \partial_{a_j} U_a(\xi) , \quad q^{(j)}_{a,2}(\xi)=\xi^j\partial_{j}q^{(j)}_{a,1}(\xi)+2q^{(j)}_{a,1}(\xi).
	\end{align*}
	Additionally, the eigenfunctions depend Lipschitz continuously on the parameter $a$, i.e.,
	\[
	\|\mb h_a-\mb h_b\| + \|\mb g_a^{(k)}-\mb g_b^{(k)}\| + \| \mb q_a^{(j)}-\mb q_b^{(j)}\|\lesssim |a-b|,
	\]
	for all $a,b\in \overline{\mathbb{B}^9_{\d^*}}$.
\end{lemma}

\begin{proof}
	Let $\varepsilon=-\frac{\o_0}{2}+\frac{1}{2}$ and $\delta > 0$.  Then take $\kappa$ defined by Proposition \ref{Prop:Structure_Spectrum}, and introduce the following two sets
	\begin{align*}
		\Omega = \{ z \in \C : \Re z \geq -\frac{\o_0}{2} \text{ and } |z|\leq  \kappa\} \quad \text{and} \quad \tilde{\Omega}=\{z\in\C : \Re z\geq -\frac{\o_0}{2}\}\setminus\Omega.
	\end{align*}
	Note that Proposition~\ref{Prop:Structure_Spectrum} implies that $\tilde{\Omega}\subset \rho(\mb L_a)$ for all $a \in \overline{\B_\d}$. Hence, we only need to investigate the spectrum in the compact set $\Omega$. First, note that by Proposition \ref{Prop:Structure_Spectrum}, the set $\Omega$ contains a finite number of eigenvalues. By a direct calculation it can be checked that $\mb q_a^{(j)}$, $
	\mb g_a^{(k)}$, and $\mb h_a$ are eigefunctions that correspond to $\la_0=0$, $\la_1=1$, and $\la_2=3$ respectively. Note that we get the explicit expression above by just Lorentz transforming the corresponding eigenfunctions for $a=0$. We now show  that there are no other eigenvalues in $\Omega$. For this, we utilize the Riesz projection onto the spectrum contained in $\Omega$, see \eqref{Eq:Riesz_a} below. This, however, necessitates that $\partial\Omega\subset\rho(\mb L_a)$, and we now show that this holds for small enough $a$. First, note that for $\la \in \partial \Omega$ we have the following identity 
	\begin{align} \label{5.16}
		\l-\mb L_a=\big[ 1-(\mb L'_a-\mb L'_0)\mb R_{\mb L_0}(\l) \big](\l-\mb L_0).
	\end{align}
	Then, from Proposition \ref{group-Sa}, we have that 
	\[
	 \nr{\mb L'_a-\mb L'_0}\nr{\mb R_{\mb L_0}(\lambda)} \lesssim |a| \max_{\lambda \in \partial \Omega} \nr{\mb R_{\mb L_0}(\lambda)} 
	\]
	for all $a \in \overline{\B_\d}$. Therefore, there is small enough $\delta^* >0$ such that
	\begin{align}\label{Eq:Resolvent_Est_a_0}
		\nr{\mb L'_a-\mb L'_0}\nr{\mb R_{\mb L_0}(\lambda)}<1,
	\end{align}
	for all $\lambda \in \partial \Omega$, and all $a\in\overline{\B_{\d^*}}$. Now from \eqref{Eq:Resolvent_Est_a_0} and \eqref{5.16} we infer that $\partial\Omega\subset\rho(\mb L_a)$ for all $a\in\overline{\B_{\d^*}}$.
	Thereupon we define the projection
	\begin{align}\label{Eq:Riesz_a}
		\tilde {\mb T}_a =\frac{1}{2\pi i}\int_{\partial \Omega}\mb R_{\mb L_a}(\l)d\l.
	\end{align}
For $a=0$, by Lemma~\ref{projections} the rank of the operator $ \tilde {\mb T}_a$ is $20$.
	 Furthermore, continuity of $a \mapsto \mb R_{\mb L_a}(\l)$ (which follows from Eq.~\eqref{5.16}) implies continuity of $a \mapsto \tilde {\mb T}_a$ on $\overline{\B_{\d^*}}$.  Thus, we conclude that $\dim \ran \tilde {\mb T}_a = 20$ for all $a \in \overline{\B_{\d^*}}$, see e.g.~\cite{Kato}, p.~34, Lemma~4.10. By this, we exclude and further eigenvalues. Lipschitz continuity for the eigenfunctions follows from the fact that they depend smoothly on $a$, c.f.~ \eqref{Eq:Lipschitzbounds_Ua} and \eqref{Eq:Selfsim_Sol_Lipschitz}.
\end{proof}

\section{Perturbations around $\mb U_a$ - Bounds for the linearized time-evolution}

We fix $\d^*>0$ as in Lemma \ref{Le:Pert_Spectrum} for the rest of this paper. In this section we propagate Lemma \ref{projections} to $\mb L_a$. 
 For that, given $a \in \overline{\B_{\delta^*}}$  we define the Riesz projections 
\[
\mb H_a:=\frac{1}{2\pi i}\int_{\g_2}\mb R_{\mb L_a}(\l)d\l, \quad \mb P_a:=\frac{1}{2\pi i}\int_{\g_1}\mb R_{\mb L_a}(\l)d\l, \quad \mb Q_a:=\frac{1}{2\pi i}\int_{\g_0}\mb R_{\mb L_a}(\l)d\l,
\]
where $\g_j(s)=\l_j+\frac{\o_0}{4}e^{2\pi is}$ for $s\in[0,1]$. 

\begin{lemma} \label{Rieszprojections2}
	We have that that
	\[
	\ran \mb H_a=\mathrm{span}\,(\mb h_a), \quad \ran\mb P_a=\mathrm{span}\,(\mb g_a^{(0)},\dots,\mb g_a^{(9)}), \quad \ran \mb Q_a=\mathrm{span} \, (\mb q_a^{(1)},\dots,\mb q_a^{(9)}),
	\]
	for all $a \in \overline{\B_{\delta^*}}$.
	 Moreover, the projections are mutually transversal,
	\[
	\mb H_a \mb  P_a=\mb P_a \mb H_a=\mb H_a \mb Q_a=\mb Q_a\mb H_a=\mb Q_a \mb  P_a= \mb P_a \mb Q_a=0,
	\]
	and depend Lipschitz continuously on the parameter $a$, i.e.,
	\[
	\nr{\mb H_a-\mb H_b}+\nr{\mb P_a-\mb P_b}+\nr{\mb Q_a-\mb Q_b}\lesssim |a-b|
	\]
	for all $a,b\in\overline{\mathbb{B}^9_\ds}$.
\end{lemma}

\begin{proof}
	 The Riesz projections depend continuously on $a$, hence the dimensions of the ranges remain the same. Transversality follows from the definition of Riesz projections. The Lipschitz bounds follow from the second resolvent identity and Proposition~\ref{group-Sa}. 
\end{proof}

Since $\mb P_a$ and $\mb Q_a$ are finite-rank, for every $\mb f\in\HH$ there are $\a^k\in\C$ and $\b^j\in\C$, such that
\begin{align*}
	\mb P_a \mb f=\sum_{k=0}^9\a^k \mb g_a^{(k)}, \text{ and } \mb Q_a \mb f=\sum_{j=1}^9\b^j \mb q_a^{(j)}.
\end{align*}
We thereby define the projections
\begin{align*}
	\mb P_a^{(k)} \mb f:=\a^k \mb g_a^{(k)}, \text{ and } \mb Q_a^{(j)}\mb f:=\b^j \mb q_a^{(j)}.
\end{align*}
Clearly, the projections satisfy the following identities,
\begin{align*}
	\mb P_a=\sum_{k=0}^9 \mb P^{(k)}_a, \quad \mb Q_a=\sum_{j=1}^9 \mb Q_a^{(j)},
\end{align*}
and
\begin{align*}
	\mb P_a^{(i)} \mb P_a^{(j)}=\d^{ij}\mb P_a^{(i)}, \quad \mb Q_a^{(k)} \mb Q_a^{(l)}=\d^{kl}\mb Q_a^{(k)}.
\end{align*}
We also define
\begin{align*}
	\mb T_a:=\mb I-\mb H_a-\mb P_a-\mb Q_a.
\end{align*}
By Lemma~\ref{Rieszprojections2}, we have that $\mb T_a$ is Lipschitz continuous with respect to $a$, and the projections $\mb T_a$, $\mb H_a$, $\mb P_a^{(k)}$, and $\mb Q_a^{(j)}$ are mutually transversal. Moreover, the Lipschitz continuity of $\mb Q_a$ and $\mb P_a$ with respect to $a$ implies that
\[
\|\mb Q_a^{(j)}-\mb Q_b^{(j)} \|\lesssim |a-b|, \quad j=1,\dots,9,
\]
and
\[
\| \mb P_a^{(k)}-\mb P_b^{(k)} \| \lesssim |a-b|, \quad k=0,\dots,9,
\]
for all $a,b\in\overline{\mathbb{B}^9_\ds}$.
In the following proposition we describe the interaction of the semigroup $(\mb S_a(\tau))_{\tau\geq 0}$ with these projections.  

\begin{proposition}\label{commute} 
The projection operators $\mb H_a$, $\mb P_a^{(k)}$, and $\mb Q_a^{(j)}$ commute with the semigroup $\mb S_a(\t)$, i.e.,
\begin{align}\label{Eq.comm_semigroup}
	[\mb S_a(\t),\mb H_a]=[\mb S_a(\t),\mb P_a^{(k)}]=[\mb S_a(\t),\mb Q_a^{(j)}]=0,
\end{align}
for $j = 1,  \dots, 9$, $k = 0, \dots 9$, and $\tau \geq 0$. Furthermore,
	\begin{align}\label{Eq.exp-growth}
		\mb S_a(\t)\mb H_a=e^{3\t}\mb H_a, \quad \mb S_a(\t)\mb P_a^{(k)}=e^\t \mb P_a^{(k)}, \quad \mb S_a(\t)\mb Q_a^{(j)}=\mb Q_a^{(j)},
	\end{align}
and there exists $\o >0$ such that
	\begin{equation}\label{Eq.exp-bound1}
		\nr{\mb S_a(\t)\mb T_a\mb u}\lesssim e^{-\o\t}\nr{\mb T_a\mb u}
	\end{equation}
    for all $\mb u\in\HH$, $a \in \overline{\mathbb{B}^9_\ds}$ and $\tau \geq 0$.
	Moreover, we have that
	\begin{equation}\label{Eq.exp-bound2}
		\nr{\mb S_a(\t)\mb T_a-\mb S_b(\t)\mb T_b}\lesssim e^{-\o\t}|a-b|,
	\end{equation}
	for all $a,b\in\overline{\mathbb{B}^9_\ds}$ and $\t\geq 0$.
\end{proposition}

\begin{proof}
	 Eq.~\eqref{Eq.comm_semigroup} follows from the properties of the Riesz projections $\mb H_a$, $\mb P_a$ and $\mb Q_a$. In particular, they commute with $\mb S_a(\tau)$, and this yields, for example,  that
	\begin{align*}
		\mb P^{(k)}_a\mb S_a(\t) \mb u=\mb P_a \mb P^{(k)}_a \mb S_a(\t)\mb u= \mb P^{(k)}_a \mb S_a\mb P_a(\t)\mb u=e^\t \mb P^{(k)}_a\mb P_a\mb u=\mb S_a(\t)\mb P^{(k)}_a\mb u.
	\end{align*}
	Eq.~\eqref{Eq.exp-growth} follows from the correspondence between point spectra of a semigroup and its generator.
	Eq.~\eqref{Eq.exp-bound1} follows from Gearhart-Pr\"uss Theorem. More precisely, we have that $\ran \mb T_a$ reduces both $\mb L_a$ and $\mb S_a(\tau)$, and furthermore
	$\mb R_{\mb L_a|_{\ran \mb T_a}}(\l)$ exists in $\{ z\in\C: \Re z \geq -\frac{\o_0}{2} \}$ and is uniformly bounded there, i.e., according to Proposition~\ref{Prop:Structure_Spectrum} there exists $c>0$ such that
	\begin{align*}
		\nr{\mb R_{\mb L_a|_{\ran \mb T_a}}(\l)}\leq c
	\end{align*}
	for all $\Re \l\geq -\frac{\o_0}{2}$ and all $a\in \overline{\B_\ds}$. Hence, by Gearhart-Pr\"uss theorem (see \cite{Engel}, page 302, Theorem 1.11), for every $\eps>0$ we have that
	\begin{align}\label{eq.GP}
		\nr{\mb S_a(\t)|_{\ran \mb T_a}}\lesssim_\ve e^{-\left(\frac{\o_0}{2}-\eps \right)\t}
	\end{align}
	for all $a \in \overline{\B_\ds}$ and $\tau \geq 0$. From here Eq.~\eqref{Eq.exp-bound1} holds for any $\o < \frac{\o_0}{2}$. 
	We remark in passing that Eq.~\eqref{Eq.exp-bound1} also follows from purely abstract considerations, see \cite{Glo21}, Theorem B.1.	Finally, to obtain the estimate \eqref{Eq.exp-bound2} we do the following. First, for $\mb u\in \DD(\mb L_a)$ we define the function
	\begin{align*}
		\Phi_{a,b}(\t)=\frac{\mb S_a(\t)\mb T_a\mb u-\mb S_b(\t)\mb T_b\mb u}{|a-b|}.
	\end{align*}
	Note that this function satisfies the evolution equation
	\begin{align*}
		\pt_\t\Phi_{a,b}(\t)=\mb L_a\mb T_a\Phi_{a,b}(\t)+\frac{\mb L_a\mb T_a-\mb L_b\mb T_b}{|a-b|}\mb S_b(\t)\mb T_b \mb u,
	\end{align*}
	with the initial condition $$
	\Phi_{a,b}(0)=\frac{\mb T_a\mb u-\mb T_b\mb u}{|a-b|},
	$$
 and therefore by Duhamel's principle we have
	\begin{align*}
		\Phi_{a,b}(\tau) = \mb S_a(\tau)  \mb T_a \frac{ \mb T_a \mb u - \mb T_b \mb u}{|a-b|} + \int_0^{\tau} \mb S_a(\tau - \tau')\mb T_a  
		\frac{\mb L_a \mb T_a - \mb L_b \mb T_b}{|a-b|} \mb S_b(\tau') \mb T_b \mb u ~ d\tau'. 
	\end{align*}
	Now, from Proposition~\ref{group-Sa} and Lemma \eqref{Rieszprojections2} we get that
	\begin{align*}
		\nr{\mb L_a\mb T_a-\mb L_b\mb T_b}\lesssim |a-b|,
	\end{align*}
	and from this and Eq.~\eqref{eq.GP} we obtain
	\begin{align*}
		\nr{\Phi_{a,b}(\t)}\lesssim e^{-\left( \frac{\o_0}{2}-\eps\right)\t}(1+\t)\nr{\mb u}\lesssim e^{-\left(\frac{\o_0}{2}-2\eps \right)}\nr{\mb u}.
	\end{align*}
	By choosing $\eps>0$ such that $\o=\frac{\o_0}{2}-2\eps>0$, we conclude the proof.
\end{proof}

\section{Nonlinear theory}\label{Sec:Nonl_U}

%Many results of this section are along the lines of \cite{GlogicSchoerkhuber2021}, respectively, \cite{DonningerSchoerkhuber2016} and we will therefore sometimes omit details. \\

With the linear theory at hand, in this section we turn to studying the Cauchy problem for the nonlinear equation \eqref{rewritten}. Following the usual approach of first constructing strong solutions, we recast Eq.~\eqref{rewritten} in an integral form \`a la Duhamel
\begin{align} \label{Duhamel}
	\Phi(\t)=\mb S_{a_\infty}(\t)\Phi(0)+\int_0^\t \mb S_{a_\infty}(\t-\s)(\mb G_{a(\s)}(\Phi(\s))-\partial_\s \mb U_{a(\s)} )d\s
\end{align}
(where $(\mb S_{a_\infty}(\t))_{\tau \geq 0}$ is the semigroup generated by $\mb L_{a_\infty}$), and resort to fixed point arguments. Our aim is to construct global and decaying solutions to \eqref{Duhamel}. An obvious obstruction to that is the presence of growing modes of $\mb S_{a_\infty}(\tau)$, see \eqref{Eq.exp-growth}, and we deal with them in the following way. First, we note that the instabilities coming from $\mb Q_{a_\infty}$ and $\mb P_{a_\infty}$ are not genuine, as they are given rise to by the Lorenz and space-time translation symmetries of Eq.~\eqref{NLW:p}. 

We take care of the Lorenz instability by modulation. Namely, the presence of the unstable space $\ran \mb Q_{a_\infty}$ is related to the freedom of choice of the function $a:[0,\infty) \mapsto \R^9$ in the ansatz \eqref{Def:Psi_modulation_ansatz}, and, roughly speaking, we prove that given small enough initial data $\Phi(0)$, there is a way to choose $a$ such that it leads to a solution $\Phi$ of Eq.~\eqref{Duhamel} which eventually (in $\tau$) gets stripped off of any remnant of the unstable space $\ran\mb Q_{a_\infty}$ brought about by initial data. 

With the rest of the instabilities, which cause exponential growth, we deal differently. Namely, we introduce to the initial data suitable correction terms which serve to suppress the growth. Also, as mentioned, the unstable space $\ran \mb P_{a_\infty}$ is another apparent instability as it is an artifact of the space-time translation symmetries, and we use it to prove that the corrections corresponding to $\mb P_{a_\infty}$ can be annihilated by a proper choice of the parameters $x_0$ and $T$, which appear in the initial data $\Phi(0)$, see \eqref{Data}. 
 The remaining instability, coming from $\mb H_{a_\infty}$, is the only genuine one, and the correction corresponding to it is reflected in the modification of the initial data in the main result, see \eqref{Eq:Initial_Data}.\\

 To formalize the process described above, we first make some technical preparations. For the rest of this paper, we fix $\o>0$ from Proposition \ref{commute}. Then, we introduce the following function spaces
\begin{align*}
	\mathcal{X}:=\{\Phi\in C([0,\infty),\HH):\nr{\Phi}_{\mathcal{X}}<\infty\},& \quad \text{where} \quad \nr{\Phi}_\mathcal{X}:=\sup_{\t>0}e^{\o\t}\nr{\Phi(\t)},
	\\
	X:=\{a\in C^1([0,\infty),\R^9): a(0)=0, \nr{a}_X<\infty\},& \quad \text{where} \quad  \nr{a}_X:=\sup_{\t>0}[e^{\o\t}|\dot{a}(\t)|+|a(\t)|].
\end{align*}
For $a\in X$, we define
\[
a_\infty:=\lim_{\t\rightarrow\infty}a(\t).
\]
Furthermore, for $\d>0$ we set
\begin{align*}
	\mathcal{X}_\d:=\{\Phi\in\mathcal{X}:\nr{\Phi}_\mathcal{X}\leq \d\}, \quad X_\d:=\{a\in X: \sup_{\t>0}[e^{\o\t}|\dot{a}(\t)|]\leq\d\}.
\end{align*}
To ensure that all terms in Eq.~\eqref{Duhamel} are defined, we must impose some size restriction on the function $a$. Note that is enough to consider $a\in X_\d$ for $\delta < \ds \o$, as then $|a(\t)|\leq \d/\o<\ds$ for all $\t\geq 0$. We will also frequently make use of the inequality
\begin{equation}\label{inequality}
	|a_\infty-a(\t)|\leq \int^\infty_\t|\dot{a}(\s)|d\s\leq \frac{\d}{\o}e^{-\o\t}.
\end{equation}
Furthermore, note that for $a,b\in X_\d$ and $\t\geq 0$ we have $|a(\t)-b(\t)|\leq \nr{a-b}_X$, in particular, we have that  $|a_\infty-b_\infty|\leq \nr{a-b}_X$.\\

\subsection{Estimates of the nonlinear terms}
With an eye toward setting up a fixed point scheme for Eq.~\eqref{Duhamel}, we now establish necessary bounds for the nonlinear terms. Namely, we treat 
\begin{align*}
	\mb G_{a(\t)}(\Phi(\t))=[\mb L'_{a(\t)}-\mb L'_{a_\infty}]\Phi(\t) + \mb F(\Phi(\t)).
\end{align*}

\begin{lemma} \label{lemma.nonlinearbounds}
	Given $\d \in(0,\d^*\omega)$ we have that 
	\begin{equation}\label{Eq:G_est}
	\begin{aligned}
		\nr{\mb G_{a(\t)}(\Phi(\t))}& \lesssim \d^2e^{-2\o\t} , 
		\\
		\nr{\mb G_{a(\t)}(\Phi(\t))-\mb G_{b(\t)}(\Psi(\t))}& \lesssim \d e^{-2\o\t}\left(\nr{\Phi-\Psi}_\mathcal{X}+\nr{a-b}_X \right), 
	\end{aligned}
\end{equation}
	for all $\Phi,\Psi\in \mathcal{X}_\d$, $a,b\in X_\d$, and $\t\geq 0$, where the implicit constants in the above estimates are absolute.
\end{lemma}

\begin{proof}
	First, since $H^5(\B)$ is a Banach algebra we have that
	\[
	\nr{u_1^2 - v_1^2}_{H^4(\B)}\lesssim \nr{u_1 + v_1}_{H^5} \nr{u_1 - v_1}_{H^5},
	\]
	and hence
	\begin{align}\label{Eq:Nonl_Lipschitz}
		\nr{\mb F(\mb u)-\mb F(\mb v)}\lesssim (\nr{\mb u}+\nr{\mb v})\nr{\mb u-\mb v},
	\end{align}
	for all $\mb u,\mb v\in  \HH$. Next, we prove the second estimate in Lemma \ref{lemma.nonlinearbounds}, as the first one follows from it. From Eq.~\eqref{Eq:Nonl_Lipschitz}, Proposition~\ref{group-Sa}, and inequality \eqref{inequality} we obtain
	\begin{equation}\label{Eq:G_ests}
		\begin{aligned}
		\nr{\mb F(\Phi(\t))-\mb F(\Psi(\t))} \lesssim \d e^{-2\o\t}\nr{\Phi-\Psi}_\mathcal{X},\\
		\nr{[\mb L'_{a(\t)}-\mb L'_{a_\infty}](\Phi(\t)-\Psi(\t))}\lesssim \d e^{-2\o\t}\nr{\Phi-\Psi}_\mathcal{X},
	\end{aligned}
	\end{equation} 
	for $\Phi,\Psi\in \mathcal{X}_\d$ and $a  \in X_\d$. Furthermore, using the fact that 
	\begin{align}
		V_{a_\infty}(\xi) - V_{a(\tau)}(\xi) = 
		\int_{\tau}^{\infty} \partial_{s} V_{a(s)}(\xi) ds = 
		\int_{\tau}^{\infty} \dot a^{k}(s) \varphi_{a(s),k}(\xi) ds
	\end{align}
	with $\varphi_{a,k}(\xi) = \partial_{a^{k}} V_{a}(\xi)$, together with the smoothness of $\varphi_{a,k}$ we infer that 
	\begin{align*}
		\nr{\big([\mb L'_{a(\t)}-\mb L'_{a_\infty}]-[\mb L'_{b(\t)}-\mb L'_{b_\infty}]\big)\mb u}  & \lesssim  \| u_1 \|_{H^4(\B)}  \int_{\tau}^{\infty} \| \dot a^{k}(s)  \varphi_{a(s),k}(\xi)  - \dot b^{k}(s)  \varphi_{b(s),k}(\xi) \|_{W^{4,\infty}(\B)}  ds \\
		& \lesssim \|\mb u \|  \int_{\tau}^{\infty} | \dot a(s) - \dot b(s)|ds + \|\mb u \|  \int_{\tau}^{\infty} | \dot a(s)| | a(s) -  b(s)|ds \\
		& \lesssim \|\mb u\|   \int_{\tau}^{\infty} e^{-\omega s} \|a -b \|_X ds.
	\end{align*} 
	Hence 
	\begin{align*}
		\nr{\big([\mb L'_{a(\t)}-\mb L'_{a_\infty}]-[\mb L'_{b(\t)}-\mb L'_{b_\infty}]\big)\Psi(\t)}\lesssim \d e^{-2\o\t}\nr{a-b}_X
	\end{align*}
	for $a,b \in X_\d$ and $\Psi \in \mc X_\d$, and this, together with \eqref{Eq:G_ests} concludes the proof.
\end{proof}

%In order to construct a Lipschitz manifold of suitable initial data we argue as in Section $6$ of \cite{GlogicSchoerkhuber2021}. In the following, we summarize the main arguments.\\

\subsection{Suppressing the instabilities}
 In this section we formalize the process of taming the instabilities. In particular, by introducing correction terms to the initial data we arrive at a modified equation, to which we prove existence of global and decaying solutions.\\

   We first derive the so-called modulation equation for the parameter $a$. Recall that $\pt_\tau \mb U_{a(\tau)} = \dot{a}_j(\tau)\mb q^{(j)}_{a(\tau)} =\sum_{j=1}^9 \dot{a}^j(\tau)\mb q^{(j)}_{a(\tau)}$, see Remark \ref{Rem:EF_0}. We introduce a smooth cut-off function $\chi:[0,\infty)\rightarrow[0,1]$ satisfying $\chi(\tau)=1$ for $\tau \in [0,1]$, $\chi(\tau)=0$ for $\tau\geq 4$, and $|\chi'(\tau)|\leq 1$ for all $\tau \in (0,\infty)$. The aim is to construct a function $a:[0,\infty) \mapsto \R^9$ such that it yields a solution $\Phi$ to Eq.~\eqref{Duhamel} for which
\begin{align} \label{eq.require}
	\mb Q_{a_\infty}^{(j)}\Phi(\tau)=\chi(\tau)\mb Q_{a_\infty}^{(j)} \Phi(0)
\end{align}
for all $\tau \geq 0$. In that case, although $\mb Q_{a_\infty}^{(j)}\Phi(0) \neq 0$ in general, we have that $\mb Q_{a_\infty}^{(j)}\Phi(\tau) = 0$ eventually in $\tau$. According to Eq.~\eqref{Duhamel} and Proposition~\ref{commute}, Eq.~\eqref{eq.require} adopts the following form
\begin{align*}
	(1-\chi(\tau))\mb Q^{(j)}_{a_\infty}\mb u+\int_0^\tau\left(\mb Q^{(j)}_{a_\infty}\mb G_{a(\s)}(\Phi(\s))-\mb Q^{(j)}_{a_\infty}\dot{a}_i(\s)\mb q^i_{a(\s)}\right) d\sigma=0,
\end{align*}
where for convenience we write $\mb u$ instead of $\Phi(0)$. Using $\mb Q^{(j)}_{a_\infty}\mb q^{(i)}_{a_\infty}=\d^{ij}\mb q^{(j)}_{a_\infty}$, we get the modulation equation
\begin{align*}
	a^j(\tau)\mb q^{(j)}_{a_\infty}=-\int_0^\tau \chi'(\s)\mb Q^{(j)}_{a_\infty}\mb u \,d\s+\int_0^\tau\left(\mb Q^{(j)}_{a_\infty}\mb G_{a(\s)} (\Phi(\s))-\mb Q^{(j)}_{a_\infty}\dot{a}_i(\s)(\mb q^{(i)}_{a(\s)}-\mb q^{(i)}_{a_\infty})\right)d\s,
\end{align*}
for $j=1,\dots,9$.
By introducing the notation
\begin{align*}
	\mb A_j(a,\Phi,\mb u)(\s) := \chi'(\s)\mb Q^{(j)}_{a_\infty}\mb u+\left(\mb Q^{(j)}_{a_\infty}\mb G_{a(\s)} (\Phi(\s))-\mb Q^{(j)}_{a_\infty}\dot{a}_i(\s)(\mb q^{(i)}_{a(\s)}-\mb q^{(i)}_{a_\infty})\right),
\end{align*}
the modulation equation can be written succinctly as
\begin{align} \label{eq-modulation}
	a_j(\tau)= A_j(\cdot, \Phi,\mb u) :=\nr{\mb q^{(j)}_{a_\infty}}^{-2}\int_0^\tau \left(\mb A_j(a,\Phi,\mb u)(\s)|\mb q^{(j)}_{a_\infty}\right)d\s, \quad j=1,\dots,9.
\end{align}
In the following we prove that for small enough $\Phi$ and $\mb u$, the system \eqref{eq-modulation} admits a global (in $\tau$) solution.
% We also need the following notation 
%\[ \mathcal{B}_{\d} := \{ \mb u \in \mc H : \|\mb u \| \leq \delta \}. \]

\begin{lemma} \label{lemma-modulation}
	For all sufficiently small $\d>0$ and all sufficiently large $C>0$ the following holds. For every $\mb u\in\HH$ satisfying $\nr{\mb u}\leq \frac{\d}{C}$ and every $\Phi \in \mathcal{X}_\d$, there exists a unique $a=a(\Phi, \mb u) \in X_\d$ such that \eqref{eq-modulation} holds for $\tau \geq 0$. Moreover,
	\begin{align}\label{Eq:Est_a}
		\nr{a(\Phi,\mb u)-a(\Psi,\mb v)}_X \lesssim \nr{\Phi-\Psi}_\mathcal{X} + \nr{\mb u-\mb v}
	\end{align}
	for all $\Phi, \Psi \in \mathcal{X}_\d$ and $\mb u,\mb v\in \mathcal{B}_{\d/C}$.
\end{lemma}

\begin{proof}
%	The proof goes along the lines of the proof of Lemma 6.3 in \cite{GlogicSchoerkhuber2021}. See also Proposition~8.3 in \cite{ChaDon19} and the proof of Lemma~5.6 in \cite{DonningerSchoerkhuber2016}. 
	We use a fixed point argument. Using the bounds from Lemma~\ref{lemma.nonlinearbounds}, one can show that given $\mb u$ and $\Phi$ that satisfy the above assumptions, the following estimates hold
	\begin{gather*}
		\|\mb A_j(a,\Phi,\mb u)(\t)\|\lesssim \left(\tfrac{\d}{C}+\d^2\right)e^{-2\o\t},\\
		\nr{\mb A_j(a,\Phi,\mb u)(\t)-\mb A_j(b,\Phi,\mb u)(\t)}\lesssim \d e^{-\o\t}\nr{a-b}_X,
	\end{gather*}
	for all $a,b \in X_\d.$ From here, according to the definition in \eqref{eq-modulation} we have that for all small enough $\d>0$ and all large enough $C>0$, given $\Phi \in \mc X_\d$ and $\mb u \in \mc B_{\d/C}$ the ball $X_\delta$ is invariant under the action of the operator $A(\cdot, \Phi,\mb u)$,   which is furthermore a contraction on $X_\d$.
	Hence, the equation \eqref{eq-modulation} has a unique solution in $X_\d$. The Lipschitz continuity of the solution map follows from the following estimate
	\begin{align*}
		\|a-b\|_X &\leq \| A(a,\Phi,\mb u)-A(b,\Phi,\mb u) \|_X + \| A(b,\Phi,\mb u)-A(b,\Phi,\mb v) \|_X \\&+ \| A(b,\Phi,\mb v)-A(b,\Psi,\mb v) \|_X
		\lesssim \d\|a-b \|_X + \|\mb u- \mb v \| + \|\Phi -\Psi \|_{\mc X},
	\end{align*}
by taking small enough $\d>0$.
\end{proof}

For the remaining instabilities, we introduce the following correction terms
\begin{align*}
	\mb C_1(\Phi,a,\mb u):=\mb P_{a_\infty}\left(\mb u+\int_0^\infty e^{-\s}\left(\mb G_{a(\s)}(\Phi(\s))-\pt_\s \mb U_{a(\s)} \right)d\s \right),
	\\
	\mb C_2(\Phi,a,\mb u):=\mb H_{a_\infty} \left(\mb u+\int_0^\infty e^{-3\s}\left(\mb G_{a(\s)}(\Phi(\s))-\pt_\s \mb U_{a(\s)} \right)d\s\right),
\end{align*}
and set $\mb C:=\mb C_1+\mb C_2$. Consequently, we investigate the modified integral equation
\begin{align} \label{modified-duhamel}
	\Phi(\t)&=\mb S_{a_\infty}(\t)\big( \mb u-\mb C(\Phi,a,\mb u)\big)+\int_0^\t \mb S_{a_\infty}(\t-\s)\left(\mb G_{a(\s)}(\Phi(\s))-\pt_\s  \mb U_{a(\s)} \right)d\s
	\\
	&=:\mb K(\Phi,a,\mb u)(\t). \nonumber
\end{align}
\begin{proposition} \label{prop.correctionterms}
	For all sufficiently small $\d>0$ and all sufficiently large $C>0$ the following holds.  For every $\mb u \in \mc H$ with $\nr{\mb u}\leq \frac{\d}{C}$ there exist functions $\Phi \in\mathcal{X}_\d$ and $a\in X_\d$ such that \eqref{modified-duhamel} holds for $\t\geq 0$. Furthermore, the solution map 
	$\mb u \mapsto (\Phi(\mb u),a(\mb u))$ is Lipschitz continuous, i.e.,
	\begin{align}\label{Eq:Lip_Phi}
		\nr{\Phi(\mb u)-\Phi(\mb v)}_\mathcal{X}+\nr{a(\mb u)-a(\mb v)}_X\lesssim \nr{\mb u-\mb v}
	\end{align}
	for all $\mb u,\mb v\in \mathcal{B}_{\d/C}$.
\end{proposition}

\begin{proof}
%	The proof follows verbatim the proof of Proposition~6.4. in \cite{GlogicSchoerkhuber2021}.
 We choose $C>0$ and $\d>0$ such that Lemma~\ref{lemma-modulation} holds. Then for fixed $\mb u\in \mathcal{B}_{\d/C}$ there is a unique $a=a(\Phi,\mb u)\in X_\d$ associated to every $\Phi\in\mathcal{X}_\d$, such that the modulation equation~\eqref{eq-modulation} is satisfied. Hence we can define $\mb K_{\mb u}(\Phi):=\mb K(\Phi,a,\mb u)$. We intend to show that for small enough $\d>0$ the operator $\mb K_{\mb u}$ is a contraction on $\mathcal{X}_\d$. To show the necessary bounds, we first split $\mb K_{\mb u}(\Phi)$ according to projections $\mb P_{a_\infty}$, $\mb Q_{a_\infty}$, $\mb H_{a_\infty}$, and $\mb T_{a_\infty}$, and then estimate each part separately.
 
 First, note that the transversality of the projections implies that
	\begin{align*}
		\mb P_{a_\infty}\mb K_{\mb u}(\Phi)(\t)=-\int_\t^\infty e^{\t-\s}\mb P_{a_\infty}\left(\mb G_{a(\s)}(\Phi(\s))-\pt_\s \mb U_{a(\s)}\right)d\s,
		\\
		\mb H_{a_\infty}\mb K_{ \mb u}(\Phi)(\t)=-\int_\t^\infty e^{3(\t-\s)}\mb H_{a_\infty}\left(\mb G_{a(\s)}(\Phi(\s))-\pt_\s\mb U_{a(\s)}\right)d\s.
	\end{align*}
	Now, since  
	\[  \pt_\t \mb U_{a(\t)} =  \dot a_j(\tau) \mb q_{a_{\infty}}^{(j)} + \dot a_j(\tau) [\mb q_{a(\t)}^{(j)} - \mb q_{a_{\infty}}^{(j)}], \]
	and $\| \mb q_{a(\t)}^{(j)} - \mb q_{a_{\infty}}^{(j)} \| \lesssim \d e^{-\o \t}$,
	we have that 
	\begin{align}\label{Bound_ProjdtU}
		\nr{\mb H_{a_\infty} \pt_\t \mb U_{a(\t)}}+\nr{\mb P_{a_\infty}\pt_\t \mb U_{a(\t)}} + \| (1-\mb Q_{a_\infty})\pt_\t \mb U_{a(\t)}\| \lesssim \d^2 e^{-2\o\t},	\end{align}
	for all $a \in X_\d$. This, together with  Lemma~\ref{lemma.nonlinearbounds} and the fact that \begin{equation}\label{Eq:Modulation}
		\mb Q_{a_\infty}\mb K_{\mb u}(\Phi)(\t)=\chi(\t)\mb Q_{a_\infty}\mb u
	\end{equation} (see Eq.~\eqref{eq.require})
	yields the following bounds
	\begin{align}\label{Eq:HPQK}
		\nr{\mb H_{a_\infty}\mb K_{\mb u}(\Phi)(\t)}+\nr{\mb P_{a_\infty}\mb K_{\mb u}(\Phi)(\t)} \lesssim \d^2e^{-2\o\t} \quad \text{and} \quad
			\nr{\mb Q_{a_\infty}\mb K_{\mb u}(\Phi)(\t)}\lesssim \tfrac{\d}{C}e^{-2\o\t}
	\end{align}
	for all $\Phi \in \mc X_\d$. On the other hand, for the stable subspace we have
	\begin{align*}
		\mb T_{a_\infty}\mb K_{\mb u}(\Phi)(\t)=\mb S_{a_\infty}(\t)\mb T_{a_\infty}\mb u+\int_0^\t \mb S_{a_\infty}(\t-\s)\mb T_{a_\infty}\left(\mb G_{a(\s)}(\Phi(\s))-\pt_\s\mb U_{a(\s)}\right)d\s,
	\end{align*}
	and by Lemma~\ref{lemma.nonlinearbounds}, Proposition~\ref{commute}, and Eq.~\eqref{Bound_ProjdtU}, we get that
	\begin{align}\label{Eq:TK}
		\nr{\mb T_{a_\infty}\mb K_{\mb u}(\Phi)(\t)}\lesssim \left(\tfrac{\d}{C}+\d^2 \right)e^{-\o\t}
	\end{align}
	for all $\Phi \in \mc X_\d$. Now, from \eqref{Eq:HPQK} and \eqref{Eq:TK} we see that  $\mb K_{\mb u}$ maps $\mathcal{X}_\d$ into itself for all $\d>0$ sufficiently small and all $C>0$ sufficiently large. 
	The contraction property of $\mb K_{\mb u}$ 
	is established similarly. Namely, there is the analogue of Eq.~\eqref{Bound_ProjdtU}
	\begin{align*}
			\nr{ \mb H_{a_\infty} \pt_\t \mb U_{a(\t)} - \mb H_{b_\infty} \pt_\t \mb U_{b(\t)}} &+ \nr{ \mb P_{a_\infty}\pt_\t \mb U_{a(\t)} - \mb P_{b_\infty}\pt_\t \mb U_{b(\t)}}\\
			 &+ \| (1-\mb Q_{a_\infty})\pt_\t \mb U_{a(\t)} - (1-\mb Q_{b_\infty})\pt_\t \mb U_{b(\t)}\| \lesssim \d^2 e^{-2\o\t}
	\end{align*}
	for all $a,b \in X_\d$. Furthermore, by Lemma \ref{lemma.nonlinearbounds}, Eq.~\eqref{Eq:Modulation}, and Lemma \ref{lemma-modulation} we get the analogous estimates to \eqref{Eq:HPQK}, namely, we have that
	\begin{align*}
	\nr{ \mb H_{a_\infty} \mb K_{\mb u}(\Phi)(\t) - \mb H_{b_\infty} \mb K_{\mb u}(\Psi)(\t)} &+ \nr{ \mb P_{a_\infty}\mb K_{\mb u}(\Phi)(\t) - \mb P_{b_\infty}\mb K_{\mb u}(\Psi)(\t)}\\
	&+ \| \mb Q_{a_\infty}\mb K_{\mb u}(\Phi)(\t) - \mb Q_{b_\infty}\mb K_{\mb u}(\Psi)(\t)\| \lesssim \d e^{-2\o\t}\nr{\Phi-\Psi}_\mathcal{X}
\end{align*}
	for all $\Phi,\Psi \in \mc X_\d$, where $a=a(\Phi,\mb u)$ and $b=a(\Psi,\mb u)$.
	Also in line with \eqref{Eq:TK} we have that
	\begin{align*}
		\nr{\mb T_{a_\infty}\mb K_{\mb u}(\Phi)(\t)-\mb T_{b_\infty}\mb K_{\mb u}(\Psi)(\t)} \lesssim \d e^{-\o\t}\nr{\Phi-\Psi}_\mathcal{X}
	\end{align*}
	for all $\Phi,\Psi \in \mc X_\d$.
	By combining these estimates we get that
		\begin{align} \label{eq.contraction}
			\nr{\mb K_{\mb u}(\Phi)-\mb K_{\mb u}(\Psi)}_\mathcal{X}\lesssim \d \nr{\Phi-\Psi}_\mathcal{X}
		\end{align}
	for all $\Phi,\Psi \in \mc X_\d$, and contractivity follows by taking small enough $\d >0$.
	
	For the Lipschitz continuity, similarly to proving Eq.~\eqref{Eq:Est_a}, we use the integral equation \eqref{modified-duhamel} to show that given sufficiently small $\delta>0$,
	\begin{equation*}
		\| \Phi(\mb u) - \Phi(\mb v)  \|_{\mc X} \lesssim \| \mb u - \mb v\|
	\end{equation*}
for all $\mb u \in \mc B_{\d/C}$, and then Eq.~\eqref{Eq:Est_a} implies \eqref{Eq:Lip_Phi}.
\end{proof}

\subsection{Conditional stability in similarity variables} According to Proposition \ref{prop.correctionterms} and Eq.~\eqref{Def:Psi_modulation_ansatz} there exists a family of initial data close to $\mb U_0$ which lead to global (strong) solutions to Eq.~\eqref{Eq:Abstract_NLW_sim}, which furthermore converge to $\mb U_{a_\infty}$, for some $a_\infty$ close to $a=0$; with minimal modifications, the same argument can be carried out for $\mb U_a$ for any $a \neq 0$. In conclusion, we have conditional asymptotic orbital stability of the family $\{ \mb U_a : a \in \R^9 \}$, the condition being that the initial data belong to the set which ensures global existence and convergence. In this section we show that this set represents a Lipschitz manifold of co-dimension eleven.\\

 Let $\d>0$ and $C>0$ be as in Proposition~\ref{prop.correctionterms}, and let $\mb u\in \mc B_{\d/C}$. Also, let us denote
\[
\mb C(\mb u):=\mb C(\Phi(\mb u),a(\mb u),\mb u),
\]
where the mapping $\mb u\mapsto (\Phi(\mb u),a(\mb u))$ is defined in Proposition \ref{prop.correctionterms}. Moreover, we denote the projection corresponding to all unstable directions by
\[
\mb J_a:=\mb P_a + \mb H_a. 
\]
Note that by definition $\mb J_{a_\infty}\mb C(\mb u)=\mb C(\mb u)$, and we have the Lipschitz estimate
\[
\|\mb J_a-\mb J_b\|\lesssim |a-b|
\]
for all $a,b\in X_\d$.

\begin{proposition}\label{prop.manifold1}
	There exists $C>0$ such that for all sufficiently
	 small $\d>0$ there exists a co-dimension eleven Lipschitz manifold 
	$\mathcal{M}=\mc M_{\d,C} \subset \mathcal{H}$ with $\mb 0\in\mathcal{M}$, defined as the graph of a Lipschitz continuous function 
	$\mb M:  \ker \mb J_0 \cap \mc B_{\d/2C} \rightarrow \ran \mb  J_0$,
	\begin{align*}
		\mathcal{M}:=\left\{\mb v+\mb M(\mb v) \,\big|\, \mb v\in\ker\mb J_0,\|\mb v\|\leq \frac{\d}{2C}\right\} \subset
		\left\{\mb u\in \mc B_{\d/C}\, \big|\, \mb C( \mb u)=0\right\}.
	\end{align*} 
	Furthermore, for every $\mb u \in \mathcal{M}$ there exists $(\Phi, a)=(\Phi_{\mb u}, a_{\mb u}) \in \mathcal{X}_\d\times X_\d$ satisfying the equation
	\begin{align}\label{eq.Duhamel}
		\Phi(\t)&=\mb S_{a_\infty}(\t) \mb u +\int_0^\t \mb S_{a_\infty}(\t-\s)\left(\mb G_{a(\s)}(\Phi(\s))-\pt_\s  \mb U_{a(\s)} \right)d\s 
	\end{align}
	for all $\tau\geq 0$. Moreover, there exists $K>C$ such that
	$
		\left\{\mb u\in \mc B_{\d/K}\, \big|\, \mb C( \mb u)=0\right\} \subset \mc M_{\d,C}.
	$
\end{proposition}

\begin{proof}
	First, we show that for small enough $\d>0$, $\mb C (\mb u)=0$ if and only if $\mb J_0 \mb C(\mb u)=0$. Assume that $\mb J_0 \mb C( \mb u)=0$. Then we obtain the estimate
	\[
	\|\mb C( \mb  u)\|\leq\|\mb J_0 \mb C(\mb u)+(\mb J_{a_{\mb u, \infty}}-\mb J_0)\mb C(\mb u) \|\lesssim |a_{\mb u, \infty}| \|\mb C(\mb u)\|.
	\]
	Since $a_{\mb u,\infty}= O(\d)$, we get $\mb C(\mb u)=0$. The other direction is obvious.
	Now we construct the mapping $\mb M$. Let $\mb u \in \mathcal{H}$ and decompose it as $\mb u = \mb v +\mb w \in \ker \mb J_0 \oplus \ran \mb J_0$. 
	Fix $\mb v\in \ker \mb J_0$ and define
	\[
	\tilde{\mb C}_{\mb v}: \ran \mb J_0 \rightarrow \ran \mb J_0, \quad \tilde{\mb C}_{\mb v}(\mb w)=\mb J_0\mb C(\mb v+\mb w).
	\]
	We establish that this mapping is invertible at zero, provided that $\mb v$ is small enough, and we obtain 
	$\mb w=\tilde{\mb C}^{-1}_{\mb v}(\mb 0)$. This defines a mapping
	\[
	\mb M:\ker \mb J_0\rightarrow \ran \mb J_0, \quad \mb M(\mb v):=\tilde{\mb C}^{-1}_{\mb v}(\mb 0).
	\]
	To show this, we use a fixed point argument. Recall the definition of the correction terms 
	$\mb C=\mb C_1 +\mb C_2$, $\mb C_1=\sum_{k=0}^9 \mb C^k$ with
	\[
	\mb C_1^k(\Phi, a,\mb u)=\mb P^{(k)}_{a_\infty}\mb u + \mb P^{(k)}_{a_\infty}\mb I_1(\Phi,a),
	\]
	and
	\[
	\mb C_2(\Phi,a,\mb u)+\mb H_{a_\infty}\mb u+\mb H_{a_\infty}\mb I_2(\Phi,a),
	\]
	where
	\begin{align*}
		\mb I_1(\Phi,a):=\int_0^\infty e^{-\s}[\mb G_{a(\s)}(\Psi(\s))-\pt_\s\mb U_{a(\s)}]d\s,
		\\
		\mb I_2(\Phi,a):=\int_0^\infty e^{-3\s}[\mb G_{a(\s)}(\Psi(\s)-\pt_\s \mb U_{s(\s)})]d\s.
	\end{align*}
	We denote
	\begin{align*}
		\mb F_1(\mb u):=\sum_{k=0}^9\mb F_1^k(\mb u)=\sum_{k=0}^9\mb P^{(k)}_{a_\infty} \mb I_1(\Phi_{\mb u},a_{\mb u}),
	\end{align*}
	and
	\[
	\mb F_2(\mb u):=\mb H_{a_\infty}\mb I_2(\Phi_{\mb u},a_{\mb u}).
	\]
	By Lemma~\ref{lemma.nonlinearbounds} and Eq.~\eqref{Bound_ProjdtU} we infer that
	\begin{align}\label{eq.boundsdelta}
		\|\mb F_1^k(\mb u)\|\lesssim \d^2,\quad \|\mb F_2(\mb u)\|\lesssim\d^2.
	\end{align}

	Now for $\mb v\in \ker \mb J_0$ we get
	\begin{align*}
		\tilde{\mb C}_{\mb v}(\mb w)&=\mb J_0\mb C(\mb v+\mb w)=\mb J_0\mb J_{a_\infty}(\mb v+\mb w)+\mb J_0(\mb F_1(\mb v+\mb w)+\mb F_2(\mb v+\mb w))
		\\
		&=\mb J^2_0\mb w+\mb J_0(\mb J_{a_\infty}-\mb J_0)\mb w +\mb J_0\mb J_{a_\infty}\mb v
		+\mb J_0(\mb F_1(\mb v+\mb w)+\mb F_2(\mb v+\mb w))
		\\
		&=\mb w+\mb J_0(\mb J_{a_\infty}-\mb J_0)(\mb v+\mb w)+\mb J_0(\mb F_1(\mb v+\mb w)+\mb F_2(\mb v+\mb w)).
	\end{align*}
	Introducing the notation
	\[
	\Omega_{\mb v}(\mb w):=\mb J_0(\mb J_0-\mb J_{a_\infty})(\mb v+\mb w)-\mb J_0(\mb F_1(\mb v+\mb w)+\mb F_2(\mb v+\mb w)),
	\]
	we rewrite equation $\tilde{\mb C}_{\mb v}(\mb w)=0$ as
	\begin{align}\label{eq.omega}
		\mb w=\Omega_{\mb v}(\mb w).
	\end{align}
	Now, for $\d>0$ and $C>0$ from Proposition \ref{prop.correctionterms}, we set
	\[
	\tilde{B}_{\d/C}(\mb v):=\left\{\mb w\in \ran \mb J_0 \, \big| \, \|\mb v+ \mb w\|\leq \frac{\d}{C} \right\}.
	\]
	We show that $\Omega_{\mb v} : \tilde{B}_{\d/C}(\mb v)\rightarrow \tilde{B}_{\d/C}(\mb v)$ is a contraction mapping for sufficiently small 
	$\mb v$. Let $\mb v\in \mathcal{H}$ with $\|\mb v\|\leq \frac{\d}{2C}$, and let $\mb w \in\tilde{B}_{\d/C}(\mb v)$. Using the 
	\eqref{eq.boundsdelta}, we estimate
	\begin{align*}
		\|\Omega_{\mb v}(\mb w)\|\leq \|\mb J_0-\mb J_{a_\infty}\|\|\mb v+\mb w\|+\|\mb F_1(\mb v+ \mb w)\|+\|\mb F_1(\mb v+\mb w)\|\lesssim
		\frac{\d^2}{C}+\d^2.
	\end{align*}
	Hence, by fixing $C>0$ we have that $\|\mb v+ \Omega_{\mb v}(\mb w)\|\leq \frac{\d}{C}$ for all small enough $\d>0$. So the ball $\tilde{B}_{\d/C}(\mb v)$ is invariant under the action of $\O_{\mb v}$. To prove contractivity, first for $\mb w, \tilde{\mb w} \in \tilde{B}_{\d/C}(\mb v)$, we associtate to $\mb v+\mb w$ and $\mb v+ \tilde{\mb w}$ the functions $(\Phi,a)$ and $(\Psi,b)$ in $\mathcal{X}_\d\times X_\d$ by Proposition~\ref{prop.correctionterms}. Then we obtain
	\begin{align*}
		\|\Omega_{\mb v}(\mb w)-\Omega_{\mb v}(\tilde{\mb w})\|&\leq \|\mb J_0(\mb J_0-\mb J_{a_\infty})(\mb v+\mb w)-\mb J_0(\mb J_0-\mb J_{b_\infty})(\mb v+ \tilde{\mb w})\|
		\\
		&+\|\mb F_1(\mb v+\mb w)-\mb F_1(\mb v+\tilde{\mb w})\|+\|\mb F_2(\mb v+\mb w)-\mb F_2(\mb v+\tilde{\mb w})\|,
	\end{align*}
	and writing
	\begin{align*}
		\mb J_0&(\mb J_0-\mb J_{a_\infty})(\mb v+\mb w)-\mb J_0(\mb J_0-\mb J_{b_\infty})(\mb v+\tilde{\mb w})
		\\
		&=\mb J_0(\mb J_0-\mb J_{a_\infty})(\mb w-\tilde{\mb w})-\mb J_0(\mb J_{a_\infty}-\mb J_{b_\infty})(\mb v+\tilde{\mb w}),
	\end{align*}
	we get by Proposition~\ref{prop.correctionterms} the following estimate
	\begin{align*}
		\|\mb J_0(\mb J_0-\mb J_{a_\infty})(\mb w-\tilde{\mb w})\|&+\|\mb J_0(\mb J_{a_\infty}-\mb J_{b_\infty})(\mb v+\tilde{\mb w})\|
		\lesssim
		|a_\infty|\|\mb w-\tilde{\mb w}\|+|a_\infty-b_\infty|\|\mb v+ \mb w\|
		\\
		&\lesssim \d \|\mb w-\tilde{\mb w}\|+\frac{\d}{C}\|a-b\|_X\lesssim \d\|\mb w-\tilde{\mb w}\|.
	\end{align*}
	On the other hand, by Lemma \ref{lemma.nonlinearbounds} and Eq.~\eqref{Bound_ProjdtU} we obtain for $k=0,\dots,9$ that
	\[
	\|\mb P^{(k)}_{a_\infty}\mb I_1(\Phi,a)-\mb P^{(k)}_{b_\infty}\mb I_2(\Psi,b)\|\lesssim\d(\|\Phi-\Psi\|_\mathcal{X}+\|a-b\|_X),
	\]
	and
	\[
	\|\mb H_{a_\infty}\mb I_2(\Phi,a)-\mb H_{b_\infty}\mb I_2(\Psi,b)\|\lesssim\d(\|\Phi-\Psi\|_\mathcal{X}+\|a-b\|_X).
	\]
	Thus we get the following Lipschitz estimate:
	\[
	\|\mb F_1(\mb v+\mb w)-\mb F_1(\mb v+\tilde{\mb w})\|+\|\mb F_2(\mb v+\mb w)-\mb F_2(\mb v+\tilde{\mb w})\|\lesssim \d \|\mb w-\tilde{\mb w}\|,
	\]
	and we conclude that for all small enough $\d>0$ the operator $\Omega_{\mb v} : \tilde{B}_{\d/C}(\mb v)\rightarrow \tilde{B}_{\d/C}(\mb v)$ is contractive, with the contraction constant $\tfrac{1}{2}$. Consequently, by the contraction map principle we get that for every $\mb v\in \ker \mb J_0 \cap \mc B_{\d/2C}$ there exists a unique 
	$\mb w\in \tilde{B}_{\d/C}(\mb v)$ that solves \eqref{eq.omega}, hence $\mb C(\mb v+\mb w)=\tilde{\mb C}_{\mb v}(\mb w)=0$.
	
	Next, we establish the Lipschitz-continuity of the mapping $\mb v\mapsto \mb M(\mb v)$. Let $\mb v,\tilde{\mb v}\in \ker \mb J_0 \cap \mc B_{\d/2C}$, and $\mb w, \tilde{\mb w}\in \tilde{B}_{\d/C}$ be the corresponding solutions to \eqref{eq.omega}. We get
	\begin{align*}
		\|\mb M(\mb v)-\mb M(\tilde{\mb v})\|&=\|\mb w-\tilde{\mb w}\|\leq \|\Omega_{\mb v}(\mb w)-\Omega_{\mb v}(\tilde{\mb w})\| + 
		\|\Omega_{\mb v}(\tilde{\mb w})-\Omega_{\tilde{\mb v}}(\tilde{\mb w})\|
		\\
		&\leq \frac{1}{2}\|\mb w-\tilde{\mb w}\|+\|\Omega_{\mb v}(\tilde{\mb w})-\Omega_{\tilde{\mb v}}(\tilde{\mb w})\|.
	\end{align*}
	The second term we estimate with
	\begin{align*}
		\|\Omega_{\mb v}(\tilde{\mb w})&-\Omega_{\tilde{\mb v}}(\tilde{\mb w})\|=\| \mb J_0(\mb J_0-\mb J_{a_{\mb v+\tilde{\mb w},\infty}})(\mb v+\tilde{\mb w})-\mb J_0(\mb F_1(\mb v+\tilde{\mb w})+\mb F_2(\mb v+\tilde{\mb w}))
		\\
		&-\mb J_0(\mb J_0-\mb J_{a_{\tilde{\mb v}+\tilde{\mb w},\infty}})+\mb J_0(\mb F_1(\tilde{\mb v}+\tilde{\mb w})+\mb F_2(\tilde{\mb v}+\tilde{\mb w}))\|
		\\
		\lesssim \, & \|\mb J_0(\mb J_{a_{\tilde{\mb v}+\tilde{\mb w}},\infty}-\mb J_{a_{\mb v+\tilde{\mb w},\infty}})\tilde{\mb w}\|
		+\|\mb J_0(\mb J_{a_{\mb v+\tilde{\mb w},\infty}}\mb v-\mb J_{a_{\tilde{\mb v}+\tilde{\mb w},\infty}}\mb w) \|
		\\
		&+\|\mb J_0(\mb F_1(\tilde{\mb v}+\tilde{\mb w})+\mb F_2(\tilde{\mb v})+\tilde{\mb w}) \|
		\\
		\lesssim \,& |a_{\tilde{\mb v}+\tilde{\mb w},\infty}-a_{\mb v+\tilde{\mb w}}|\|\tilde{\mb w}\|
		+\|\tilde{\mb v}-\mb v\|+|a_{\tilde{\mb v}+\tilde{\mb w},\infty}|\|\tilde{\mb v}\|+\d\|\tilde{\mb v}-\mb v\|
		\\
		\lesssim \, & \tfrac{\d}{C}\|\tilde{\mb v}-\mb v\|+\tfrac{\d}{2C}\|\tilde{\mb v}-\mb v\|+\d\|\tilde{\mb v}-\mb v\|\lesssim \|\tilde{\mb v}-\mb v\|.
	\end{align*}
	Thereby we obtain the claimed Lipschitz estimate
	\[
	\|\mb M(\mb v)-\mb M(\tilde{\mb v})\|\leq 2\|\Omega_{\mb v}(\tilde{\mb w})-\Omega(\tilde{\mb w})\|\lesssim \|\mb v-\tilde{\mb v}\|.
	\]
	
	We note that for $\mb u = 0$, the associated $(\Phi,a)$ is trivial, i.e., $\Phi=0$ and $a=0$. Thus, we have 
	$\mb C(\mb 0)=\mb F_1(\mb 0)+\mb F_2(\mb 0)=0$. Moreover, $\mb u=\mb v+\mb w=\mb 0$ if and only if $\mb v=\mb w=\mb 0$. Since in this case 
	$\mb v$ satisfies the smallness condition, $\mb w$ solving $\mb C(\mb 0+\mb w)=\mb 0$ is unique, hence $\mb M(\mb 0)=\mb 0$.
	
	Finally, let $\mb u\in \HH$ satisfying $\mb C(\mb u)=\mb 0$. Then, since $1-\mb J_0$ is a bounded operator on $\mc H$,
	\[
	\|(1-\mb J_0)\mb u \|\lesssim \|\mb u\|.
	\]
	We obtain $\mb v_{\mb u}:=(1-\mb J_0)\mb u\in \ker \mb J_0$ and $\|\mb v_{\mb u}\|\leq \frac{\d}{2C}$ for $\|\mb u\|\leq\frac{\d}{K}$ for $K>C$ large enough. Uniqueness yields $\mb w_{\mb u}:=\mb J_0\mb u=\mb M(\mb v_{\mb u})$, hence $\mb u\in \mathcal{M}_{\d,C}$.
\end{proof}

\begin{remark}
	For each correction term, the same argument yields the existence of Lipschitz manifolds $\mathcal{M}_1,\mathcal{M}_2\subset \HH$ of co-dimension ten, respectively, one, characterized by the vanishing of $\mb C_1$ and $\mb C_2$. In particular, in a small neighborhood around zero, $\mc M$ can be characterized as a subset of the intersection $\mathcal{M}_1 \cap \mathcal{M}_2$. 
\end{remark}

\subsection{Proof of Proposition ~\ref{prop.instability-similarity} and Proposition ~\ref{Prop:smoothness}}\label{Sec:Proof_Propsitions}

\begin{proof}[Proof of Proposition~\ref{prop.instability-similarity}.]
	Let 
	$\Phi_0\in\mathcal{M}_{\d,C}$, where $\mathcal{M}_{\d,C}$ is the manifold defined in Proposition~\ref{prop.manifold1}. In particular, $\|\Phi_0\|\leq\frac{\d}{C}$ and $\mb C(\Phi_0)=0$. By 
	Proposition~\ref{prop.manifold1} there is a pair $(\Phi,a)\in \mathcal{X}_\d\times X_\d$ which solves equation~\eqref{eq.Duhamel} with initial data $\mb u=\Phi_0$. Furthermore, after substituting the variation of constants formula
		\begin{equation*}
			\mb S_{a_{\infty}}(\tau)=\mb S(\tau) + \int_0^{\tau} \mb S(\tau -  \sigma) \mb L'_{a_\infty}\mb S_{a_{\infty}}(\sigma) d \sigma
		\end{equation*} 
		into Eq.~\eqref{eq.Duhamel}, a straightforward calculation yields that 
		$\Psi(\t):= \mb U_{a(\t)} +\Phi(\t)$ satisfies 
		\begin{align}\label{Duh:S}
			\begin{split}
				\Psi(\t)&=\mb S(\tau)(\mb U_0+\Phi_0) + \int_0^{\tau} \mb S(\tau -  \sigma)\mb F(\Psi(\sigma)) d \sigma
			\end{split},
	\end{align}
	for all $\tau \geq 0$.
		Then, based on \eqref{Eq:Selfsim_Sol_Lipschitz} and \eqref{inequality} we infer that
		\[
		\|\Psi(\t)-\mb U_{a_\infty}\|\leq \|\Phi(\t)\|+\|\mb U_{a(\t)}-\mb U_{a_\infty}\|\lesssim \delta e^{-\o\t},
		\]
		for all $\tau \geq 0$, as claimed.

\end{proof}

\begin{proof}[Proof of Proposition~\ref{Prop:smoothness}]
	Let $\Phi_0 \in  \mathcal{M}\cap (C^{\infty}(\overline{\B}) \times C^{\infty}(\overline{\B})) $ and let $\Psi  \in C([0,\infty),\HH)$ be the solution of Eq.~\eqref{Duh:S} associated to $\Phi_0$ via Proposition \ref{prop.instability-similarity}. To prove smoothness of $\Psi(\tau)$ (for fixed $\tau$) we use the representation \eqref{Duh:S}.
		Recall that we defined $\mb S(\tau) := \mb S_k(\tau)$ for $k=\frac{d+1}{2}=5$ with $(\mb S_k(\tau))_{\tau \geq 0}$ denoting the free wave evolution of Proposition \ref{time-evolution}.
		Now, by using Lemma \ref{lemma.restriction} we infer from ~\eqref{Duh:S} that 
		\begin{align*}
			\|\Psi(\tau)\|_{H^6(\mathbb B^9)\times H^{5}(\mathbb B^9)} &\lesssim e^{-\frac{\tau}{2}}\|\mb U_0+\Phi_0\|_{H^6(\mathbb B^9)\times H^{5}(\mathbb B^9)}
			 +
			\int_0^\tau e^{-\frac{1}{2} (\tau - \sigma) } \|\mb F(\Psi(\sigma))\|_{H^6(\mathbb B^9)\times H^{5}(\mathbb B^9)} d\sigma	 \\
			& \lesssim e^{-\frac{\tau}{2}}\|\mb U_0+\Phi_0\|_{H^6(\mathbb B^9)\times H^{5}(\mathbb B^9)}
			+
			\int_0^\tau e^{-\frac{1}{2} (\tau - \sigma) } \| \Psi(\sigma) \|_{H^5(\mathbb B^9)\times H^{4}(\mathbb B^9)}^2 d\sigma	\lesssim 1,
		\end{align*}
		for all $\tau \geq 0$. Then inductively, for $k \geq 5$ we get that
	$ \| \Psi(\tau) \|_{H^k(\mathbb B^9)\times H^{k-1}(\mathbb B^9)} \lesssim 1 $
		for all $\tau \geq 0$. Consequently, by Sobolev embedding we have that $\Psi(\tau) \in C^{\infty}(\overline{\B}) \times 
		C^{\infty}(\overline{\B})$ for all $\tau \geq 0.$\\
		To get regularity in $\tau$ we do the following. First, by local Lipschitz continuity of $\mb F: \mc H_k \mapsto \mc H_k$ for every $k \geq 5$, and Gronwall's lemma we get from Eq.~\eqref{Duh:S} that $\Psi: [0,\mc T] \mapsto \mc H_k$ is Lipschitz continuous for every $\mc T>0$ and $k \geq 5$. Consequently, $\mb F(\Psi(\cdot)), \mb L \mb F (\Psi(\cdot)) : [0,\mc T] \mapsto \mc H_k$ are Lipschitz continuous. The latter is immediate from interpreting $\mb L$ as a map from $\mc H_{k}$ to $\mc H_{k+2}$ and using the Lipschitz continuity of $\Psi$. Therefore $\Psi \in C^1([0,\infty),\mc H_k)$, with
		\begin{align}
				\partial_\tau \Psi(\tau) 
				& =\mb L \Psi (\tau) + \mb F(\Psi(\tau)) \label{Eq:Psi_classical}\\
				&=
				\mb S(\tau)\mb L(\mb U_0+\Phi_0)+\int_0^\tau \mb S(\tau-\sigma)\mb L\mb F(\Psi(\sigma))d\sigma + \mb F(\Psi(\tau)), \label{Eq:Psi_1st_der} 
		\end{align}
	for every $\tau \geq 0$ (see e.g.~\cite{Pazy}, p.~108, Corollary 2.6).  Consequently, by regularity of $\mb F$, $\mb F(\Psi(\cdot)),\mb L^m\mb F(\Psi(\cdot)) \in C^1([0,\infty),\mc H_k)$ for all $m \geq 0$ and $k \geq 5$. Therefore, from Eq.~\eqref{Eq:Psi_1st_der} we get that $\partial_{\tau}\Psi \in C^1([0,\infty),\mc H_k)$, with
	\begin{equation*}
		\partial^2_\tau \Psi(\tau) =
		\mb S(\tau)\mb L^2(\mb U_0+\Phi_0)+\int_0^\tau \mb S(\tau-\sigma)\mb L^2\mb F(\Psi(\sigma))d\sigma + \mb L \mb F(\Psi(\tau)) + \partial_{\tau}\mb F(\Psi(\tau))
	\end{equation*}
	for all $\tau \geq 0$. Inductively, we get that $\Psi \in C^m([0,\infty),\mc H_k)$ for all $m \geq 0$ and $k \geq 5$. In particular, by Sobolev embedding,  $\partial^m_\tau \Psi(\tau) \in C^{\infty}(\overline{\B}) \times 
	C^{\infty}(\overline{\B})$. Also, by Sobolev embedding $ H^k(\mathbb{B}^9) \hookrightarrow L^\infty(\mathbb{B}^9)$ for $k \geq 5$, we get that the derivatives in $\tau$ hold pointwise. As a consequence, by (a strong version of) the Schwarz theorem (see, e.g., \cite{Rud76}, p.~235, Theorem 9.41), we get that mixed derivatives of all orders in $\tau$ and $\xi$  exist, so $\Psi \in C^\infty(\mc Z) \times C^\infty(\mc Z)$, and Eq.~\eqref{Eq:Psi_classical} holds classically.
\end{proof}

\subsection{Variation of blow-up parameters and proof of Proposition \ref{Prop:Data_Manifold}}
In this section we prove boundedness and continuity properties of the initial data operator $\Upsilon$ (see  \ref{Data}) which are necessary to establish Proposition \ref{Prop:Data_Manifold}. We assume that $x_0\in \overline{\B_{1/2}}$ and $T\in \left[ \frac{1}{2},\frac{3}{2}\right]=:I$. We also introduce the notation
\begin{align*}
	\mathcal{Y}:=H^6(\B_2)\times H^5(\B_2),
\end{align*}
and denote by $\mathcal{B_Y}$ the unit ball in $\mathcal{Y}$.

\begin{lemma} \label{lemma-init.data}
	The initial data operator $\Upsilon: \mathcal{B_Y}\times I \times  \overline{\B_{1/2}}\rightarrow \HH$ is Lipschitz continuous, i.e.,
	\begin{align*}
		\nr{\Upsilon(\mb v,T_1,x_0)-\Upsilon(\mb w,T_2,y_0)}\lesssim \nr{\mb v-\mb w}_\mathcal{Y} + |T_1-T_2| + |x_0 - y_0 |
	\end{align*}
	for all $\mb v,\mb w \in \mathcal{B_Y}$, all $T_1,T_2 \in I$, and all $x_0,y_0 \in  \overline{\B_{1/2}}$. Furthermore, for $\d>0$ sufficiently small we have that 
	\begin{align*}
		\nr{\Upsilon(\mb v,T,x_0)}\lesssim \d,
	\end{align*}
	for all $\mb v \in \mc Y$ with $\nr{\mb v}_\mathcal{Y}\leq \d$, all $T\in [1-\d,1+\d] \subset I$ and all $x_0 \in \overline{\B_\d}$.
\end{lemma}

\begin{proof}
	Let $v\in C^\infty(\overline{\B_2})$. Let $T\in \left[\frac{1}{2},\frac{3}{3}\right]$ and $x_0,y_0 \in \overline{\B_{1/2}}$. Then we get by the fundamental theorem of calculus that
	\begin{align*}
		v(T\xi+x_0)-v(T\xi+y_0)=(x_0^i-y_0^i)\int_0^1 \pt_iv(T\xi+y_0+s(x_0-y_0))ds.
	\end{align*}
	This implies that $\nr{v(T\cdot +x_0)-v(T\cdot +y_0)}_{L^2(\B)} \lesssim \nr{v}_{H^1(\B_2)}|x_0-y_0|$. The same argument yields for all $k\in\N$ that
	\begin{align} \label{eq.translation}
		\nr{v(T\cdot+x_0)-v(T\cdot +y_0)}_{H^k(\B)}\lesssim \nr{v}_{H^{k+1}(\B_2)}|x_0-y_0|.
	\end{align} 
	Similarly,  we get for all $T_1,T_2\in \left[\frac{1}{2},\frac{3}{2}\right]$ and all $x_0\in \overline{\B_{1/2}}$ that
	\begin{align} \label{eq.dilation}
		\nr{v(T_1\cdot +x_0)-v(T_2\cdot +x_0)}_{H^k(\B)} \lesssim \nr{v}_{H^{k+1}(\B_2)}|T_1-T_2|,
	\end{align}
	where $k\in\N$. The estimates \eqref{eq.translation} and \eqref{eq.dilation} can be extended to $v\in H^{k+1}(\B_2)$ by density. Now let $\mb v,\mb w \in \mathcal{Y}$, $T_1,T_2 \in \left[\frac{1}{2},\frac{3}{2} \right]$, and $x_0,y_0 \in \overline{\B_{1/2}}$. Inequalities \eqref{eq.translation} and \eqref{eq.dilation} imply
	\begin{align} \label{eq-ineq1}
		\nr{\mc R(\mb v,T_1,x_0)-\mc R(\mb w,T_2,y_0)} \lesssim \nr{v}_\mathcal{Y}\big(|T_1-T_2|+|x_0-y_0|\big)+\nr{\mb v-\mb w}_\mathcal{Y}.
	\end{align}
	Moreover, since $\mb U_0$ is smooth, we have
	\begin{align}\label{eq-ineq2}
		\nr{\mc R(\mb U_0,T_1,x_0)-\mc R(\mb U_0,T_2,y_0)} \lesssim |T_1-T_2|+|x_0-y_0|,
	\end{align} 
	for all $T_1,T_2 \in \left[\frac{1}{2},\frac{3}{2} \right]$, and $x_0,y_0 \in \overline{\B_{1/2}}$. 
	Now the inequalities \eqref{eq-ineq1} and \eqref{eq-ineq2} imply the first part of the statement. The same inequalities imply
	\begin{align*}
		\nr{\Upsilon(\mb v,T,x_0)} \lesssim \nr{\mb v}_\mathcal{Y} + |T-1| + |x_0|,
	\end{align*}
	which proves the second part of the statement. 
\end{proof}

We have the following result, which has Proposition
\ref{Prop:Data_Manifold} as a direct consequence. To shorten the notation, we write $\mb h := \mb h_0$. 

\begin{lemma} \label{prop.parameters}
	There exist $M>0$ such that for all sufficiently small $\d>0$ the following holds. For every real-valued $ \mb v\in \mc Y$ that satisfies $\nr{\mb v}_{\mc Y} \leq \frac{\d}{M^2}$, there exist $\Phi \in \mathcal{X}_\d$, $a\in X_\d$,  and parameters $\a \in \left[-\frac{\d}{M},\frac{\d}{M} \right]$, $T \in \left[1-\frac{\d}{M}, 1+\frac{\d}{M} \right]\subset\left[\frac{1}{2},\frac{3}{2} \right]$, $x_0 \in \overline{\B_{\d/M}} \subset  \overline{\B_{1/2}}$, such that
	\begin{align} \label{eq.varofpar}
		\mb C(\Phi,a,\Upsilon(\mb v+\a \mb h_0, T, x_0))=0.
	\end{align}
	Moreover, the parameters depend Lipschitz continuously on the data, i.e.,
	\begin{align*}
		|\a(\mb v)-\a(\mb w)|+|T(\mb v)-T(\mb w)|+|x_0(\mb v)-x_0(\mb w)|\lesssim \nr{\mb v-\mb w}_{\mc Y}
	\end{align*}
	for all $\mb v,\mb w\in \mc Y$ satisfying the above assumptions. In particular, $\Upsilon(\mb v+\a \mb h,T,x_0)\in \mathcal{M}_{\d,C}$.
\end{lemma}

\begin{proof}
	Fix constants $C>0$ and $K>0$ from Proposition \ref{prop.manifold1}. By Lemma \ref{lemma-init.data}, we have that for all $M>0$ large enough and all $\d>0$ small enough, the inequality
	\begin{align}\label{Est:Ups}
		\|\Upsilon(\mb v+\a \mb h,T,x_0)\| \leq \frac{\delta}{K}
	\end{align}
	holds for every $\| \mb v \|_{\mc Y} \leq \tfrac{\d}{M}$,
	$\a \in \left[-\frac{\d}{M},\frac{\d}{M} \right]$, $T \in \left[1-\frac{\d}{M}, 1+\frac{\d}{M} \right]$, and $x_0 \in \overline{\B_{\d/M}}$. Furthermore, in view of Eq.~\eqref{Est:Ups} and Proposition~\ref{prop.correctionterms} we get that given $\nr{\mb v}_\mathcal{Y}\leq \frac{\d}{M^2}$, for every $\a\in \left[-\frac{\d}{M},\frac{\d}{M} \right]$, $T\in \left[1-\frac{\d}{M},1+\frac{\d}{M} \right]$, and $x_0\in \overline{\B_{\d/M}}$, there are functions $\Phi=\Phi(\mb v+\a \mb h,T,x_0)$ and $a=a(\mb v+\a\mb h,T,x_0)$, which solve the modified integral equation 
	\begin{align}\label{ModInt_Ydata}
		\begin{split}
			\Phi(\t)= \mb S_{a_\infty}(\t)\big(\Upsilon(\mb v,T,x_0)&-\mb C(\Phi,a,\Upsilon(\mb v,T,x_0))\big) \\
			& +\int_0^\t\mb S_{a_\infty}(\t-\s)
			\left(\mb G_{a(\s)}(\Phi(\s))-\pt_\s\mb U_{a(\s)} \right) d\s
		\end{split}
	\end{align}
	for all $\t\geq 0$. For such $\Phi$ and $a$, we show that one can associate to any $\| \mb v\|_{\mc Y} \leq \tfrac{\d}{M^2}$ suitable parameters $T$, $x_0$ and $\alpha$, such that Eq.~\eqref{eq.varofpar} holds. From this, via Proposition \ref{prop.manifold1}, we conclude that $\Upsilon(\mb v+\a \mb h,T,x_0) \in \mc M_{\d,C}$. Recall that the correction terms can be written as $\mb C=\mb C_1+\mb C_2=\sum_{k=0}^9\mb C_1+\mb C_2$, where
	\begin{align*}
		\mb C_1^k(\Phi,a,\mb u)=\mb P^{(k)}_{a_\infty}\mb u+\mb P^{(k)}_{a_\infty}\mb I_1(\Phi,a),
		\\
		\mb C_2(\Phi,a,\mb u)=\mb H_{a_\infty}\mb u+\mb H_{a_\infty}\mb I_2(\Phi,a),
	\end{align*}
	and where the integrals are denoted by
	\begin{align*}
		\mb I_1(\Phi,a)=\int_0^\infty e^{-\s}\left(\mb G_{a(\s)}(\Phi(\s))-\pt_\s\mb U_{a(\s)}\right)d\s,
		\\
		\mb I_2(\Phi,a)=\int_0^\infty e^{-3\s}\left(\mb G_{a(\s)}(\Phi(\s))-\pt_\s\mb U_{a(\s)}\right)d\s,
	\end{align*}
	and we have
	\begin{align} \label{eq.estimate_on_integral}
		\nr{\mb P^{(k)}_{a_\infty}\mb I_1(\Phi,a)} \lesssim \d^2, \quad \nr{\mb H_{a_\infty}\mb I_2(\Phi,a)} \lesssim \d^2,
	\end{align}
	see Eq.~\eqref{eq.boundsdelta}. We will show that there are parameters $T$, $\a$, and $x_0$, such that for $k=0,\dots, 9$
	\begin{align} \label{eq.variation}
		\begin{split}
			\left(\mb C^k_1(\Phi,a,\Upsilon(\mb v+\a\mb h,T,x_0))|\mb g^{(k)}_{a_\infty}\right)=0,
			\\
			\left(\mb C_2(\Phi,a,\Upsilon(\mb v+\a\mb h,T,x_0))|\mb h_{a_\infty} \right)=0,
		\end{split}
	\end{align}
	which implies \eqref{eq.varofpar}. To this end we expand the initial data operator. First, by Taylor expansion we get that for $T\in \left[1-\frac{\d}{M},1+\frac{\d}{M} \right]$ and $x_0\in \overline{\B_{\d/M}}$
	\begin{align*}
		\mc R(\mb U_0,T,x_0)-\mc R(\mb U_0,1,0)=c_0(T-1)\mb g^{(0)}_0+\sum_{j=1}^9c_jx_0^j\mb g_0^{(j)}+r(T,x_0),
	\end{align*}
	where the remainder satisfies
	\begin{align*}
		\nr{r(T,x_0)-r(\tilde{T},\tilde{x}_0)} \lesssim \d \big(|T-\tilde{T}|+|x_0-\tilde{x}_0|\big).
	\end{align*}
	Hence we obtain
	\begin{align*}
		\Upsilon(\mb v+\a\mb h,T,x_0)=\mc R(\mb v+\a\mb h,T,x_0)+c_0(T-1)\mb g^{(0)}_{a_\infty}+\sum_{j=1}^9c_jx_0^j\mb g_{a_\infty}^{(j)}+r_{a_\infty}(T,x_0),
	\end{align*}
	where
	\begin{align*}
		r_{a_\infty}(T,x_0)=c_0(T-1)\left(\mb g_0^{(0)}-\mb g_{a_\infty}^{(0)} \right)+\sum_{j=1}^9c_jx^j_0\left(\mb g_0^{(j)}-\mb g^{(j)}_{a_\infty} \right)+r(T,x_0).
	\end{align*}
	It is straightforward to check that
	\begin{align} \label{eq.estimate_on_remainder}
		\nr{r_a(T,x_0)-r_b(\tilde{T},\tilde{x}_0)} \lesssim \d \big(|a-b|+|T-\tilde{T}|+|x_0-\tilde{x}_0|\big),
	\end{align}
	for all $a,b\in \B_\d$, $T,\tilde{T}\in \left[1-\frac{\d}{M},1+\frac{\d}{M} \right]$, and $x_0,\tilde{x}_0\in \overline{\B_{\d/M}}$. We now express
	\begin{align*}
		\mc R(\mb v+\a\mb h,T,x_0)=\mc R(\mb v,T,x_0)+\a \mc R(\mb h_{a_\infty},T,x_0)+\a \mc R(\mb h-\mb h_{a_\infty},T,x_0).
	\end{align*}
	The last term can be estimated by
	\begin{align*}
		\nr{ \mc R(\mb h- \mb  h_{a_\infty},T,x_0)} \lesssim  \lesssim |a_\infty|.
	\end{align*}
	By taking the Taylor expansion of $\mc R(\mb  h_{a_\infty},T,x_0)$ at $(T,x_0)=(1,0)$ we obtain
	\begin{align*}
		\mc R(\mb  v+\a \mb  h,T,x_0)=\mc R(\mb  v,T,x_0)+\a \mb  h_{a_\infty}+\a\tilde{r}_a(T,x_0),
	\end{align*}
	where the remainder satisfies
	\begin{align} \label{eq.estimate_on_remainder_a}
		\nr{\tilde{r}_a(T,x_0)-\tilde{r}_b(\tilde{T},\tilde{x}_0)}\lesssim |a-b|+|T-\tilde{T}|+|x_0-\tilde{x}_0|.
	\end{align}
	Hence we obtain the following expansion
	\begin{align*}
		\Upsilon(\mb v+\a\mb h, T,x_0)&=\mc R(\mb v,T,x_0)+\a\mb h_{a_\infty} +c_0(T-1)\mb g^{(0)}_{a_\infty}+\sum_{j=1}^9c_jx^j_0\mb g^{(j)}_{a_\infty}
		\\
		&+r_{a_\infty}(T,x_0)+\a\tilde{r}_{a_\infty}(T,x_0).
	\end{align*}
	By applying the projections to the initial data operator we get
	\begin{align*}
		\mb P^{(0)}_{a_\infty}\Upsilon(\mb v+\a\mb h,T,x_0)&=\mb P^{(0)}_{a_\infty}\mc R(\mb v,T,x_0)+c_0(T-1)\mb g^{(0)}_{a_\infty}+\mb P^{(0)}_{a_\infty}r_{a_\infty}(T,x_0)+\a\mb P^{(0)}_{a_\infty} \tilde{r}_{a_\infty}(T,x_0),
		\\
		\mb P^{(j)}_{a_\infty}\Upsilon(\mb v+\a\mb h,T,x_0)&=\mb P^{(j)}_{a_\infty}\mc R(\mb v,T,x_0)+c_jx^j_0\mb g^{(j)}_{a_\infty}+\mb P^{(j)}_{a_\infty}r_{a_\infty}(T,x_0)+\a\mb P^{(j)}_{a_\infty} \tilde{r}_{a_\infty}(T,x_0),
		\\
		\mb H_{a_\infty}\Upsilon(\mb v+\a h,T,x_0)&=\mb H_{a_\infty}\mc R(\mb v,T,x_0)+\a\mb h_{a_\infty}+\mb H_{a_\infty}r_{a_\infty}(T,x_0)+\a\mb H_{a_\infty} \tilde{r}_{a_\infty}(T,x_0).
	\end{align*}
	Hence, by introducing the notation $\b=T-1$, we define for $k=0,\dots,9$
	\begin{align*}
		\mb \Gamma^{(k)}_{\mb v}(\a,\b,x_0)&=\mb P^{(k)}_{a_\infty}\mc R(\mb v,\b+1,x_0)+\mb P^{(k)}_{a_\infty}r_{a_\infty}(\b,x_0)+\a\mb P^{(k)}_{a_\infty} \tilde{r}_{a_\infty}(\b,x_0)+\mb P^{(k)}_{a_\infty}\mb I_1(\a,\b,x_0),
		\\
		\mb\Gamma^{(10)}_{\mb v}(\a,\b,x_0)&=\mb H_{a_\infty}\mc R(\mb v,\b+1,x_0)+\mb H_{a_\infty}r_{a_\infty}(\b,x_0)+\a\mb H_{a_\infty} \tilde{r}_{a_\infty}(\b,x_0)+\mb H_{a_\infty}\mb I_2(\a,\b,x_0).
	\end{align*}
	Using this notation we can rewrite equation \eqref{eq.variation} as
	\begin{equation} \label{eq-rewrite}
		\begin{aligned}
			\b&= \Gamma^{(0)}_{\mb v}(\a,\b,x_0) := \tilde{c}_0 \left(\mb \Gamma^{(0)}_{\mb v}(\a,\b,x_0)|\mb g^{(0)}_{a_\infty}\right),
			\\
			x_0^j&=\Gamma^{(j)}_{\mb v}(\a,\b,x_0) := \tilde{c}_j \left(\mb \Gamma^{(j)}_{\mb v}(\a,\b,x_0)|\mb g^{(j)}_{a_\infty}\right),
			\\
			\a&=\Gamma^{(10)}_{\mb v}(\a,\b,x_0) := \tilde{c}_{10}\left(\mb \Gamma^{(10)}_{\mb v}(\a,\b,x_0)|\mb h_{a_\infty} \right),
		\end{aligned}
	\end{equation}
	for $j=1,\dots,9$, and some constants $\tilde{c}_0,\tilde{c}_j,\tilde{c}_{10}\in \R$. We will show that 
	$\Gamma_{\mb v}=\big( \Gamma^{(0)}_{\mb v},\dots, \Gamma^{(10)}_{\mb v}\big)$ is a contraction on $\overline{\mathbb{B}^{11}_{\d/M}}$ for sufficiently small $\d>0$ and for sufficiently large $M>0$. Thereby the first part of the statement follows by Banach's fixed point theorem.
	
	First we observe that $\Gamma_{\mb v}$ maps $\overline{\mathbb{B}^{11}_{\d/M}}$ into itself. Indeed, by the proof of Lemma~\ref{lemma-init.data} we know that $\nr{\mc R(\mb v,1+\b,x_0)}\lesssim \nr{\mb v}_\mathcal{Y}$. Now estimates \eqref{eq.estimate_on_remainder}-\eqref{eq.estimate_on_remainder_a}, and the integral estimates \eqref{eq.estimate_on_integral} imply
	\begin{align*}
		\Gamma^{(j)}_{\mb v}(\a,\b,x_0)= O\left(\tfrac{\d}{M^2}\right) + O(\d^2),
	\end{align*}
	for all $j=0,\dots,10$. Thus, there is a choice of large enough $M>0$ such that for all sufficiently small $\d>0$ the following inequality
	\begin{align*}
		\left|\Gamma_{\mb v}(\a,\b,x_0)\right| \leq \tfrac{\d}{M}
	\end{align*}
	holds for all $(\a,\b,x_0) \in \overline{\mathbb{B}^{11}_{\d/M}}$.
	Next we show that by restricting, if necessary, to even smaller $\d>0$, the operator $\Gamma_{\mb v}$ is a contraction on $\overline{\mathbb{B}^{11}_{\d/M}}$. Let $(\Phi,a)\in \mathcal{X}_\d\times X_\d$ be the functions solving Eq.~\eqref{ModInt_Ydata} corresponding to parameters $\mb v+\a \mb h$, $T=1+\b$, and $x_0$. Furthermore, let $(\Psi,b)\in \mathcal{X}_\d\times X_\d$ be the function corresponding to $\mb v+\tilde{\a} \mb h$, $\tilde{T}=1+\tilde{\b}$, and $\tilde{x}_0$. Then, we obtain
	\begin{align*}
		\nr{\Phi-\Psi}_\mathcal{X}+\nr{a-b}_X &\lesssim \| \Upsilon(\mb v+\a\mb h,T,x_0)-\Upsilon(\mb v+\tilde{\a}\mb h,\tilde{T},\tilde{x}_0) \|
		\\
		&\lesssim |\a-\tilde{\a}|+|\b-\tilde{\b}|+|x_0-\tilde{x}_0|.
	\end{align*}
	Hence, by Lemma~\ref{lemma.nonlinearbounds}, we get for $k=0,\dots,9$ that
	\begin{align*}
		\nr{\mb P^{(k)}_{a_\infty}\mb I_1(\Phi, a)-\mb P^{(k)}_{b_\infty}\mb I_1(\Psi,b)} \lesssim \d\big(\nr{\Phi-\Psi}_\mathcal{X} + \nr{a-b}_X\big) \lesssim \d \big(|\a-\tilde{\a}|+|\b-\tilde{\b}|+|x_0-\tilde{x}_0|\big),
	\end{align*}
	and
	\begin{align*}
		\nr{\mb H_{a_\infty}\mb I_2(\Phi,a)-\mb H_{b_\infty}\mb I_2(\Psi,b)} \lesssim \d \big(|\a-\tilde{\a}|+|\b-\tilde{\b}|+|x_0-\tilde{x}_0|\big).
	\end{align*}
	Furthermore, by \eqref{eq-ineq1} and the Lipschitz continuity of the Riesz projections $\mb P^{(k)}_a$ and $\mb H_a$, we obtain
	\begin{align*}
		&\nr{\mb P^{(k)}_{a_\infty} \mc R(\mb v,T,x_0)-\mb P^{(k)}_{b_\infty}\mc R(\mb v,\tilde{T},\tilde{x}_0)}+\nr{\mb H_{a_\infty}\mc R(\mb v,T,x_0)-\mb H_{b_\infty}\mc R(\mb v,\tilde{T},\tilde{x}_0)}
		\\
		&\lesssim \nr{\mb v}_\mathcal{Y}\big(\nr{a-b}_X+|T-\tilde{T}|+|x_0-\tilde{x}_0|\big) \lesssim \d \big(|\a-\tilde{\a}|+|\b-\tilde{\b}|+|x_0-\tilde{x}_0|\big).
	\end{align*} 
	Moreover, for $k=0,\dots,9$ we have
	\begin{align*}
		\nr{\mb P^{(k)}_{a_\infty}r_{a_\infty}(T,x_0)-\mb P^{(k)}_{b_\infty}r_{b_\infty}(\tilde{T},\tilde{x}_0)} &+
		\nr{\a \mb P^{(k)}_{a_\infty}\tilde{r}_{a_\infty}(T,x_0)-\tilde{\a}\mb P^{(k)}_{b_\infty}\tilde{r}_{b_\infty}(\tilde{T},\tilde{x}_0)}
		\\
		&\lesssim \d \big(|\a-\tilde{\a}|+|\b-\tilde{\b}|+|x_0-\tilde{x}_0|\big),
	\end{align*}
	and
	\begin{align*}
		\nr{\mb H_{a_\infty}r_{a_\infty}(T,x_0)-\mb H_{b_\infty}r_{b_\infty}(\tilde{T},\tilde{x}_0)} &+
		\nr{\a\mb H_{a_\infty}\tilde{r}_{a_\infty}(T,x_0)-\tilde{\a}\mb H_{b_\infty}\tilde{r}_{b_\infty}(\tilde{T},\tilde{x}_0)}
		\\
		&\lesssim \d \big(|\a-\tilde{\a}|+|\b-\tilde{\b}|+|x_0-\tilde{x}_0|\big).
	\end{align*}
	From these estimates we infer that
	\begin{align}\label{eq.contraction1}
		\nr{\Gamma^{(j)}_{\mb v}(\a,\b,x_0) -\Gamma^{(j)}_{\mb v}(\tilde{\a},\tilde{\b},\tilde{x}_0)} \lesssim \d \big(|\a-\tilde{\a}|+|\b-\tilde{\b}|+|x_0-\tilde{x}_0|\big),
	\end{align}
	for $j=0,\dots,10$. Therefore, $\Gamma_{\mb v}$ is a contraction for all small enough $\d>0$, and this concludes the proof of the first part of the statement.

	It remains to establish the Lipschitz continuity of the parameters with respect to the initial data. Let 
	$\mb v, \mb w\in\mathcal{Y}$ satisfy the smallness condition and let $(\a,\b,x_0)$ and $(\tilde{\a},\tilde{\b},\tilde{x}_0)$ be the corresponding set of parameters. The first line in \eqref{eq-rewrite} implies
	\begin{align*}
		|\b-\tilde{\b}|& = |\Gamma^{(0)}_{\mb v}(\a,\b,x_0)-\Gamma^{(0)}_{\mb w}(\tilde{\a},\tilde{\b},\tilde{x}_0)| \lesssim
		|\Gamma^{(0)}_{\mb v}(\a,\b,x_0)-\Gamma^{(0)}_{\mb w}(\a,\b,x_0)|
		\\
		&+|\Gamma^{(0)}_{\mb w}(\a,\b,x_0)-\Gamma^{(0)}_{\mb w}(\tilde{\a},\tilde{\b},\tilde{x}_0)|.
	\end{align*}
	The second term can be estimated with \eqref{eq.contraction1}. To estimate the first term, we use the Lipschitz continuity of the Riesz projections to get
	\begin{align*}
		&\nr{\mb P^{(0)}_{a_\infty(\mb v,\b,x_0)}\mc R(\mb v,1+\b,x_0)-\mb P^{(0)}_{a_\infty(\mb w,\b,x_0)}\mc R(\mb w,1+\b,x_0)}
		\\
		&\lesssim \nr{\mb v}_\mathcal{Y}\nr{a_\infty(\mb v,\b,x_0)-a_\infty(\mb w,\b,x_0)}_X +\nr{\mb v- \mb w}_\mathcal{Y} \lesssim
		\nr{\mb v- \mb w}_\mathcal{Y} 
	\end{align*}
	Similar estimates using \eqref{eq.estimate_on_remainder}-\eqref{eq.estimate_on_remainder_a} and Lemma~\ref{lemma.nonlinearbounds} yield
	\begin{align*}
		|\Gamma^{(0)}_{\mb v}(\a,\b,x_0) -\Gamma^{(0)}_{\mb w}(\a,\b,x_0)| \lesssim \nr{\mb v-\mb w}_\mathcal{Y}.
	\end{align*}
	In summary, we obtain
	\begin{align*}
		|\b-\tilde{\b}|\lesssim \d(|\a-\tilde{\a}|+|\b-\tilde{\b}|+|x_0-\tilde{x}_0|)+\nr{\mb v-\mb w}_\mathcal{Y},
	\end{align*}
	and similar estimates for the remaining components yield that
	\begin{align*}
		|\a-\tilde{\a}|+|\b-\tilde{\b}|+|x_0-\tilde{x}_0|\lesssim \d(|\a-\tilde{\a}|+|\b-\tilde{\b}|+|x_0-\tilde{x}_0|)+\nr{\mb v-\mb w}_\mathcal{Y},
	\end{align*}
	which concludes the proof.
	\end{proof}

\subsection{Proof of Theorem \ref{main}}\label{Sec:ProofMainTh}

\begin{proof}
	Let $M>0$ be from Proposition~\ref{Prop:Data_Manifold}. For $\d>0$ define $\d':=\frac{\d}{M}$, . Then consider 
	$
		(f,g) \in C^\infty (\overline{\B_2}) \times C^\infty (\overline{\B_2})
	$ satisfying
	\begin{align*}
		\nr{(f,g)}_{H^6(\B_2)\times H^5(\B_2)} \leq \frac{\d'}{M} = \frac{\d}{M^2}.
	\end{align*}
	
	By Propositions~\ref{Prop:Data_Manifold} and \ref{prop.instability-similarity} we have that for all $\d>0$ sufficiently small there exist $a \in \overline{\mathbb
	B^9_{M\d'/\o}}$, $T \in [1-\d',1+\d']$, $x_0 \in \overline{\B_{\d'}}$ and  $\a \in [-\d',\d']$, depending Lipschitz continuously on $(f,g)$ with respect to the norm on $\mathcal{Y}$, as well a function $\Psi\in C([0,\infty),\HH)$ that solves 
	\begin{align}\label{Eq:Evol_Psi_proof}
		\begin{split}
			\Psi(\t)&=\mb S(\tau)\big[\mb U_{0} + \Upsilon( (f,g) +\a\mb h,T,x_0)\big] + \int_0^{\tau} \mb S(\tau -  \sigma)\mb F(\Psi(\sigma)) d \sigma,
		\end{split}
	\end{align}
	and obeys the estimate
	\begin{align}\label{ExpDecay}
		\| \Psi(\t)-\mb U_{a} \|\lesssim   \delta e^{-\o\t}
	\end{align}
	for all $\tau \geq 0$. By standard arguments, $\Psi$ is the unique solution to Eq.~\eqref{Eq:Evol_Psi_proof} in $C([0,\infty),\HH).$ Now, from the smoothness of $f$ and $g$, we have that the initial data $\Psi(0)=\mb U_{0} + \Upsilon( (f,g) +\a\mb h,T,x_0)$ belongs to $C^\infty (\overline{\B_2}) \times C^\infty (\overline{\B_2})$, and therefore from Proposition \ref{Prop:smoothness} we infer that $\Psi$ is smooth and solves Eq.~\eqref{Eq:Abstract_NLW_sim} classically. More precisely, by writing $\Psi(\tau)=(\psi_1(\tau,\cdot),\psi_2(\tau,\cdot))$, we have that  $\psi_j \in C^{\infty}(\mc Z)$ for $j = 1,2$, and
	\begin{align*}
		\begin{split}
			\partial_{\tau} \psi_1(\tau,\xi) & = -\xi \cdot\nabla \psi_1(\tau,\xi) -2 \psi_1(\tau,\xi) +\psi_2(\tau,\xi),  \\
			\partial_{\tau} \psi_2(\tau,\xi)&  = \Delta\psi_1(\tau,\xi) -\xi\cdot\nabla \psi_2(\tau,\xi)-3\psi_2(\tau,\xi), + \psi_1(\tau,\xi)^2,
		\end{split}
	\end{align*}
	for $(\tau,\xi) \in \mc Z$, with 
	\begin{align*}
		\begin{split}
			\psi_1(0,\cdot) = T^2 [\mb U_0]_1(T\cdot + x_0) + T^2 f(T\cdot + x_0) + \alpha T^2 h_1(T\cdot + x_0),  \\
			\psi_2(0,\cdot) = T^3 [\mb U_0]_2(T\cdot + x_0) + T^3 g(T\cdot + x_0) + \alpha T^3 h_2(T\cdot + x_0).
		\end{split}
	\end{align*}
	Furthermore, by denoting $\Phi(\tau) = \Psi(\tau) - \mb U_{a}$ where $\Phi(\tau)=(\varphi_1(\tau,\cdot),\varphi_2(\tau,\cdot))$, from \eqref{ExpDecay} we have that
	\begin{equation}\label{Eq:Phi_1_2_estimates}
		\| \varphi_1(\tau,\cdot) \|_{H^5(\mathbb{B}^9)} \lesssim \delta e^{-\omega \tau} \quad \text{and} \quad  \| \varphi_2(\tau,\cdot) \|_{H^4(\mathbb{B}^9)} \lesssim	\d e^{-\omega \tau},
	\end{equation}
	for all $\tau \geq 0$. Furthermore, by Sobolev embedding we have for the first component that
	\begin{equation}\label{Eq:L_infty_bound}
		\| \varphi_1(\tau,\cdot) \|_{L^{\infty}(\B)} \lesssim \d e^{-\omega \tau},
	\end{equation}
	for all $\tau \geq 0$.
	Now, we translate these results back to physical coordinates and let
	\begin{align*}
		u(t,x)=\frac{1}{(T-t)^2}\psi_1\left(-\log(T-t)+\log T,\frac{x-x_0}{T-t}\right).
	\end{align*}
	Based on the smoothnes properties of $\psi_1$, we conclude that $u \in C^\infty(\mc C_{T,x_0})$. Furthermore, $u$ solves 
	\begin{equation*}
		(\partial^2_t - \Delta_x) u(t,x)=u(t,x)^2
	\end{equation*}
	on $\mc C_{T,x_0}$, and satisfies
	\[u(0,\cdot) = U(|\cdot|) + f + \alpha h_1 , \quad \partial_t u(0,\cdot)   = 2 U(|\cdot|) + |\cdot| U'(|\cdot|) + g + \alpha h_2 \]
	on $\overline{\mathbb  B^{9}_{T}(x_0)}$.  Uniqueness of $u$ follows from uniqueness of $\Psi$, though it also follows by standard results concerning wave equations in physical coordinates. Furthermore,
	\[ u(t,x) =  \frac{1}{(T-t)^2}  \left [ U_a\left (\frac{x-x_0}{T-t}\right ) + \varphi(t,x) \right ],\]
	with $\varphi(t,x) := \varphi_1(-\log(T-t)+\log T,\frac{x-x_0}{T-t})$. The bound \eqref{Eq:L_infty_bound} yields
	\begin{align}\label{decay_phi}
		\|\varphi(t,\cdot)\|_{L^{\infty}(\B_{T-t}(x_0))} = 
		\|\varphi_1(-\log(T-t)+\log T, \cdot) \|_{L^{\infty}(\B)} \lesssim \d(T-t)^{\omega} 
	\end{align}
	for all $t \in [0,T)$. Furthermore, by Eq.~\eqref{Eq:Phi_1_2_estimates} 
	\[ (T-t)^{k-\frac{9}{2}}\nr{\varphi(t,\cdot) }_{\dot H^k(\B_{T-t}(x_0))} = \| \varphi_1(-\log(T -t)+\log T, \cdot) \|_{\dot{H}^k(\B)}
	\lesssim  (T-t)^{\omega}
	\]
	for $k = 0, \dots, 5$, which implies the first line in Eq.~\eqref{Conv_Sol}. The second line follows also from Eq.~\eqref{Eq:Phi_1_2_estimates} and the fact that 
	\begin{align*}
		\partial_t u(t,x)=\frac{1}{(T-t)^3}\psi_2\left(-\log(T-t)+\log T,\frac{x-x_0}{T-t}\right).
	\end{align*}
	Relabelling $\d'$ with $\d$ concludes the proof of Theorem \ref{main} for Eq.~\eqref{NLW:p}.
	Now, let $c_0$ be the constant from Eq.~\eqref{Eq:PosU^*}. Recall that the above conclusions hold for all sufficiently small $\d>0$. Therefore, from Eq.~\eqref{decay_phi} we see that we can choose small enough $\delta>0$ so as to ensure
	\[  \|\varphi(t,\cdot)\|_{L^{\infty}(\B_{T-t}(x_0))} \leq \frac{c_0}{2} \]
	for all $t \in [0,T)$.	As a consequence,  $u$ is strictly positive on $\mc C_{T,x_0}$ and therefore provides a solution to Eq.~\eqref{NLW:Abs_p} as well.
\end{proof}

\section{Proof of Theorem  \ref{main_ODE} - Stable ODE blowup}\label{Sec:Stable_ODE_blowup}

The proof of Theorem \ref{main_ODE} follows mutatis mutandis the proof of Theorem \ref{main}. However, for convenience of the reader we outline  the most important steps and stick to the notation introduced above. Starting with Eq.~\eqref{Eq:Abstract_NLW_sim}, we consider solutions of the form  $\Psi(\tau) = \bm \kappa_{a(\tau)} + \Phi(\tau)$, which yields 
\begin{align} \label{rewritten_ODE}
	\begin{split}
		\pt_\t\Phi(\t)&=[\mb L+ \mb L_{\kappa_{a_{\infty}}}']\Phi(\t)+ \tilde{ \mb G}_{a(\tau)}(\Phi(\t))-\pt_\t \bm \kappa_{a(\t)},
	\end{split}
\end{align}
where
\[ \mb L_{\kappa_a}' \mb u(\xi) = \begin{pmatrix} 
	0 \\ 2 \kappa_a(\xi) u_1(\xi) \end{pmatrix},
\]
and 
\[
\tilde{ \mb G}_{a(\tau)}(\Phi(\t))=[\mb L'_{\kappa_{a(\t)}}-\mb L_{\kappa_{a_{\infty}}}']\Phi(\t)+\mb F(\Phi(\t)).
\]
In this equation, $\mb L: \mc D(\mb L) 
\subset \mc H \to \mc H$ denotes, as usual, the operator describing the free wave evolution. This is fully characterized for both $d=7$ and $d=9$ in Sec.~\ref{Sec:Free} (recall that $\mc H := \mc H_{k}$ for $k = \frac{d+1}{2}$). For the perturbation theory, the spectral analysis is crucial. Once this is obtained, most results are purely abstract and the proofs can be adapted from previous sections.

Since $ \mb L_{\kappa_a}'$ is compact and depends Lipschitz continuously on $a$, the results of Sec.~\ref{Sec:Lin} apply.
In particular, for small enough $a$, the spectrum of $\mb L +  \mb L_{\kappa_a}'$ in the right half plane consists of isolated eigenvalues confined to a compact region. Furthermore, an analogous result to Proposition \ref{Prop:Spectral_ODE} holds with $V$ replaced a constant. This substantially simplifies the spectral analysis and with the above prerequisites it is easy to derive the following statement. For all of the ensuing statements $d \in \{7,9\}$.

\begin{proposition}
	There are constants $\d^* >0$ and $\o>0$, such that the following holds. For any $a\in\overline{\mathbb{B}^d_\ds}$, the operator $\mb L +  \mb L_{\kappa_a}' : \mc D(
	\mb L) \subset \mc H \to \mc H$ generates a strongly continuous
	semigroup $( \mb S_{\kappa_a}(\tau))_{\tau \geq 0}$ on $\mc H$. Furthermore, 
	there exist projections $\tilde {\mb P}_a, \tilde {\mb Q}_a^{(k)} \in \mc B(\mc H)$, $k = 1, \dots, d$, of rank one that are mutually transversal and depend Lipschtiz continuously on $a$. Furthermore, they commute with $ \mb S_{\kappa_a}(\tau)$ and for all  $\mb u\in\HH$ and  $\tau \geq 0$,
	\[ \mb S_{\kappa_a}(\tau) \tilde {\mb P}_a  \mb u = e^{
		\tau} \mb u, \quad  \mb S_{\kappa_a}(\tau)  \tilde {\mb Q}_a^{(k)} 
	\mb u = \mb u, \]
	as well as 
	\begin{equation*}
		\nr{ \mb S_{\kappa_a}(\tau) [ 1 - \tilde {\mb P}_a - \tilde {\mb Q}_a  ] \mb u}\lesssim e^{-\o \tau}\nr{[ 1 - \tilde {\mb P}_a - \tilde {\mb Q}_a] \mb u}
	\end{equation*}
	with $\tilde {\mb Q}_a = \sum_{k=1}^{d}  \tilde {\mb Q}_a^{(k)}$.
	Moreover, 
	\begin{equation}\label{Eq.exp-bound2_ode}
		\nr{ \mb S_{\kappa_a}(\tau)  [ 1 - \tilde {\mb P}_a - \tilde {\mb Q}_a  ] -\mb S_{\kappa_b}(\tau) [ 1 - \tilde {\mb P}_b - \tilde {\mb Q}_b ]  }\lesssim e^{-\o \t}|a-b|,
	\end{equation}
	for all $a,b\in\overline{\mathbb{B}^9_\ds}$, and $\t\geq 0$. Also, 
	\begin{align}\label{ODESpec_Proj}
		\ran \tilde {\mb P}_a =\mathrm{span}( \tilde {\mb g}_a), 
		\quad  \ran \tilde {\mb Q}^{(k)}_a = \mathrm{span}(\tilde{\mb  q}_a^{(k)}),
	\end{align}
	where 
	\begin{align}\label{EF_boost_ODE}
		\tilde {\mb g}_a(\xi) = \begin{pmatrix} A_0(a)[A_0(a) - A_j \xi^{j} ]^{-3} \\ 3 A_0(a)^2[A_0(a) - A_j \xi^{j} ]^{-4} \end{pmatrix}, \quad \tilde {\mb q}_a^{(k)} = \partial_{a_k} \bm {\kappa}_a,
	\end{align}
	for $k = 1, \dots, d$. 
\end{proposition}

\begin{proof}
	We only sketch the main steps of the proof, since many parts are abstract operator theory and can be copied verbatim from previous sections. \\
	
	The results of Sec.~\ref{Sec:Free} together with the Bounded Perturbation Theorem immediately imply that $\mb L +  \mb L_{\kappa_a}'$ generates a strongly continuous semigroup, which we denote by $( \mb S_{\kappa_a}(\tau))_{\tau \geq 0}$. 
	Furthermore, the result of Proposition \ref{Prop:Structure_Spectrum} and \ref{Prop:Spectral_ODE} hold in particular for our case at hand and we infer that for $\mathrm{Re} \lambda > -\frac{1}{2}$ the spectrum of $\mb L +  \mb L_{\kappa_a}'$ consists of isolated eigenvalues confined to a compact region. For $a=0$, Proposition \ref{Prop:Spectral_ODE} holds mutatis mutandis with $V$ replaced by the constant potential $2 \kappa_0 = 12$. In this case, in the spectral ODE the number of regular singular points can be reduced to three, and we can therefore resolve the connection problem by using the standard theory of hypergeometric equations. This is outlined in the following, where we show that there exists $0<\mu_0\leq \frac{1}{2}$, such that
	\begin{align*}
		\sigma(\mb L + \mb L_{\kappa_0}') \subset \{\lambda \in\mathbb{C} : \Re \lambda \leq -\mu_0 \} \cup \{0, 1  \}.
	\end{align*}
	In fact, we convince ourselves that 
	\begin{align}\label{SpecODE_1}
		\{ \lambda \in \C: \Re \lambda \geq 0 \} \setminus \{0,1\} \subset \rho(\mb L + \mb L_{\kappa_0}').
	\end{align}
	
	We argue by contradiction. Let us assume that  $\lambda \in \sigma(\mb L + \mb L_{\kappa_0}') \setminus \{0,1\}$ and $\Re \lambda \geq  0$. Then, for some $\ell \in \N_0$,  Eq.~\eqref{diff-op} with potential $2 \kappa_0$ must have an analytic solution on $[0,1]$. We show that this cannot be the case. By changing variables and setting $f(\rho) = \rho^\ell v(\rho^2)$, Eq.~\eqref{diff-op} transforms into the standard hypergeometric form
	\begin{equation*}
		z(1-z)v''(z) + (c - (a+b+1)z)v'(z) + a b v(z) = 0,
	\end{equation*}
	with $a = \frac12 (\lambda + \ell  -1)$, $b=  \frac12 (\lambda + \ell + 6)$, $c = \frac{d}{2} + \ell$. Fundamental systems around the regular singular points $\rho =0$ and $\rho = 1$ are given by $\{v_0,\tilde v_0\}$ and $\{v_1,\tilde v_1\}$, where 
	\begin{align*}
		\begin{split}
			v_0(z)& = {}_2F_1(a,b;c;z), \\
			\tilde{v}_0(z)& = z^{1-c}{}_2F_1(a+1-c,b+1-c; 2-c; z), \\
			v_1(z)  & = {}_2F_1(a, b; a+b+1-c;1-z) \\
			\tilde v_1(z) & =(1-z)^{c-a-b} {}_2F_1(c-a,c-b; 1+c-a-b;1-z).
		\end{split}
	\end{align*}
	In fact,  this holds if $a+b-c  \neq 0$, and for  $a+b-c = 0$ the function  $\tilde v_1$ is behaves logarithmically.  In either case, the solutions $\tilde v_1$ and $\tilde{v}_0$ are not admissible  since they are not analytic for $\Re \lambda  \geq 0$. We therefore look for $\la$ which connect $v_0$ and $v_1$, i.e., for which $v_0$ and $v_1$ are constant multiples of each other. For the hypergeometric equation, the connection coefficients  are known explicitly and  we have 
	\begin{equation*}
		v_0(z) = \frac{\Gamma(c)\Gamma(c-a-b)}{\Gamma(c-a)\Gamma(c-b)} v_1(z) + \frac{\Gamma(c)\Gamma(a+b-c)}{\Gamma(a)\Gamma(b)} \tilde v_1(z)
	\end{equation*}
	see \cite{NIST10}. So the condition that quantifies our eigenvalues is
	\[ \frac{\Gamma(c)\Gamma(a+b-c)}{\Gamma(a)\Gamma(b)} = 0. \]
	This can only be the case if $a$ or $b$ are a poles of the gamma function, i.e., $-a \in \N_0$ or $- b \in \N_0$. In particular, this implies that $\lambda \in \R$.
	Since $\Re \lambda \geq 0$, $- b \in \N_0$ can be excluded. On the other hand,  $-a \in \N_0$ is possible only if $\lambda \in \{0,1\}$, which contradicts our assumption and proves \eqref{SpecODE_1}. For $\lambda =1$ and $\lambda = 0$, one can easily check that explicit solutions to the eigenvalue equation are given by $\tilde {\mb g}_0$ and $\tilde {\mb q}_0^{(k)}$, respectively, where
	\[ \tilde {\mb g}_0 = \begin{pmatrix} 1 \\ 3 \end{pmatrix}, \quad \tilde {\mb q}_0^{(k)} = \partial_{a_k} \bm {\kappa}_a|_{a=0}  =   6 \begin{pmatrix} 
		\xi_k \\ 3 \xi_k \end{pmatrix}, \]
	for $k = 1, \dots, d$. 
	Similar to the above reasoning one shows that these functions indeed span the eigenspaces for the corresponding eigenvalues, i.e., the geometric multiplicities of $\lambda = 1$ and $\lambda = 0$ are $1$ and $d$, respectively. The algebraic multiplicities are determined by the dimension of the ranges of the corresponding Riesz projections 
	\[ \tilde {\mb P}_0 = \frac{1}{2\pi i}\int_{\g_1}\mb R_{\mb L + 
		\mb L'_{\kappa_0}}(\l)d\l  , \quad  \tilde {\mb Q}_0 = \frac{1}{2\pi i}\int_{\g_0}\mb R_{\mb L + 
		\mb L'_{\kappa_0}}(\l)d\l, \] 
	where for $j \in \{0,1\}$, $\g_j(s)=\l_j+\frac{\o_0}{4}e^{2\pi is}$ for $s\in[0,1]$.   An ODE argument analogous to the proof of Lemma \ref{projections} shows that 
	\begin{align*}
		\ran \tilde {\mb P}_0 =\mathrm{span}( \tilde {\mb g}_0), 
		\quad  \ran \tilde {\mb Q}_0 = \mathrm{span}( \tilde{\mb  q}_0^{(1)}, 
		\dots \tilde{\mb  q}_0^{(d)}).
	\end{align*}
	The perturbative characterization of the spectrum of $\mb L + \mb L_{\kappa_a}'$ for $a \in \overline{\mathbb{B}^d_{\delta^*}}$ is purely abstract. Along the lines of the  the proof of Lemma \ref{Le:Pert_Spectrum}, one shows that
	\begin{align*}
		\sigma(\mb L + \mb L_{\kappa_a}') \subset \left\{ \l\in\C : \Re \l < - \frac{\omega_0}{2} \right\}\cup \left\{0,1\right\},
	\end{align*}
	where for $\lambda = 0$ and $\lambda =1$, the eigenfunctions are Lorentz boosted versions of $\tilde {\mb g}_0$ and $\tilde {\mb q}_0^{(k)}$. In fact, one can check by direct calculations that the functions $\tilde {\mb g}_a$ and $\tilde {\mb q}_a^{(k)}$ stated in Eq.\eqref{EF_boost_ODE} solve the corresponding eigenvalue equation. Eq.~\eqref{ODESpec_Proj} for the spectral projections
	\[ \tilde {\mb P}_a = \frac{1}{2\pi i}\int_{\g_1}\mb R_{\mb L + 
		\mb L'_{\kappa_a}}(\l)d\l  , \quad  \tilde {\mb Q}_a = \frac{1}{2\pi i}\int_{\g_0}\mb R_{\mb L + 
		\mb L'_{\kappa_a}}(\l)d\l \] 
	follows again from the same abstract arguments as provided in the proof of Lemma \ref{Rieszprojections2}. The same holds for the Lipschitz dependence of the projections on the parameter $a$. The growth bounds for the semigroup follow from the structure of the spectrum, resolvent bounds and the Gearhart-Pr\"uss theorem analogous to the proof of Theorem \ref{commute}. Finally, the proof for the Lipschitz continuity \eqref{Eq.exp-bound2_ode} can be copied verbatim. 
\end{proof}

The analysis of the integral equation
\begin{align*}
	\Phi(\t)=\mb S_{\kappa_{a_\infty}}(\tau)\mb u+\int_0^\t \mb S_{\kappa_{a_\infty}} (\t-\s)(\tilde{ \mb G}_{a(\s)}(\Phi(\s))-\partial_\s \bm \kappa_{a(\s)} )d\s,
\end{align*}
is completely analogous to Sec.~\ref{Sec:Nonl_U}. In particular, to derive the modulation equation for $a$, one uses the fact that $\partial_\t \bm \kappa_{a(\t)} = \dot a_k(\tau) \tilde {\mb q}_0^{(k)}$. By introducing the correction 
\begin{align*}
	\tilde{\mb C}(\Phi,a,\mb u)=\tilde{\mb P}_{a_\infty}\mb u+\tilde{\mb P}_{a_\infty}\int_0^\infty e^{-\s}\left(\tilde{ \mb G}_{a(\s)}(\Phi(\s))-\pt_\s \bm \kappa_{a(\s)} \right)d\s,
\end{align*}
it is straightforward to prove the following result. 

\begin{proposition}\label{Prop:ODE_blowup}
	There exists $\o>0$ such that for all sufficiently large $C>0$ and all sufficiently small $\d >0$ the following holds. For every $\nr{\mb v}_{\mathcal{Y}}\leq \frac{\d}{C}$, every $T\in \left[ 1-\frac{\d}{C},1+\frac{\d}{C}\right]$ and every $x_0\in \overline{\mathbb{B}^d_{\d/C}}$, there exist functions $\Phi \in \mathcal{X}_\d$ and $a\in X_\d$ such that the integral equation
	\begin{align*}
		\Phi(\t)= & \mb S_{\kappa_{a_\infty}} (\t)\left(\Upsilon(\mb v,T,x_0)-\tilde{\mb C}(\Phi,a,\Upsilon(\mb v,T,x_0))\right) \\
		& +\int_0^\t \mb S_{\kappa_{a_\infty}} (\t-\s)(\tilde{ \mb G}_{a(\s)}(\Phi(\s))-\partial_\s \bm \kappa_{a(\s)} )d\s,
	\end{align*}
	holds for $\t\geq 0$ and $\nr{\Phi(\t)}\lesssim\d e^{-\o \t}$ for all $\t\geq 0.$ Moreover, the solution map is Lipschitz continuous, i.e.,
	\begin{align*}
		\nr{\Phi(\mb v,T_1,x_0)-\Phi(\mb w,T_2,y_0)}_\mathcal{X}+\nr{a(\mb v,T_1,x_0)-a(\mb w,T_2,y_0)}_X
		\\
		\lesssim \nr{\mb v-\mb w}_\mathcal{Y}+|T_1-T_2|+|x_0-y_0|,
	\end{align*}
	for all $\mb v,\mb w \in \mathcal{Y}$ satisfying the smallness assumption, all $T_1,T_2 \in \left[ 1-\frac{\d}{C},1+\frac{\d}{C}\right]$, and all $x_0,y_0\in \overline{\mathbb{B}^d_{\d/C}}$.
\end{proposition}
We note that similarly to the manifold $\mc M$ one can construct a manifold $\mc N \subset \ker \tilde{\mb P}_{0} \oplus \ran \tilde{\mb P}_{0} $ of co-dimensions $(1+d)$ characterized by the vanishing of the correction term $\tilde{\mb C}$. However, in the context of stable blowup this is not of much interest, since the existence of this manifold is solely caused by the translation instability.  In particular, no correction of the physical initial data is necessary, if blowup time and point are chosen appropriately, i.e., for suitably small $(f,g)$, there are $T$, $x_0$, such that $\Upsilon(\mb v,T,x_0) \in \mc N$. This is contained in the following result, where $\mc Y = H^{\frac{d+3}{2}}(\mathbb B^d) \times H^{\frac{d+1}{2}}(\mathbb B^d)$.

\begin{lemma}\label{Prop:VanishingCorrection_ODE}
	There exists $C > 0$ such that for all sufficiently small $\delta >0$ the following holds. For every $\mb v \in \mc Y$ satisfying $\|\mb v \|_{\mc Y} \leq \frac{\delta}{C^2}$, there is a choice of parameters 
	$T\in \left[ 1-\frac{\d}{C},1+\frac{\d}{C}\right]$ and $x_0\in \overline{\mathbb{B}^d_{\d/C}}$ in Proposition \ref{Prop:ODE_blowup} such that 
	\begin{align*}
		\tilde{\mb C}(\Phi, a, \Upsilon(\mb v,T,x_0)) = 0.
	\end{align*}
	Moreover, the parameters depend Lipschitz continuously on the data, i.e., 
	\begin{align*}
		| T(\mb v) - T(\mb w)| + |x_0(\mb v) - x_0(\mb w)|  \lesssim \| \mb v - \mb w\|_{\mc Y}
	\end{align*}
	for all $\mb v, \mb w \in \mc Y$ satisfying the above smallness assumption. 
\end{lemma}

The proof is along the lines of the proof of Lemma \ref{prop.parameters} above with obvious simplifications. With these results, in combination with persistence of regularity that is completely analogous to Proposition \ref{Prop:smoothness}, Theorem \ref{main_ODE} is obtained by the same arguments as in the above Sec.~\ref{Sec:ProofMainTh}.

\appendix
\section{Proof of Lemma \ref{Le:Free_Evol_LP} }\label{App1}

\begin{proof}
	We will frequently use the following identities
	\begin{equation}\label{Eq:Intbyparts}
		2\Re[\xi^j\partial_jf(\xi)\overline{f}(\xi)]=\partial_j[\xi^j|f(\xi)|^2]-d|f(\xi)|^2
	\end{equation}
	and
	\begin{equation}\label{Eq:CommRel}
		\partial_{i_1}\partial_{i_2}\dots\partial_{i_k}[\xi^j\partial_jf]=k\partial_{i_1}\partial_{i_2}\dots\partial_{i_k}f(\xi)+\xi^j\partial_j\partial_{i_1}\partial_{i_2}\dots\partial_{i_k}f(\xi),
	\end{equation}
	which hold for all $k\in\N$ and $f \in C^{\infty}(\overline{\mathbb B^d})$. Furthermore, by the divergence theorem, we have
	\begin{align}\label{Eq:Div1}
		\begin{split}
			\int_{\mathbb B^d} \partial_i \Delta u(\xi) \overline{\partial^i v(\xi)} d\xi =    -  \int_{\mathbb B^d} \Delta u(\xi)  \overline{ \Delta  v(\xi) }d\xi  +
			\int_{\mathbb S^{d-1}}  \Delta  u(\omega) \overline{ \omega_i \partial^i v(\omega)} d\s(\o),
		\end{split}
	\end{align}
	for smooth $u, v$, and similarly 
	\begin{align}\label{Eq:Div2}
		\begin{split}
			\int_{\mathbb B^d} \partial_i \Delta u(\xi) \overline{\partial^i v(\xi)} d\xi =    -  \int_{\mathbb B^d} \partial_i \partial_j u(\xi)   \overline{ \partial^j \partial^i v(\xi) }d\xi  +
			\int_{\mathbb S^{d-1}} \omega_j \partial^{j} \partial_i u(\omega) \overline{  \partial^i v(\omega)} d\s(\o).
		\end{split}
	\end{align}
	We first prove the result for $d=9$ and start with those parts of  $(\cdot|\cdot)_{\mc  H_k}$ that correspond to the standard $\dot H^k \times \dot H^{k-1}$ inner product. 
	
	For the sake of concreteness, we consider the case $k = 5$, which  corresponds to the space we are going to use later on. Using the above identities, we infer that 
	\begin{align*}
		\begin{split}
			\Re\int_{\B}\partial_i\partial_j\partial_k\partial_l\partial_m[\tilde{ \mb L}  \mb u]_1(\xi)\overline{\partial^i\partial^j\partial^k\partial^l\partial^mu_1(\xi)}d\xi=\Re\int_{\B}\partial_i\partial_j\partial_k\partial_l\partial_mu_2(\xi)\overline{\partial^i\partial^j\partial^k\partial^l\partial^mu_1(\xi)}d\xi
			\\
			-\frac{5}{2}\int_{\B}\partial_i\partial_j\partial_k\partial_l\partial_mu_1(\xi)\overline{\partial^i\partial^j\partial^k\partial^l\partial^mu_1(\xi)}d\xi-\frac{1}{2}\int_{\mathbb{S}^8}\partial_i\partial_j\partial_k\partial_l\partial_mu_1(\o)\overline{\partial^i\partial^j\partial^k\partial^l\partial^mu_1(\o)}d\s(\o).
		\end{split}
	\end{align*}
	Similarly,
	\begin{align}\label{Eq:EstLP_2}
		\begin{split}
			\Re\int_{\B}\partial_i\partial_j\partial_k\partial_l[\tilde{\mb L} \mb u]_2(\xi)\overline{\partial^i\partial^j\partial^k\partial^lu_2(\xi)}d\xi=
			-\frac{5}{2}\int_{\B}\partial_i\partial_j\partial_k\partial_lu_2(\xi)\overline{\partial^i\partial^j\partial^k\partial^lu_2(\xi)}d\xi
			\\
			-\Re\int_{\B}\partial_i\partial_j\partial_k\partial_l\partial_mu_1(\xi)\overline{\partial^i\partial^j\partial^k\partial^l\partial^mu_2(\xi)}d\xi-\frac{1}{2}\int_{\Sp}\partial_i\partial_j\partial_k\partial_lu_2(\o)\overline{\partial^i\partial^j\partial^k\partial^lu_2(\o)}d\s(\o)
			\\
			+\Re\int_{\Sp}\omega^m\partial_i\partial_j\partial_k\partial_l\partial_mu_1(\o)\overline{\partial^i\partial^j\partial^k\partial^lu_2(\o)}d\s(\o),
		\end{split}
	\end{align}
	Hence,
	\begin{align}\label{Eq:EstLP_3}
		\begin{split}
			&\Re (\tilde{ \mb L} \mb u| \mb u)_5 \leq -\frac{5}{2}\nr{\mb u}^2_5-\frac{1}{2}\int_{\mathbb{S}^8}\partial_i\partial_j\partial_k\partial_l\partial_mu_1(\o)\overline{\partial^i\partial^j\partial^k\partial^l\partial^mu_1(\o)}d\s(\o)
			\\
			&-\frac{1}{2}\int_{\Sp}\partial_i\partial_j\partial_k\partial_lu_2(\o)\overline{\partial^i\partial^j\partial^k\partial^lu_2(\o)}d\s(\o)
			\\
			& +\Re\int_{\Sp}\omega^m\partial_i\partial_j\partial_k\partial_l\partial_mu_1(\o)\overline{\partial^i\partial^j\partial^k\partial^lu_2(\o)}d\s(\o)
			\leq -\frac{5}{2}\nr{ \mb u}^2_5,
		\end{split}
	\end{align}
	where we used that $\mathrm{Re}(a \overline{b}) \leq \frac{1}{2} (|a|^2+|b|^2)$ as well as the bound
	\begin{align*}
		| \sum_{k} \omega_k \partial_k u(\omega)|^2 \leq \sum_{k} (\omega_k)^2 \sum_{k} | \partial_k u(\omega)|^2  = \sum_{k} | \partial_k u(\omega)|^2. 
	\end{align*}
	A similar calculation yields 
	\begin{align*}
		\Re (\tilde{ \mb L} \mb u| \mb u)_4 \leq -\frac{3}{2}\nr{ \mb u}^2_4.
	\end{align*}
	In view of the logic of these estimates it is clear that we cannot use the standard  homogeneous inner products for integer regularities lower than $j =3$, since the bound shifts to the right and will be positive eventually. For this reason, we use tailor-made expressions for the remaining $ H^3(\B) \times  H^{2}(\B)$ part. In the following, we prove that 
	\begin{align}\label{Eq:LP_Mainbound}
		\sum_{j=1}^3  \Re(\tilde{ \mb L} \mb u|\mb u)_j \leq - \frac{1}{2} \sum_{j=1}^3  \|\mb u \|^2_{j},
	\end{align}
	which in combination with the above bounds implies the first claim in Lemma~\ref{Le:Free_Evol_LP} for $d=9$ and $k = 5$ (and in fact for any $3 \leq k \leq 5$.) For higher regularities, we add again the corresponding standard homogeneous parts. Analogous to the above calculations, one shows that  
	\begin{align*}
		&\Re\int_{\B}\pt_{i_1 \dots i_k}[\tilde{\mb L}\mb u]_1(\xi)\overline{\pt^{i_1 \dots i_k} u_1(\xi)}d\xi=
		\left(\frac{5}{2}-k\right)\int_{\B}\pt_{i_1 \dots i_k}u_1(\xi) \overline{\pt^{i_1 \dots i_k}u_1(\xi)}
		d\xi
		\\
		&-\frac{1}{2}\int_{\mathbb{S}^8}\pt_{i_1 \dots i_k}u_1(\o) \overline{\pt^{i_1 \dots i_k}u_1(\o)} d\s(\o)
		+\Re\int_{\B}\pt_{i_1 \dots i_k}u_1(\xi)\overline{\pt^{i_1 \dots i_k}u_2(\xi)}d\xi
	\end{align*}
	and
	\begin{align*}
		&\Re\int_{\B}\pt_{i_1 \dots i_{k-1}}[\tilde{\mb L}\mb u]_2(\xi)\overline{\pt^{i_1 \dots i_{k-1}}u_2(\xi)}d\xi=
		\\
		&\left(\frac{5}{2}-k\right)\int_{\B} \pt_{i_1 \dots i_{k-1}}u_2(\xi) \overline{\pt^{i_1 \dots i_{k-1}}u_2(\xi)   } d\xi
		-\Re\int_{\B}\pt_{i_1 \dots i_k}u_1(\xi)\overline{\pt^{i_1 \dots i_k}u_2(\xi)}d\xi
		\\
		&\Re\int_{\mathbb{S}^8}\o^{i_k} \pt_{i_1 \dots i_k}u_1(\o)\overline{\pt^{i_1 \dots i_{k-1}}u_2(\o)}d\s(\o)
		-\frac{1}{2}\int_{\mathbb{S}^8} \pt_{i_1 \dots i_{k-1}}u_2(\o) \overline{\pt^{i_1 \dots i_{k-1}}u_2(\o)}
		d\s(\o).
	\end{align*}
	As in Eq.~\eqref{Eq:EstLP_3} we thus obtain for $j \geq 6$ the bound
	\begin{align*}
		\Re(\tilde{ \mb L} \mb u|\mb u)_j  \leq \left(\frac{5}{2}-j \right) \| \mb u \|_j^2.
	\end{align*}

	It is left to prove Eq.~\eqref{Eq:LP_Mainbound}. We first consider  $\Re(\tilde{ \mb L} \mb u| \mb u)_3$. Using Eq.~\eqref{Eq:CommRel}, Eq.~\eqref{Eq:Intbyparts} and the divergence theorem we calculate
	\begin{align*}
		\begin{split}
			&\Re\int_{\B}\pt_i\pt_j\pt_k[\tilde{\mb L} \mb u]_1(\xi)\overline{\pt^i\pt^j\pt^ku_1(\xi)}d\xi=-\frac{1}{2} \int_{\B}\pt_i\pt_j\pt_k u_1(\xi)\overline{\pt^i\pt^j\pt^ku_1(\xi)}d\xi
			\\
			&-\frac{1}{2} \int_{\Sp}\pt_i\pt_j\pt_ku_1(\o)\overline{\pt^i\pt^j\pt^ku_1(\o)}d\s(\o)
			+\Re\int_{\B}\pt_i\pt_j\pt_ku_2(\xi)\overline{\pt^i\pt^j\pt^ku_1(\xi)}d\xi,
		\end{split}
	\end{align*}
	An application of Eq.~\eqref{Eq:Div2} shows that 
	\begin{align*}
		\begin{split}
			&\Re\int_{\B}\pt_i\pt_j[\tilde{\mb L} \mb u]_2(\xi)\overline{\pt^i\pt^ju_2(\xi)}d\xi=
			\Re\int_{\B}\pt_i\pt_j\pt^k\pt_ku_1(\xi)\overline{\pt^i\pt^ju_2(\xi)}d\xi
			\\
			&-\frac{1}{2}\int_{\B}\pt_i\pt_ju_2(\xi)\overline{\pt^i\pt^ju_2(\xi)}d\xi-\frac{1}{2}\int_{\Sp}\pt_i \pt_ju_2(\o)\overline{\pt^i\pt^ju_2(\o)}d\s(\o)
			\\
			&=\Re\int_{\Sp}\o^k \pt_k \pt_i\pt_j u_1(\o)\overline{\pt^i\pt^j u_2(\o)}d\s(\o)
			-\Re\int_{\B}\pt_i\pt_j\pt_ku_1(\xi)\overline{\pt^i\pt^j\pt^ku_2(\xi)}d\xi
			\\
			&-\frac{1}{2}\int_{\B}\pt_i\pt_ju_2(\xi)\overline{\pt^i\pt^ju_2(\xi)}d\xi-\frac{1}{2} \int_{\Sp}\pt_i \pt_ju_2(\o)\overline{\pt^i\pt^ju_2(\o)}d\s(\o),
		\end{split}
	\end{align*}
	and finally
	\begin{align*}
		\begin{split}
			\int_{\Sp}\pt_i\pt_j[\tilde{\mb L} \mb u]_1(\o)\overline{\pt^i\pt^ju_1(\o)}d\s(\o)=
			-\Re\int_{\Sp}\o^k  \pt_k\pt_i\pt_j u_1(\o)\overline{\pt^i\pt^ju_1(\o)}d\s(\o)
			\\
			-4 \int_{\Sp}\pt_i\pt_ju_1(\o)\overline{\pt^i\pt^ju_1(\o)}d\s(\o)
			+\Re\int_{\Sp}\pt_i\pt_ju_2(\o)\overline{\pt^i\pt^ju_1(\o)}d\s(\o).
		\end{split}
	\end{align*}
	In summary, we infer that 
	\begin{align*}
		\begin{split}
			\Re(\tilde{\mb L} \mb u|\mb u)_3=-\frac{1}{2}\nr{\mb u}^2_3 -12 \int_{\Sp}\pt_i\pt_ju_1(\o)\overline{\pt^i\pt^j u_1(\o)}d\s(\o)+ 4 \int_{\Sp}A(\o)d\s(\o),
		\end{split}
	\end{align*}
	where
	\begin{align*}
		\begin{split}
			A(\o)=-\frac{1}{2}\pt_i\pt_j\pt_ku_1(\o)\overline{\pt^i\pt^j\pt^ku_1(\o)}
			-\frac{1}{2}\pt_i\pt_ju_2(\o)\overline{\pt^i\pt^ju_2(\o)}
			\\
			-\frac{1}{2}\pt_i\pt_ju_1(\o)\overline{\pt^i\pt^ju_1(\o)}
			+\Re\left(\o^k\pt_i\pt_j\pt_ku_1(\o)\overline{\pt^i\pt^ju_2(\o)}\right)
			\\
			-\Re\left( \o^k\pt_i\pt_j\pt_ku_1(\o)\overline{\pt^i\pt^ju_1(\o)}\right)
			+\Re\left( \pt_i\pt_ju_2(\o)\overline{\pt^i\pt^ju_1(\o)}\right).
		\end{split}
	\end{align*}
	By using the inequality
	\begin{align*}
		\Re(a\overline{b})+\Re(a\overline{c})-\Re(b\overline{c})\leq \frac{1}{2}\left(|a|^2+|b|^2+|c|^2\right), \quad a,b,c\in\C,
	\end{align*}
	we get that $A(\o)\leq 0$. Analogously, to estimate $\Re(\tilde{ \mb L} \mb u|\mb u)_2$, we get
	\begin{align*}
		\begin{split}
			\Re\int_{\B}\pt_i\pt^j\pt_j[\tilde{ \mb L} \mb u]_1(\xi)\overline{\pt^i\pt_l\pt^lu_1(\xi)}d\xi=
			-\frac{1}{2}\Re\int_{\B}\pt_i\pt^j\pt_ju_1(\xi)\overline{\pt^i\pt_l\pt^lu_1(\xi)}d\xi
			\\
			-\frac{1}{2}\int_{\Sp}\pt_i\pt^j\pt_ju_1(\o)\overline{\pt^i\pt_l\pt^lu_1(\o)}d\s(\o)+
			\Re\int_{\B}\pt_i\pt^j\pt_ju_1(\xi)\overline{\pt^i\pt_l\pt^lu_2(\xi)}d\xi,
		\end{split}
	\end{align*}
	and 
	\begin{align*}
		\begin{split}
			\Re \int_{\Sp} \pt_i[\tilde{ \mb L} \mb u]_2(\o)\overline{\pt^iu_2(\o)}d\s(\o)=\Re \int_{\Sp} \pt_i\pt^j\pt_j u_1(\o)\overline{\pt^iu_2(\o)}d\s(\o)
			\\
			-\Re \int_{\Sp}\o^j\pt_i\pt_ju_2(\o)\overline{\pt^iu_2(\o)}d\s(\o)
			-4\Re\int_{\Sp}\pt_iu_2(\o)\overline{\pt^iu_2(\o)}d\s(\o).
		\end{split}
	\end{align*}
	For the remaining term, we do a similar calculation as in Eq.~\eqref{Eq:EstLP_2}, but we use instead Eq.~\eqref{Eq:Div1} in order to cancel the mixed term. In summary, we obtain 
	\begin{align*}
		\begin{split}
			\Re (\tilde{\mb L} \mb u| \mb u)_2=-\frac{1}{2}\nr{ \mb u}^2_2-3\int_{\Sp} \pt_iu_2(\o)\overline{\pt^iu_2(\o)}d\s(\o) +\int_{\Sp}B(\o)d\s(\o),
		\end{split}
	\end{align*}
	where
	\begin{align*}
		\begin{split}
			B(\o)=-\frac{1}{2}\pt_i\pt^j\pt_ju_1(\o)\overline{\pt^i\pt_l\pt^lu_1(\o)}
			-\frac{1}{2}\pt_i\pt_ju_2(\o)\overline{\pt^i\pt^ju_2(\o)}
			\\
			-\frac{1}{2}\pt_i u_2(\o)\overline{\pt^iu_2(\o)}
			+\Re\left( \o^k\pt_i\pt^j\pt_ju_1(\o)\overline{\pt^i\pt_ k u_2(\o)}\right)
			\\
			+\Re\left( \pt_i\pt^j\pt_ju_1(\o)\overline{\pt^iu_2(\o)}\right)
			-\Re\left( \o^j\pt_i\pt_ju_2(\o)\overline{\pt^iu_2(\o)}\right)
		\end{split}
	\end{align*}
	and we observe that  $B(\o)\leq 0$. Now, we consider $\Re(\tilde{ \mb L} \mb u|\mb u)_1$, which consists only of boundary integrals. For the first term, we get that
	\begin{align*}
		\begin{split}
			\Re \int_{\Sp} \pt_i[ \mb L \mb u]_1(\o)\overline{\pt^iu_1(\o)}d\s(\o)=-3\Re\int_{\Sp}\pt_iu_1(\o)\overline{\pt^iu_1(\o)}d\s(\o)
			\\
			-\Re\int_{\Sp}\o^j\pt_i\pt_ju_1(\o)\overline{\pt^iu_1(\o)}d\s(\o)
			+\Re\int_{\Sp}\pt_iu_2(\o)\overline{\pt^iu_1(\o)}d\s(\o).
		\end{split}
	\end{align*}
	By the Cauchy-Schwarz inequality,
	\begin{align*}
		\begin{split}
			\Re\int_{\Sp}\left( \pt_iu_2(\o)-\o^j\pt_i\pt_j u_1(\o)\right)\overline{\pt^iu_1(\o)}d\s(\o)\leq
			\frac{1}{2}\int_{\Sp}\pt_iu_1(\o)\overline{\pt^iu_1(\o)}d\s(\o)
			\\
			+\int_{\Sp}\pt_iu_2(\o)\overline{\pt^iu_2(\o)}d\s(\o)+\int_{\Sp}\pt_i\pt_ju_1(\o)\overline{\pt^i\pt^ju_1(\o)}d\s(\o),
		\end{split}
	\end{align*} 
	which implies that
	\begin{align*}
		\begin{split}
			\Re \int_{\Sp} \pt_i[ \mb L \mb u]_1(\o)\overline{\pt^iu_1(\o)}d\s(\o) \leq -\frac{5}{2}\Re\int_{\Sp}\pt_iu_1(\o)\overline{\pt^iu_1(\o)}d\s(\o)
			\\
			+\int_{\Sp}\pt_iu_2(\o)\overline{\pt^iu_2(\o)}d\s(\o)+\int_{\Sp}\pt_i\pt_ju_1(\o)\overline{\pt^i\pt^ju_1(\o)}d\s(\o).
		\end{split}
	\end{align*}
	Analogously, 
	\begin{align*}
		\begin{split}
			\Re\int_{\Sp}[\tilde{ \mb L} \mb u]_2(\o)&\overline{u_2(\o)}d\s(\o)=-3\int_{\Sp}|u_2(\o)|^2d\s(\o)
			+\Re \int_{\Sp}\pt^i\pt_iu_1(\o)\overline{u_2(\o)}d\s(\o)
			\\
			&-\Re\int_{\Sp}\o^i\pt_iu_2 (\o)\overline{u_2(\o)}d\s(\o)
			\\
			&\leq -\frac{5}{2}\int_{\Sp} |u_2(\o)|^2d\s(\o)+\int_{\Sp}|\D u_1(\o)|^2d\s(\o)
			+\int_{\Sp}\pt_iu_2(\o)\overline{\pt^iu_2(\o)}d\s(\o),
		\end{split}
	\end{align*}
	and
	\begin{align*}
		\begin{split}
			\Re\int_{\Sp} & [\tilde{\mb L} \mb u]_1(\o)\overline{u_1(\o)}d\s(\o)=-2\int_{\Sp}|u_1(\o)|^2d\s(\o)
			-\Re\int_{\Sp}\o^i\pt_iu_1(\o)\overline{u_1(\o)}d\s(\o)
			\\
			&+\Re\int_{\Sp}u_2(\o)\overline{u_1(\o)}d\s(\o)
			\\
			&\leq 
			-\frac{3}{2}\int_{\Sp}|u_1(\o)|^2d\s(\o)+\int_{\Sp}\pt_iu_1(\o)\overline{\pt^iu_1(\o)}d\s(\o)
			+\int_{\Sp} |u_2(\o)|^2d\s(\o),
		\end{split}
	\end{align*} 
	Hence,
	\begin{align*}
		\begin{split}
			\Re(\tilde{\mb L} \mb u| \mb u)_1 \leq - \frac{3}{2}  \|\mb u \|_1^2  &  +2  \int_{\Sp}\pt_iu_2(\o)\overline{\pt^iu_2(\o)}d\s(\o) \\
			& +\int_{\Sp}\pt_i\pt_ju_1(\o)\overline{\pt^i\pt^ju_1(\o)}d\s(\o) +\int_{\Sp}|\D u_1(\o)|^2d\s(\o).	
		\end{split}
	\end{align*}
	
	In conclusion, 
	\begin{align*}
		\begin{split}
			\sum_{j=1}^3 \Re(\tilde{ \mb L} \mb u|\mb u)_j  & \leq -\frac{1}{2}  \sum_{ j=1}^3  \|\mb u \|^2_{j} - 11 \int_{\Sp}\pt_i\pt_ju_1(\o)\overline{\pt^i\pt^ju_1(\o)}d\s(\o)+\int_{\Sp}|\D u_1(\o)|^2d\s(\o).
		\end{split}
	\end{align*}
	By the Cauchy-Schwarz inequality,
	\begin{align*}
		|\D u(\o)|^2\leq \left| \sum_{i=1}^9 \pt_i^2u(\o)\right|^2 \leq 9\sum_{i=1}^9|\pt_i^2u(\o)|^2\leq 9
		\sum_{i,j=1}^9|\pt_i\pt_ju(\o)|^2,
	\end{align*}
	which proves Eq.~\eqref{Eq:LP_Mainbound}. \\
	
	Analogous calculations for $d=7$ and $k \geq 3$ yield an even better bound, namely
	\begin{align}\label{Eq:BetterBound7}
		\mathrm{Re} (\tilde{\mb L} \mb u | \mb u )_{\HH_k} \leq  - \tfrac{3}{2} \| \mb u \|^2_{\mc H_k},
	\end{align}
	for all $\mb u \in \DD(\tilde{ \mb L})$ from which we obtain in particular the claimed estimate. Another way to see that \eqref{Eq:BetterBound7} holds is by Lemma $3.2$ of \cite{GlogicSchoerkhuber2021}, which is formulated in terms of above inner product for the specific case $d=7$, $k =3$. The operator considered there corresponds to $\tilde{ \mb L}$ shifted by a constant, which immediately gives the inequality \eqref{Eq:BetterBound7}. 
\end{proof}

\pagestyle{plain}
\bibliography{references_paper}
\bibliographystyle{plain}

\end{document}